\documentclass[fleqn,10pt]{article}
\usepackage[utf8]{inputenc}

\usepackage{graphicx}
\usepackage{float} 
\usepackage{subfigure} 
\usepackage{amsmath,amssymb}
\usepackage{amsthm}
\usepackage{appendix}
\usepackage[hidelinks]{hyperref}
\usepackage{url}
\usepackage{comment,xcolor}
\usepackage[utf8]{inputenc}
\usepackage{verbatim}
\usepackage{bm}
\usepackage{array}

\usepackage{graphicx}
\usepackage{float} 
\usepackage{subfigure} 
\usepackage{amsmath,amssymb}
\usepackage{appendix}
\usepackage[scr]{rsfso}

\usepackage{natbib}
\newtheorem{remark}{Remark}

\newtheorem{theorem}{\bf Theorem}

\newcommand{\mbs}[1]{\mathbf{#1}}
\newcommand{\mbb}[1]{\mathbb{#1}}

\title{
A Functionally Connected Element Method for  Solving Boundary Value Problems 
}
\author{Jielin Yang, Suchuan Dong\thanks{Author of correspondence. Emails: yang1659@purdue.edu (J.~Yang), sdong@purdue.edu (S.~Dong)} \\
Center for Computational and Applied Mathematics \\
Department of Mathematics\\ Purdue University, USA
}
\date{(\today)}

\begin{document}

\maketitle

\begin{abstract}

We present the general forms of piece-wise functions on partitioned domains satisfying an intrinsic  $C^0$ or $C^1$ continuity across the sub-domain boundaries. These general forms are constructed based on a strategy stemming from the theory of functional connections, and we refer to partitioned domains endowed with these general forms as functionally connected elements (FCE). We further present a method, incorporating functionally connected elements and  a least squares collocation approach, for solving boundary and initial value problems. This  method exhibits a spectral-like accuracy, with the free functions involved in the FCE  form represented by polynomial bases or by non-polynomial bases of quasi-random sinusoidal functions. The FCE method offers a unique advantage over traditional element-based methods for boundary value problems involving  relative boundary conditions. A number of linear and nonlinear numerical examples in one and two dimensions are presented to demonstrate the performance of the FCE method developed herein.

\end{abstract}


\section{Introduction}

% what aspects should be touched upon?
% spectral elements, finite elements,
% unconventional element-based methods
% least squares finite elements
% Discontinuous Galerkin Least Sqaures FEM
% DGs(?)
% TFC
% least squares collocation approach
% 

% what is the problem are you considering?
% what has already been done? what has not been done?
% what is this paper about?
% what are the novelties?

This work concerns the development of piece-wise functions on partitioned domains satisfying a prescribed continuity  across the sub-domain boundaries, and their applications to solving boundary and initial value problems (BVP/IVP). These piece-wise functions underpin the element-based numerical techniques such as the finite elements and spectral/hp elements~\cite{SzaboB1991,KarniadakisS2005,BabuskaS1990,BabuskaS1994}. 
% review C^1 finite/spectral elements here
$C^0$ continuous elements have seen widespread applications in  finite element  type methods owing to the simplicity in its bases construction~\cite{Hughes1987,Bathe2006,KarniadakisS2005}. For a number of application problems elements with a higher continuity (such as $C^1$) may be needed or favored, and these have attracted extensive research over the past decades (see e.g.~\cite{BognerFS1965,ArgyrisFS1968,Bell1969,Veubeke1968,BrennerS2005,LaiS2007,BucheggerJM2016,KaplST2019b,WuXL2020,HughesSTT2021,KaplST2021}, among others).

In this paper we restrict our attention  to regular domain partitions, where the sub-domain boundaries are aligned with the coordinate lines or planes, and we pursue the following question:
\begin{itemize}
\item What is the {\em general form} of  piece-wise functions satisfying  an intrinsic  $C^0$ or $C^1$ continuity across the sub-domain boundaries?
\end{itemize}
We are interested in the general form in the sense that it should encompass any piece-wise function with  $C^0$ or $C^1$ continuity over the paritioned domains.
We devise constructions of such general forms in one and two dimensions (1D/2D) herein, noting that the construction procedure can be extended to three and higher dimensions straightforwardly, albeit with the constructed form becoming significantly more involved in dimensions higher than two. 
%Our construction is for  the general forms, in the sense that any piece-wise function satisfying the $C^0$ or $C^1$ continuity across the sub-domain boundaries is included in the constructed form. 
We refer to partitioned domains equipped with such general forms of piece-wise functions with a certain continuity as Functionally Connected Elements (FCEs). 
We use ``elements" and ``sub-domains" interchangeably in this paper.

% how to construct such forms, TFC

The main strategy for the construction of these piece-wise functions stems from the theory of functional connections (TFC). TFC was originally developed by Mortari and collaborators~\cite{Mortari2017,MortariL2019,LeakeJM2022}. It provides a systematic approach for handling linear constraints, and has been widely applied in  dealing with boundary and initial conditions (see e.g.~\cite{Mortari2017,MortariL2019,Schiassietal2021,LeakeJM2022}, among others). 
%
% what are the issues with dealing with continuity conditions between elements?
% need symmetric formulation among coupled functions
For the class of problems in this study, to enforce  $C^0$ or $C^1$ continuity across the sub-domain boundaries, it is crucial to devise ``symmetric" forms for the representation of  piece-wise functions on the sub-domains. The form is symmetric in the sense that the continuity constraints are treated in the identical fashion in the piece-wise functions for different sub-domains. In this paper the symmetric treatment of the continuity conditions for different sub-domains is achieved by the introduction of a set of free parameters or free functions associated with the sub-domain boundaries. Overall, the constructed general form for the piece-wise function on a partitioned domain involves a set of free functions associated with the sub-domains and a set of free parameters or free functions associated with the sub-domain boundaries.
The constructed piece-wise function exactly satisfies the $C^0$ or $C^1$ continuity across the sub-domain boundaries, for any arbitrary form of the free functions or any arbitrary value of the free parameters involved therein. 

In order to arrive at a concrete numerical technique, we restrict the free functions involved in the FCE form to some finite-dimensional function space with  sufficient approximation power. Two types of function spaces are adopted in this paper, the polynomial space and a non-polynomial  space spanned by quasi-random sinusoidal functions.
When representing a field function by FCE, the unknowns are the expansion coefficients of the free functions together with the free parameters associated with the sub-domain boundaries in the FCE form. These unknowns will collectively be referred to as the FCE coefficients.

% how to use FCEs to solve BVP/IVP?
% discrete least squares formulations
% their connections to least squares FEM etc

To solve a boundary value problem on a partitioned domain using FCEs, we adopt a least squares collocation  formulation, and  represent the solution field to  partial or  ordinary differential equations (PDE/ODE) by the FCE form. We impose appropriate $C^k$ (with $k\leqslant 1$ related to the PDE order, with reformulations if necessary) continuity  across the sub-domain boundaries. We enforce the differential equations on a set of collocation points in each sub-domain, and with appropriate $C^0$ or $C^1$ FCEs the continuity conditions across the sub-domains are exactly and automatically satisfied. The commonly used boundary conditions, such as Dirichlet and Neumann types, can also be exactly satisfied by FCEs on domains with regular  geometries. Thus, the overall  problem gives rise to a linear or nonlinear algebraic system about the FCE coefficients. The number of equations and the number of unknowns in this system are not equal in general. We seek a least squares solution to this system, and compute the solution by a linear least squares method (for the linear case) or a Gauss-Newton method (for the nonlinear case)~\cite{Bjorck1996}. To solve initial-boundary value problems using FCEs, we employ a space-time approach, in which the space and time variables are treated on the same footing. The solution procedure is then analogous to that for the boundary value problems.

% review of least squares FEM/SEM etc

The solution technique as discussed above  belongs to the class of discrete least squares formulations~\cite{Eason1976}, which attempts to minimize the squared residuals over a set of discrete points in the domain~\cite{EasonM1977,ChangG1990,HeinrichsK2008,ZengTBW2019,DongL2021,DongL2021bip,DongY2022rm,NiD2023,DongW2023}. 
Apart from the discrete formulation, the continuous formulation of least squares is another widely-used approach, in which the discretization is performed after the least squares functional is defined~\cite{BochevG1998,CaiLMM1994,CaiMM1997}. The continuous formulation underlies the least squares finite elements, least squares spectral/hp elements, and a number of related techniques~\cite{Jiang1998,BrambleLP1998,KeithPFD2017,GerritsmaP2002,PontazaR2004}. We refer to the monographs~\cite{Jiang1998b,BochevG2009} and the review articles~\cite{Eason1976,BochevG1998} for detailed discussions of the least squares approach.

% characteristics of FCE methods
% comparison with hp-FEM, SEM, LS-FEM/LS-SEM etc

We present a number of numerical examples in one and two dimensions involving linear and nonlinear PDEs/ODEs to test the performance of the FCE method. This method exhibits an exponential convergence with respect to the number of expansion coefficients within the sub-domains, for both polynomial bases and the non-polynomial bases of quasi-random sinusoidal functions. When the number of sub-domains is varied systematically, a near-algebraic convergence is observed. These  characteristics can be compared to those of the spectral, spectral elements, or hp finite elements~\cite{SzaboB1991,KarniadakisS2005,ShenTW2011,SherwinK1995,KirbyS2006,YuKK2017,ZhengD2011,DongY2009,DongS2012,Dong2015}.

% unique capability for relative BCs

The FCE method has a unique advantage, compared with conventional spectral element or finite element methods, for boundary value problems involving the so-called relative boundary conditions. Relative boundary conditions are conditions representing relative constraints, which can be linear or nonlinear, of the solution field or its derivatives on the  boundary and also possibly on the domain interior. The FCE method can handle this type of problems  straightforwardly and enforce the relative boundary conditions exactly, thanks to its formulation. In contrast, these problems are significantly more challenging to traditional spectral element or finite element type methods. Several test problems of this type are presented for the purpose of illustration.

% flexbility of FCEs and least squares formulation
% mixed FCE, FCE-NC etc for solving PDEs

Additionally, the FCE method and the least squares formulation offer a great deal of flexibility in solving boundary and initial value problems. Using $C^0$ or $C^1$ FCEs can automatically satisfy the imposed $C^0$ or $C^1$ continuity  across sub-domains in the PDE problem. What is most interesting lies in that FCEs can be applied in a mixed fashion, for example, using $C^0$ FCEs for solving second-order PDEs where $C^1$ continuity  across the sub-domains needs to be imposed (for strong form of PDEs). In this case, the use of $C^0$ FCEs  ensures that the $C^0$ conditions are exactly satisfied across the sub-domains, while the $C^1$ conditions can be enforced in the least squares sense. We find that the mixed mode of usage of the FCE method is highly cost-effective, whose accuracy is comparable to that of $C^1$ FCEs with a complexity comparable to that of $C^0$ FCEs. 

% comments on C1 finite elements and spectral elements

% what are the novelties?

The contributions of this paper lie in three aspects: (i) the algorithmic construction of the general forms of piece-wise functions satisfying  exact $C^0$ or $C^1$ continuity over partitioned domains, (ii) the development of the FCE method for  solving boundary and initial value problems, (iii) the demonstration of the effectiveness and the unique advantage of the FCE method, in particular for problems involving non-traditional relative boundary conditions.

The rest of this paper is structured as follows. In Section~\ref{sec_2} we discuss the algorithmic construction of the general forms of piece-wise functions satisfying the intrinsic $C^0$ or $C^1$ continuity over partitioned domains. In Section~\ref{sec_3} we outline the least squares collocation formulation with the FCE method for solving linear and nonlinear boundary/initial value problems. In Section~\ref{sec_4} we present a number of linear and nonlinear numerical examples to demonstrate the performance of the FCE method, employing polynomial bases and non-polynomial bases for representing the free functions involved in the FCE formulation. In particular, we present a set of test problems with linear or nonlinear relative boundary conditions to illustrate the unique capability of the FCE method. Finally, Section~\ref{sec_summary} provides further discussions of the FCE method to conclude the presentation.

\section{Construction of Functionally Connected Elements}
\label{sec_2}

\subsection{Overview of Theory of Functional Connections (TFC)}
%\subsubsection{general description}
\label{sec_tfc}

The theory of functional connections (TFC)~\cite{Mortari2017,MortariL2019} provides a systematic approach for formulating functions that exactly satisfy a set of given linear constraints, such as the boundary or initial conditions.
We next  briefly discuss the aspects of TFC that are relevant to the construction of functionally connected elements in later sections.

Suppose $u(x)$ is a function defined on domain $\Omega$ satisfying the boundary condition
\begin{equation}\label{eq_1}
\mathcal Lu|_{\partial\Omega} = \kappa(x),
\quad x\in\partial\Omega,
\end{equation}
where $\mathcal L$ is a linear algebraic or differential operator, and $\kappa(x)$ is a function given on the boundary $\partial\Omega$. Then $u$ can be expressed in the following form,
\begin{equation}\label{eq_2}
u(x) = g(x) - \mathcal Ag(x) + \mathcal AG(x),
\quad x\in\Omega,
\end{equation}
where $g(x)$ is an arbitrary (free) function on $\Omega$, and $G(x)$ is a particular function satisfying~\eqref{eq_1}. $\mathcal A$ is a linear operator satisfying the following property. For any function $f(x)$ with $\mathcal Lf(x)|_{x\in\partial\Omega}$ defined,  $\mathcal Af(x)$ is a function defined for all $x\in\Omega$, which depends only on $\mathcal Lf|_{\partial\Omega}$ and satisfies, 
\begin{equation}\label{eq_3}
\mathcal L(\mathcal Af)|_{\partial\Omega} = \mathcal L f(x), \quad x\in\partial\Omega.
\end{equation}
One can verify that, for an arbitrary $g(x)$, $u(x)$ given by~\eqref{eq_2} satisfies the condition~\eqref{eq_1}.
Equation~\eqref{eq_2} is referred to as the TFC constrained expression for the condition~\eqref{eq_1}.
It should be noted that $\mathcal AG(x)$ ($x\in\Omega$) only depends on $\mathcal LG|_{\partial\Omega}=\kappa(x)$ ($x\in\partial\Omega$), and that the particular function $G(x)$ itself is not needed for the constrained expression~\eqref{eq_2}.

In the current work we restrict our attention to problems in one and two dimensions (1D/2D). 
%So what is most relevant is the 1D/2D TFC constructions. 
We next illustrate how to construct the operator $\mathcal A$ and the constrained expression~\eqref{eq_2} in 1D and 2D, respectively.

\subsubsection{One Dimension}
\label{sec_211}

Consider $\Omega=[a,b]\subset\mbb R$  and suppose the function $u(x)$ defined on $\Omega$ satisfies the following $N$ conditions,
\begin{equation}\label{eq_4}
\mathcal L_iu(x_i) = \kappa_i, 
\quad x_i\in\Omega, \quad 1\leqslant i\leqslant N,
\end{equation}
where $\mathcal L_i$ are linear algebraic/differential operators and $\kappa_i$ are prescribed  constants.
The general form for $u(x)$ satisfying~\eqref{eq_4} is given by~\eqref{eq_2}, in which the operator $\mathcal A$ is constructed as follows.

%To construct $\mathcal AC(x)$ ($x\in\Omega$), 
We choose a set of functions $p_i(x)$ ($1\leqslant i\leqslant N$), termed support functions by following~\cite{LeakeJM2022}, which satisfy a condition to be specified below. For any function $f(x)$ with $\mathcal L_i f(x_i)$ ($1\leqslant i\leqslant N$) defined, let
\begin{equation}\label{eq_5}
\mathcal Af(x) = \sum_{i=1}^N \eta_i p_i(x)
=\mbs p(x)\bm\eta,
\quad x\in\Omega,
\end{equation}
where $\mbs p(x)=(p_1(x),\dots,p_N(x))$, and $\bm\eta=(\eta_1,\dots,\eta_N)^T$ are constants to be determined.
Imposing the conditions (see~\eqref{eq_3})
\begin{equation}
\mathcal L_i(\mathcal Af)(x_i) = \mathcal L_i f(x_i),
\quad 1\leqslant i\leqslant N,
\end{equation}
gives rises to the following linear system,
\begin{equation}\label{eq_7}
\mbs P\bm\eta = \mbs F
\end{equation}
where ($\mbb M^{m\times n}$ denoting the set of $m\times n$ matrices)
\begin{equation}
\mbs P = \begin{bmatrix}
\mathcal L_1\mbs p(x_1)\\
\vdots \\
\mathcal L_N\mbs p(x_N)
\end{bmatrix}\in\mbb M^{N\times N},
\quad \mbs F=\begin{bmatrix}
\mathcal L_1f(x_1)\\ \vdots\\ \mathcal L_Nf(x_N)
\end{bmatrix}\in\mbb R^{N}.
\end{equation}
We require that $p_i(x)$ are chosen such that the matrix $\mbs P$ is non-singular.
Solving~\eqref{eq_7} for $\bm\eta$ and substituting it into
\eqref{eq_5} leads to,
\begin{equation}\label{eq_9}
\mathcal Af(x) = \mbs p(x)\mbs P^{-1}\mbs F = \bm\Phi(x)\mbs F
=\sum_{i=1}^N \mathcal [\mathcal L_if(x_i)]S_i(x),
\quad x\in\Omega,
\end{equation}
where $\bm\Phi(x)$ denotes the so-called switching functions
\begin{equation}
\bm\Phi(x) = \mbs p(x)\mbs P^{-1}
=\begin{bmatrix}
p_1(x) & \dots & p_N(x)
\end{bmatrix}
\begin{bmatrix}
\mathcal L_1p_1(x_1) & \dots & \mathcal L_1p_N(x_1)\\
\vdots & \ddots & \vdots\\
\mathcal L_Np_1(x_N) & \dots & \mathcal L_Np_N(x_N)
\end{bmatrix}^{-1}
=(S_1(x),\dots,S_N(x)).
\end{equation}
It can be observed that $\mathcal Af(x)$ given by~\eqref{eq_9} is linear with respect to $f$, and that it only depends on $\mathcal L_if(x_i)$ ($1\leqslant i\leqslant N$).
The constrained expression for $u(x)$ satisfying the conditions~\eqref{eq_4} is thus given by
\begin{equation}
u(x) = g(x) - \mathcal Ag(x) + \mathcal AG(x)
= g(x) - \sum_{i=1}^N \mathcal L_ig(x_i)S_i(x) + \sum_{i=1}^N \kappa_iS_i(x),
\end{equation}
where $g(x)$ is an arbitrary (free)  function and we have used the property that $G(x)$ is any particular function satisfying~\eqref{eq_4} by noting $\mathcal L_iG(x_i)=\kappa_i$.

Let us next consider  two particular cases of the condition~\eqref{eq_4}, which are  relevant to the current work. In the first case, we assume $\mathcal L_i=\mbs I$ (identify operator), i.e.
\begin{equation}
u(x_i) = C(x_i), \quad
x_i\in\Omega, \quad 1\leqslant i\leqslant N,
\end{equation}
where $C(x)$ is a function whose values are given on $x_i$. In this case the constrained expression is given by
\begin{equation}
u(x) = g(x) - \sum_{i=1}^N g(x_i)S_i(x)
+ \sum_{i=1}^N C(x_i)S_i(x)
\end{equation}
where $g(x)$ denotes the free function and the switching functions $S_i(x)$ are given by
\begin{equation}
\bm\Phi(x) 
=(S_1(x),\dots,S_N(x))
=\begin{bmatrix}
p_1(x) & \dots & p_N(x)
\end{bmatrix}
\begin{bmatrix}
p_1(x_1) & \dots & p_N(x_1)\\
\vdots & \ddots & \vdots\\
p_1(x_N) & \dots & p_N(x_N)
\end{bmatrix}^{-1}.
\end{equation}
One can observe that, if the support functions $p_j(x)$ are chosen as a basis of the polynomial space, then the switching functions $S_i(x)$ will become the Lagrange polynomials and $\mathcal Af(x)$ in this case will be reduced to the  polynomial that interpolates  $f$ on $x_i$ ($1\leqslant i\leqslant N$). In particular, for $N=2$ and assuming $(x_1,x_2)=(a,b)$, the switching functions become
\begin{equation}\label{eq_a60}
\phi_0(a,b,x) = S_1(x) = \frac{b-x}{b-a}, \quad
\phi_1(a,b,x) = S_2(x) = \frac{x-a}{b-a},
\end{equation}
and the general form for $u(x)$ is given by
\begin{equation}
u(x) = g(x) + [C(a) - g(a)]\phi_0(a,b,x) + [C(b)-g(b)]\phi_1(a,b,x).
\end{equation}

% conditions with function value
%   and derivatives

In the second case,
let us  consider the following conditions,
\begin{equation}\label{eq_b12}
u(x_i) = C(x_i),\quad 
\frac{du}{dx}(x_i) = C'(x_i),
\quad x_i\in\Omega, \quad 1\leqslant i\leqslant N,
\end{equation}
where $C(x)$ is a function whose values and derivatives are given on $x_i$. The general form for a function $u(x)$ ($x\in\Omega$) satisfying the conditions~\eqref{eq_b12} is  given by, 
\begin{equation}
u(x) 
= g(x) + \sum_{i=1}^N\left[C(x_i) - g(x_i)\right]S_i(x) + \sum_{i=1}^N\left[C'(x_i)-g'(x_i)\right]S_{N+i}(x),
\end{equation}
where $g(x)$ is an arbitrary (free)  function.  The switching functions $S_i(x)$ ($1\leqslant i\leqslant 2N$) are given by,
\begin{equation}\label{eq_b17}
\bm\Phi(x) 
=(S_1(x),\dots,S_{2N}(x))
= \mbs p(x)\mbs P^{-1}\mbs C,
\quad \text{where}\ 
\mbs P = \begin{bmatrix}
\mbs p(x_1)\\ \vdots\\ \mbs p(x_N)\\
\mbs p'(x_1)\\ \vdots\\ \mbs p'(x_N)
\end{bmatrix}\in\mbb M^{2N\times 2N},
\ \mbs C = \begin{bmatrix}
C(x_1)\\ \vdots\\ C(x_N)\\ C'(x_1)\\
\vdots\\ C'(x_N)
\end{bmatrix}\in\mbb R^{2N},
\end{equation}
and $\mbs p(x)=(p_1(x),\dots,p_{2N}(x))$ are the support functions chosen such that $\mbs P$ is non-singular.
If the support functions are chosen to be a basis of the polynomial space of degree at most ($2N-1$), the switching functions will be reduced to the generalized Lagrange polynomials and $\mathcal AC(x)$ in this case will become the Hermite interpolation polynomial for the conditions~\eqref{eq_b12}.
In particular, for $N=2$ and assuming $(x_1,x_2)=(a,b)$, the switching functions then become
\begin{equation}\label{eq_67}
\left\{
\begin{split}
& \psi_0(a,b,x) = S_1(x) = 3[\phi_0(a,b,x)]^2 - 2[\phi_0(a,b,x)]^3,  \\
& \psi_1(a,b,x) = S_2(x) = 3[\phi_1(a,b,x)]^2 - 2[\phi_1(a,b,x)]^3,  \\
& \varphi_0(a,b,x) = S_3(x) = \left[\left. \frac{d\phi_0(a,b,x)}{dx}\right|_{a}\right]^{-1} \left( [\phi_0(a,b,x)]^3 - [\phi_0(a,b,x)]^2\right),  \\
& \varphi_1(a,b,x) = S_4(x) = \left[\left. \frac{d\phi_1(a,b,x)}{dx}\right|_{b}\right]^{-1} \left( [\phi_1(a,b,x)]^3 - [\phi_1(a,b,x)]^2\right),
\end{split} 
\right.
\end{equation}
where $\phi_0(a,b,x)$ and $\phi_1(a,b,x)$ are defined in~\eqref{eq_a60}. The general form for $u(x)$ in this case is given by
\begin{equation}
\begin{split}
u(x) = g(x) &+ [C(a)-g(a)]\phi_0(a,b,x) + [C(b)-g(b)]\phi_1(a,b,x) \\
&+ [C'(a)-g'(a)]\varphi_0(a,b,x) + [C'(b)-g'(b)]\varphi_1(a,b,x).
\end{split}
\end{equation}
where $g(x)$ is an arbitrary function.

\subsubsection{Two Dimensions}
\label{sec_212}

We restrict our attention to the cases where the operator $\mathcal L$ in~\eqref{eq_1} is the identity  or the gradient operator, which are relevant to the construction of functionally connected elements in later sections.

Consider the domain $\Omega=[a_1,b_1]\times[a_2,b_2]\subset\mbb R^2$ and a function $u(\mbs x)$ ($\mbs x=(x,y)\in\Omega$)  satisfying the boundary conditions,
\begin{subequations}\label{eq_12}
\begin{align}
& u(x,a_2) = C(x,a_2), \quad
u(x,b_2) = C(x,b_2), \quad x\in[a_1,b_1];
\label{eq_12a} \\
& u(a_1,y) = C(a_1,y), \quad
u(b_1,y) = C(b_1,y), \quad y\in[a_2,b_2],
\label{eq_12b}
\end{align}
\end{subequations}
where $C(x,y)$ is a given distribution defined on $(x,y)\in\partial\Omega$. 
The general form of such a function is given by the following constrained expression,
\begin{equation}\label{eq_13}
u(\mbs x) = g(\mbs x) - \mathcal Ag(\mbs x) + \mathcal AC(\mbs x),
\quad \mbs x\in\Omega,
\end{equation}
where $g(\mbs x)$ is an arbitrary (free) function. Here the operator $\mathcal A$ is given by, for any function $f(\mbs x)$ that is defined on $\partial\Omega$,
\begin{equation}\label{eq_21}
\begin{split}
\mathcal Af(x,y) 
%=& f(x,a_2)\phi_0^{(2)}(y) + f(x,b_2)\phi_1^{(2)}(y) 
%+ f(a_1,y)\phi_0^{(1)}(x) + f(b_1,y)\phi_1^{(1)}(x) \\
%&-\left[ 
%f(a_1,a_2)\phi_0^{(1)}(x)\phi_0^{(2)}(y) + f(a_1,b_2)\phi_0^{(1)}(x)\phi_1^{(2)}(y)
%+ f(b_1,a_2)\phi_1^{(1)}(x)\phi_0^{(2)}(y) %\right. \\
%&\left.\quad\ 
%+f(b_1,b_2)\phi_1^{(1)}(x)\phi_1^{(2)}(y)
%\right]\\
=\ &\mbs v_1(x)^T\mbs M(x,y;f)\mbs v_2(y),
\quad (x,y)\in\Omega,
\end{split}
\end{equation}
where
\begin{equation}\label{eq_22}
\left\{
\begin{split}
%&
%\phi_0^{(i)}(x) = \frac{b_i - x}{b_i-a_i}, \quad
%\phi_1^{(i)}(x) = \frac{x-a_i}{b_i-a_i},
%\quad x\in[a_i,b_i], \quad\text{for}\ i = 1,2; \\
&
\mbs v_i(x) = \begin{bmatrix}
1 & \phi_0(a_i,b_i,x) & \phi_1(a_i,b_i,x)
\end{bmatrix}^T\in \mbb R^3,
\quad x\in[a_i,b_i], \quad\text{for}\ i = 1,2; \\
&\mbs M(x,y;f) = \begin{bmatrix}
0 & f(x,a_2) & f(x,b_2) \\
f(a_1,y) & -f(a_1,a_2) & -f(a_1,b_2) \\
f(b_1,y) & -f(b_1,a_2) & -f(b_1,b_2)
\end{bmatrix}\in\mbb M^{3\times 3}.
\end{split}
\right.
\end{equation}
One can verify that for an arbitrary  $g(\mbs x)$, the expression $u(\mbs x)$ given by~\eqref{eq_13} satisfies the conditions in~\eqref{eq_12}. One can also verify that any function $u(\mbs x)$ satisfying the conditions in~\eqref{eq_12} can be expressed in the form~\eqref{eq_13} for some $g(\mbs x)$.
Note that in~\eqref{eq_22} we have employed the support functions $p_0(x)=1$ and $p_1(x)=x$ when constructing the switching functions in both directions.

We further consider the following boundary conditions,
\begin{subequations}\label{eq_23}
\begin{align}
& u(x,a_2) = C(x,a_2), \quad
u(x,b_2) = C(x,b_2), \quad 
x\in[a_1,b_1];
\label{eq_23a} \\
& u(a_1,y) = C(a_1,y), \quad
u(b_1,y) = C(b_1,y), \quad y\in[a_2,b_2];
\label{eq_23b} \\
& \frac{\partial u}{\partial y}(x,a_2) = C_y(x,a_2), \quad
\frac{\partial u}{\partial y}(x,b_2) = C_y(x,b_2), \quad
x\in[a_1,b_1]; \label{eq_23c} \\
& \frac{\partial u}{\partial x}(a_1,y) = C_x(a_1,y), \quad
\frac{\partial u}{\partial x}(b_1,y) = C_x(b_1,y), \quad
y\in[a_2,b_2]. \label{eq_23d} 
\end{align}
\end{subequations}
Here $C(x,y)$ is a prescribed function whose value and partial derivatives  are given on $\partial\Omega$, and $C_x(x,y)$ and $C_y(x,y)$ denote the partial derivatives of $C(x,y)$ with respect to $x$ and $y$, respectively.
The general form of $u(\mbs x)$ that satisfies the conditions~\eqref{eq_23} is again given by~\eqref{eq_13}, where $g(\mbs x)$ is an arbitrary (free) function and the operator $\mathcal A$ is defined as follows. For any function $f(\mbs x)$ with its value and first/second partial derivatives defined on $\partial\Omega$, $\mathcal Af(\mbs x)$ is a function defined on $\Omega$ by,
\begin{equation}\small\label{eq_24}
\begin{split}
&\mathcal Af(x,y) 
%= f(x,a_2)\psi_0^{(2)}(y) + f(x,b_2)\psi_1^{(2)}(y) + f(a_1,y)\psi_0^{(1)}(x) + f(b_1,y)\psi_1^{(1)}(x) \\
%&+f_y(x,a_2)\varphi^{(2)}_0(y) + f_y(x,b_2)\varphi_1^{(2)}(y) + f_x(a_1,y)\varphi_0^{(1)}(x) + f_x(b_1,y)\varphi_1^{(1)}(x) \\
%&-\left[ f(a_1,a_2)\psi_0^{(1)}(x)\psi_0^{(2)}(y) + f(a_1,b_2)\psi_0^{(1)}(x)\psi_1^{(2)}(y)+ f(b_1,a_2)\psi_1^{(1)}(x)\psi_0^{(2)}(y) +f(b_1,b_2)\psi_1^{(1)}(x)\psi_1^{(2)}(y) \right. \\
%&\left. \ + f_x(a_1,a_2)\varphi_0^{(1)}(x)\psi_0^{(2)}(y) + f_x(b_1,a_2)\varphi_1^{(1)}(x)\psi_0^{(2)}(y) + f_x(a_1,b_2)\varphi_0^{(1)}(x)\psi_1^{(2)}(y) + f_x(b_1,b_2)\varphi_1^{(1)}(x)\psi_1^{(2)}(y) \right. \\
%&\left. \ + f_y(a_1,a_2)\psi_0^{(1)}(x)\varphi_0^{(2)}(y) + f_y(a_1,b_2)\psi_0^{(1)}(x)\varphi_1^{(2)}(y) + f_y(b_1,a_2)\psi_1^{(1)}(x)\varphi_0^{(2)}(y) + f_y(b_1,b_2)\psi_1^{(1)}(x)\varphi_1^{(2)}(y) \right. \\
%&\left. \ +f_{xy}(a_1,a_2)\varphi_0^{(1)}(x)\varphi_0^{(2)}(y) + f_{xy}(b_1,a_2)\varphi_1^{(1)}(x)\varphi_0^{(2)}(y) + f_{xy}(a_1,b_2)\varphi_0^{(1)}(x)\varphi_1^{(2)}(y)  + f_{xy}(b_1,b_2)\varphi_1^{(1)}(x)\varphi_1^{(2)}(y) \right] \\
&= \mbs v_1(x)^T\mbs M(x,y;f)\mbs v_2(y),
\quad (x,y)\in\Omega,
\end{split}
\end{equation}
where 
\begin{equation}\label{eq_25}
\left\{
\begin{split}
%& \psi_0^{(i)}(x) = 3\left[\phi_0^{(i)}(x) \right]^2 - 2\left[\phi_0^{(i)}(x) \right]^3, \quad \psi_1^{(i)}(x) = 3\left[\phi_1^{(i)}(x) \right]^2 - 2\left[\phi_1^{(i)}(x) \right]^3, \\
%& \varphi_0^{(i)}(x) = \left[\left. \frac{d\phi_0^{(i)}}{dx}\right|_{a_i}\right]^{-1}\left(\left[\phi_0^{(i)}(x) \right]^3-\left[\phi_0^{(i)}(x) \right]^2  \right), \\
%&\varphi_1^{(i)}(x) = \left[\left. \frac{d\phi_1^{(i)}}{dx}\right|_{b_i}\right]^{-1}\left(\left[\phi_1^{(i)}(x) \right]^3 - \left[\phi_1^{(i)}(x) \right]^2 \right),  \qquad x\in[a_i,b_i], \quad \text{for}\ i=1,2; \\
& \mbs v_i(x)=\begin{bmatrix}
1 & \psi_0(a_i,b_i,x) & \psi_1(a_i,b_i,x) & \varphi_0(a_i,b_i,x) & \varphi_1(a_i,b_i,x)
\end{bmatrix}^T\in\mbb R^5,
\ x\in[a_i,b_i], \ \text{for}\ i=1,2; \\
&\mbs M(x,y;f)=\begin{bmatrix}
0 & f(x,a_2) & f(x,b_2) & f_y(x,a_2) & f_y(x,b_2) \\
f(a_1,y) & -f(a_1,a_2) & -f(a_1,b_2) & -f_y(a_1,a_2) & -f_y(a_1,b_2) \\
f(b_1,y) & -f(b_1,a_2) & -f(b_1,b_2) & -f_y(b_1,a_2) & -f_y(b_1,b_2) \\
f_x(a_1,y) & -f_x(a_1,a_2) & -f_x(a_1,b_2) & -f_{xy}(a_1,a_2) & -f_{xy}(a_1,b_2) \\
f_x(b_1,y) & -f_x(b_1,a_2) & -f_x(b_1,b_2) & -f_{xy}(b_1,a_2) & -f_{xy}(b_1,b_2)
\end{bmatrix}\in\mbb M^{5\times 5}.
\end{split}
\right.
\end{equation}
Here $\phi_0$, $\psi_1$, $\varphi_0$ and $\varphi_1$ are defined in~\eqref{eq_67}, and $f_{xy}=\frac{\partial^2f}{\partial x\partial y}$. 
One can verify that for an arbitrary function $g(\mbs x)$ the function $u(\mbs x)$ given by~\eqref{eq_13}, with the operator $\mathcal A$ given by~\eqref{eq_24}, satisfies the conditions~\eqref{eq_23}.
One can also verify that any function $u(\mbs x)$ that satisfies the conditions in~\eqref{eq_23} can be written into the form~\eqref{eq_13} with the operator $\mathcal A$ given by~\eqref{eq_24}, for some $g(\mbs x)$.

\subsection{ Representation of  Functions  Coupled by Linear Constraints}
\label{sec_22}

We next consider how to formulate the general forms of two or more  functions that are coupled through prescribed linear constraints. This underpins the construction of piece-wise functions with  exact $C^0$ or $C^1$ continuity  across the sub-domain boundaries  in the subsequent section. We discuss the formulation in one and two dimensions individually.

\subsubsection{One Dimension}

Consider the domain $\Omega\in[a,b]\subset\mbb R$. Suppose $u_1(x)$ and $u_2(x)$ are two functions defined on $\Omega$ satisfying the following conditions,
\begin{equation}\label{eq_26}
\mathcal L_1u_1(c_i) = \mathcal L_2u_2(c_i) + \kappa_i,
\quad c_i\in[a,b], \quad 1\leqslant i\leqslant N,
\end{equation}
where $\mathcal L_1$ and $\mathcal L_2$ are two linear algebraic/differential operators, and $\kappa_i$ are prescribed values. $u_1(x)$ and $u_2(x)$ are coupled  because of these constraints.
We are interested in the general form of $u_1(x)$ and $u_2(x)$ that satisfy~\eqref{eq_26} exactly.

Let 
\begin{equation}\label{eq_27}
\mathcal L_1 u_1(c_i) = \alpha_i, 
\quad 1\leqslant i\leqslant N,
\end{equation}
where $\alpha_i$ are parameters to be determined. Then equation~\eqref{eq_26} is reduced to
\begin{equation}\label{eq_28}
\mathcal L_2 u_2(c_i) = \alpha_i- \kappa_i, \quad 1\leqslant i\leqslant N.
\end{equation}
One can note that the conditions~\eqref{eq_27} and~\eqref{eq_28} are now two independent linear constraints for $u_1$ and $u_2$, respectively.  In light of the discussions in Section~\ref{sec_211}, the general forms for $u_1(x)$ and $u_2(x)$ are given by,
\begin{subequations}\label{eq_29}
\begin{align}
&
u_1(x) = g_1(x) + \sum_{i=1}^N\left[\alpha_i - \mathcal L_1 g_1(c_i) \right]S_i(x),
\quad  x\in[a,b], \\
& u_2(x) = g_2(x) + \sum_{i=1}^N\left[\alpha_i-\kappa_i-\mathcal L_2 g_2(c_i) \right]R_i(x),
\quad  x\in[a,b],
\end{align}
\end{subequations}
where $g_1(x)$ and $g_2(x)$ are two arbitrary (free) functions, and $\bm\alpha = (\alpha_1,\dots,\alpha_N)^T$ is a set of free parameters. The switching functions $S_i(x)$ and $R_i(x)$ are given by,
\begin{subequations}
\begin{align}
&
\begin{bmatrix} S_1(x) & \dots & S_N(x) \end{bmatrix}
=\mbs p(x)\mbs P_1^{-1}
=\begin{bmatrix} p_1(x) & \dots & p_N(x) \end{bmatrix}
\begin{bmatrix} \mathcal L_1p_1(c_1) & \dots & \mathcal L_1p_N(c_1) \\
\vdots & \ddots & \vdots \\
L_1p_1(c_N) & \dots & \mathcal L_1p_N(c_N)
\end{bmatrix}^{-1}, \\
&
\begin{bmatrix} R_1(x) & \dots & R_N(x) \end{bmatrix}
=\mbs p(x)\mbs P_2^{-1}
=\begin{bmatrix} p_1(x) & \dots & p_N(x) \end{bmatrix}
\begin{bmatrix} \mathcal L_2p_1(c_1) & \dots & \mathcal L_2p_N(c_1) \\
\vdots & \ddots & \vdots \\
L_2p_1(c_N) & \dots & \mathcal L_2p_N(c_N)
\end{bmatrix}^{-1},
\end{align}
\end{subequations}
where $\mbs p(x)=(p_1(x),\dots,p_N(x))$ are the support functions such that the matrices $\mbs P_1$ and $\mbs P_2$ are non-singular.
%One can verify that the functions $u_1(x)$ and $u_2(x)$ given by~\eqref{eq_29}, for arbitrary functions $g_1(x)$ and $g_2(x)$ and for arbitrary values of $\bm\alpha$, satisfy the conditions~\eqref{eq_26}. 
%One can also verify that any two functions $u_1(x)$ and $u_2(x)$ that satisfy the conditions in~\eqref{eq_26} can be expressed in the form~\eqref{eq_29}, for some $g_1(x)$ and $g_2(x)$ and some value for $\bm\alpha$.

\begin{theorem}\label{thm_1}
(i) The functions $u_1(x)$ and $u_2(x)$ given by~\eqref{eq_29} satisfy~\eqref{eq_26}, for arbitrary functions $g_1(x)$ and $g_2(x)$ and arbitrary values of the parameter $\bm\alpha$ therein.
(ii) Any two functions $u_1(x)$ and $u_2(x)$ satisfying~\eqref{eq_26} can be expressed into~\eqref{eq_29}, for some $g_1(x)$ and $g_2(x)$ and some value of the parameter $\bm\alpha$.
\end{theorem}
\begin{proof}
(i) By straightforward verification and noting that $\mathcal L_1S_i(c_j)=\delta_{ij}$ and $\mathcal L_2R_i(c_j)=\delta_{ij}$ ($\delta_{ij}$ denoting the Kronecker delta).
(ii) Suppose $u_1(x)$ and $u_2(x)$ denote any two functions satisfying~\eqref{eq_26}. They can be written into~\eqref{eq_29}, by setting  $g_1(x)=u_1(x)$, $g_2(x)=u_2(x)$, and $\alpha_i=\mathcal L_1u_1(c_i)$ ($1\leqslant i\leqslant N$) in~\eqref{eq_29}. 
\end{proof}

Theorem~\ref{thm_1} indicates that the expressions~\eqref{eq_29} are the general forms of $u_1(x)$ and $u_2(x)$ satisfying the condition~\eqref{eq_26}.
It is noted that the general forms for $u_1$ and $u_2$ contain, besides the two free functions $g_1$ and $g_2$, a set of free parameters $\bm\alpha$.

We next consider two particular cases, $N=1$ and $N=2$, for examples. In the first case, let $c_1=c\in[a,b]$ and we consider the constraint 
\begin{equation}
u_1(c)=u_2(c).
\end{equation}
Then the general forms for $u_1(x)$ and $u_2(x)$ are given by
\begin{subequations}
\begin{align}
& u_1(x) = g_1(x) + [\alpha - g_1(c)], \quad x\in\Omega; \\
& u_2(x) = g_2(x) + [\alpha - g_2(c)], \quad x\in\Omega,
\end{align}
\end{subequations}
where $\alpha$ is a free parameter,  $g_1(x)$ and $g_2(x)$ are two free functions, and we have used the set $\{ 1  \}$ as the support function.
In the second case, let $c_1=c_2=c\in[a,b]$, and we consider the conditions,
\begin{equation}
u_1(c) = u_2(c), \quad \left.\frac{du_1}{dx}\right|_{c} = \left.\frac{du_2}{dx}\right|_{c}.
\end{equation}
Then the general forms for $u_1(x)$ and $u_2(x)$ are given by
\begin{subequations}
\begin{align}
& u_1(x) = g_1(x) + [\alpha_1 - g_1(c)] + [\alpha_2 - g_1'(c)](x-c), \quad x\in\Omega; \\
& u_2(x) = g_2(x) + [\alpha_1 - g_2(c)] + [\alpha_2 - g_2'(c)](x-c), \quad x\in\Omega,
\end{align}
\end{subequations}
where $(\alpha_1,\alpha_2)$ are free parameters, $g_1(x)$ and $g_2(x)$ are two free functions, and we have used the set $\{ 1,x\}$ as the support functions.

\begin{remark}
% comment on difference in constructions with respect to component constraints in Leake et al (2022).

By introducing certain free parameters (or certain additional free functions), the different functions  coupled through the constraints can become de-coupled, and the general form for each individual function can be constructed in the usual fashion based on TFC. This is the essence of the construction algorithm discussed here. We note that this construction is very different from that of~\cite{LeakeJM2022} (see e.g.~Section 1.3.4 therein). In~\cite{LeakeJM2022}, the constructed expressions are not symmetric with respect to different coupled functions, in the sense that the coupling constraints are embedded in the expressions for some of the functions only, while they do not appear in the constructed expressions for the rest of the functions.
In contrast, the general forms constructed here  are ``symmetric" with respect to the coupling constraints, in the sense that the constraint that  each individual function needs to satisfy is similar. The ``symmetric'' general forms constructed here are crucial to the functionally connected elements in the subsequent section.  

\end{remark}

\begin{remark}
The  construction discussed here  can be generalized in a straightforward way to more than two coupled functions. As an illustration, let us consider  three functions $u_1(x)$, $u_2(x)$ and $u_3(x)$ coupled through the following constraints,
\begin{equation}\label{eq_31}
\mathcal L_1u_1(c_i) + \mathcal L_2u_2(c_i) + \mathcal L_3u_3(c_i) = \kappa_i,
\quad c_i\in[a,b], \quad 1\leqslant i\leqslant N,
\end{equation}
where $\mathcal L_i$ ($i=1,2,3$) are linear operators.
By introducing the free parameters $\bm\alpha=(\alpha_1,\dots,\alpha_N)^T$ and $\bm\beta=(\beta_1,\dots,\beta_N)^T$,
we can reduce the constraint~\eqref{eq_31} into,
\begin{equation}
\left\{
\begin{split}
& \mathcal L_1u_1(c_i) = \alpha_i, \\
& \mathcal L_2u_2(c_i) = \beta_i, \\
& \mathcal L_3 u_3(c_i) = \kappa_i-\alpha_i-\beta_i,
\quad 1\leqslant i\leqslant N.
\end{split}
\right.
\end{equation}
Therefore, the general forms for $u_1(x)$, $u_2(x)$ and $u_3(x)$ can be obtained based on the procedure discussed in Section~\ref{sec_211}. It should be noted that these general forms will contain three free functions and two sets of free parameters $\bm\alpha$ and $\bm\beta$.

\end{remark}

\subsubsection{Two Dimensions}
\label{sec_222}

In 2D we  focus on the following settings only, which are relevant to the construction of functionally connected elements in the subsequent section.
Consider the domains $\Omega_1=[a_1,c]\times[a_2,b_2]\subset\mbb R^2$, $\Omega_2=[c,b_1]\times[a_2,b_2]\subset\mbb R^2$, and $\Omega=\Omega_1\cup\Omega_2=[a_1,b_1]\times[a_2,b_2]\subset \mbb R^2$, where $c\in(a_1,b_1)$. Let $u(\mbs x)$ denote a piece-wise function defined on $\Omega$ by
\begin{equation}
u(\mbs x)=\left\{
\begin{array}{ll}
u_1(\mbs x), & \mbs x\in\Omega_1,\\
u_2(\mbs x), & \mbs x\in\Omega_2,
\end{array}
\right.
\end{equation}
and satisfying the conditions
\begin{equation}\label{eq_34}
\left\{
\begin{split}
&
u(a_1,y) = C(a_1,y), \quad
u(b_1,y) = C(b_1,y), \quad
\forall y\in[a_2,b_2]; \\
&
u(x,a_2) = C(x,a_2), \quad
u(x,b_2) = C(x,b_2), \quad
\forall x\in[a_1,b_1],
\end{split}\right.
\end{equation}
where $C(\mbs x)$ denotes a prescribed function  on $\partial\Omega$.
In addition, we require that  $u_1(\mbs x)$ and $u_2(\mbs x)$ satisfy the constraint,
\begin{equation}\label{eq_35}
u_1(c,y) = u_2(c,y) + \kappa(y), \quad \forall y\in[a_2,b_2],
\end{equation}
or the constraints,
\begin{equation}\label{eq_36}
\left\{
\begin{split}
&
u_1(c,y) = u_2(c,y) + \kappa(y), \quad \forall y\in[a_2,b_2], \\
&
\left.\frac{\partial u_1}{\partial x}\right|_{(c,y)}
= \left.\frac{\partial u_2}{\partial x}\right|_{(c,y)}
+ \kappa_1(y),
\quad \forall y\in[a_2,b_2],
\end{split} \right.
\end{equation}
where $\kappa(y)$ and $\kappa_1(y)$ are prescribed functions.
We assume that the constraints~\eqref{eq_35} and~\eqref{eq_36} are compatible with~\eqref{eq_34} at $(x,y)=(c,a_2)$ and $(c,b_2)$, that is,
\begin{equation}
\kappa(a_2)=\kappa(b_2)=0, \quad
\kappa_1(a_2)=\kappa_1(b_2)=0.
\end{equation}
We are interested in the general forms for $u_1(\mbs x)$ and $u_2(\mbs x)$ that satisfy the conditions~\eqref{eq_34} and~\eqref{eq_35} or the conditions~\eqref{eq_34} and~\eqref{eq_36}. 

Let us first consider the general forms of $u_1(\mbs x)$ and~$u_2(\mbs x)$  satisfying~\eqref{eq_34} and~\eqref{eq_35}. Let
\begin{equation}\label{eq_38}
u_1(c,y) = \alpha(y), \quad
y\in[a_2,b_2],
\end{equation}
where $\alpha(y)$ is a function to be determined that satisfies
\begin{equation}\label{eq_39}
\alpha(a_2) = C(c,a_2), \quad
\alpha(b_2) = C(c,b_2).
\end{equation} 
Then the condition~\eqref{eq_35} is reduced to,
\begin{equation}\label{eq_40}
u_2(c,y) = \alpha(y) - \kappa(y),
\quad y\in[a_2,b_2].
\end{equation}
The general form of $\alpha(y)$ satisfying~\eqref{eq_39} is given by, based on  Section~\ref{sec_211},
\begin{equation}\label{eq_41}
\alpha(y) = \hat\alpha(y) + \left[ C(c,a_2)-\hat\alpha(a_2\right]\phi_0(a_2,b_2,y) + \left[C(c,b_2)-\hat\alpha(b_2) \right]\phi_1(a_2,b_2,y),
\quad y\in[a_2,b_2],
\end{equation}
where $\hat\alpha(y)$ is an arbitrary (free) function.

For $u_1(\mbs x)$ the conditions in~\eqref{eq_34} are reduced to
\begin{equation}\label{eq_42}
u_1(a_1,y) = C(a_1,y), \quad\forall y\in[a_2,b_2];
\quad 
u_1(x,a_2)=C(x,a_2), \quad
u_1(x,b_2) = C(x,b_2), \quad
\forall x\in[a_1,c].
\end{equation}
For $u_2(\mbs x)$ they are reduced to
\begin{equation}\label{eq_43}
u_2(b_1,y) = C(b_1,y), \quad\forall y\in[a_2,b_2];
\quad 
u_2(x,a_2)=C(x,a_2), \quad
u_2(x,b_2) = C(x,b_2), \quad
\forall x\in[c,b_1].
\end{equation}

In light of  Section~\ref{sec_212}, the general form of $u_1(\mbs x)$ satisfying~\eqref{eq_42} and~\eqref{eq_38} is given by
\begin{equation}\small\label{eq_44}
\left\{
\begin{split}
& u_1(\mbs x) = g_1(\mbs x) + \mbs v_1(x)^T\mbs M_1 \mbs v_2(y), \\
& \mbs v_1(x) = \begin{bmatrix}
1 & \phi_0(a_1,c,x) & \phi_1(a_1,c,x)
\end{bmatrix}^T, \quad
\mbs v_2(y) = \begin{bmatrix}
1 & \phi_0(a_2,b_2,y) & \phi_1(a_2,b_2,y)
\end{bmatrix}^T, \\
& \mbs M_1 = \begin{bmatrix}
0 & C(x,a_2)-g_1(x,a_2) & C(x,b_2)-g_1(x,b_2)\\
C(a_1,y)-g_1(a_1,y) & -[ C(a_1,a_2)-g_1(a_1,a_2)] & -[C(a_1,b_2)-g_1(a_1,b_2)]\\
\alpha(y)-g_1(c,y) & -[C(c,a_2)-g_1(c,a_2)] & -[C(c,b_2)-g_1(c,b_2)]
\end{bmatrix}, 
\end{split} \right.
\end{equation}
in which $g_1(\mbs x)$ is an arbitrary (free) function and $\alpha(y)$ is given by~\eqref{eq_41}.
The general form of $u_2(\mbs x)$ satisfying~\eqref{eq_43} and~\eqref{eq_40} is given by
\begin{equation}\small\label{eq_45}
\left\{
\begin{split}
& u_2(\mbs x) = g_2(\mbs x) + \mbs w_1(x)^T \mbs M_2\mbs v_2(y), \\
& \mbs w_1(x) = \begin{bmatrix} 1 & \phi_0(c,b_1,x) & \phi_1(c,b_1,x) \end{bmatrix}^T, \\
& \mbs M_2 = \begin{bmatrix}
0 & C(x,a_2)-g_2(x,a_2) & C(x,b_2)-g_2(x,b_2)\\
\alpha(y)-\kappa(y)-g_2(c,y) & -[ C(c,a_2)-g_2(c,a_2)] & -[C(c,b_2)-g_2(c,b_2)]\\
C(b_1,y)-g_2(b_1,y) & -[C(b_1,a_2)-g_2(b_1,a_2)] & -[C(b_1,b_2)-g_2(b_1,b_2)]
\end{bmatrix}, 
\end{split} \right.
\end{equation}
where $g_2(\mbs x)$ is an arbitrary (free) function, $\alpha(y)$ is given by~\eqref{eq_41}, and $\mbs v_2(y)$ is defined in~\eqref{eq_44}.
%
%Equations~\eqref{eq_44} and~\eqref{eq_45} are the general expressions for $u_1(\mbs x)$ and $u_2(\mbs x)$ that  satisfy the conditions~\eqref{eq_34} and~\eqref{eq_35}. 

\begin{theorem}\label{thm_2}
(i) The functions $u_1(\mbs x)$ and $u_2(\mbs x)$ given by~\eqref{eq_44} and~\eqref{eq_45} satisfy the conditions~\eqref{eq_42},~\eqref{eq_43} and~\eqref{eq_35}, for arbitrary functions $g_1(\mbs x)$, $g_2(\mbs x)$, and $\hat\alpha(y)$ involved therein.
(ii) Any two functions $u_1(\mbs x)$ and $u_2(\mbs x)$ satisfying the conditions~\eqref{eq_42},~\eqref{eq_43} and~\eqref{eq_35} can be expressed into~\eqref{eq_44} and~\eqref{eq_45}, for some $g_1(\mbs x)$, $g_2(\mbs x)$, and $\hat\alpha(y)$.
\end{theorem}
\begin{proof}
(i) The proof follows from straightforward verification.
(ii) Suppose $\omega_1(\mbs x)$ and $\omega_2(\mbs x)$ denote any two functions satisfying~\eqref{eq_42},~\eqref{eq_43} and~\eqref{eq_35}, and that $u_1(\mbs x)$ and $u_2(\mbs x)$ are defined by~\eqref{eq_44} and~\eqref{eq_45}. Choose $g_1=\omega_1(\mbs x)$, $g_2=\omega_2(\mbs x)$, and $\hat\alpha(y)=\omega_1(c,y)$ in~\eqref{eq_44} and~\eqref{eq_45}. Then~\eqref{eq_44} and~\eqref{eq_45} are reduced to $u_1(\mbs x)=\omega_1(\mbs x)$ and $u_2(\mbs x)=\omega_2(\mbs x)$, respectively.
\end{proof}

This theorem shows that the expressions~\eqref{eq_44} and~\eqref{eq_45} are the general forms for $u_1(\mbs x)$ and $u_2(\mbs x)$ that satisfy the conditions~\eqref{eq_34} and~\eqref{eq_35}.
It is noted that, besides $g_1(\mbs x)$ and $g_2(\mbs x)$, these general forms  contain an additional free function $\hat\alpha(y)$ associated with the shared interface $x=c$.

We next construct the general forms for $u_1(\mbs x)$ and $u_2(\mbs x)$ that satisfy the conditions~\eqref{eq_34} and~\eqref{eq_36}. Apart from~\eqref{eq_38}, let
\begin{equation}\label{eq_46}
\left.\frac{\partial u_1}{\partial x} \right|_{(c,y)} = \beta(y),
\quad y\in[a_2,b_2],
\end{equation}
where $\beta(y)$ is a function to be determined that satisfies
\begin{equation}\label{eq_47}
\beta(a_2) = C_x(c,a_2), \quad
\beta(b_2) = C_x(c,b_2).
\end{equation}
Then the second condition in~\eqref{eq_36} is transformed into
\begin{equation}\label{eq_48}
\left.\frac{\partial u_2}{\partial x} \right|_{(c,y)} = \beta(y) - \kappa_1(y), \quad y\in[a_2,b_2].
\end{equation}
The general form for $\beta(y)$ satisfying~\eqref{eq_47} is given by,
in light of Section~\ref{sec_211},
\begin{equation}\label{eq_49}
\beta(y) = \hat\beta(y) + [C_x(c,a_2)-\hat\beta(a_2)]\phi_0(a_2,b_2,y)
+ [C_x(c,b_2)-\hat\beta(b_2)]\phi_1(a_2,b_2,y), \quad y\in[a_2,b_2],
\end{equation}
where $\hat\beta(y)$ is an arbitrary (free) function.

Therefore, the conditions that $u_1(\mbs x)$ needs to satisfy consist of~\eqref{eq_42}, \eqref{eq_38}, and~\eqref{eq_46}. The general form for such a function is given by
\begin{equation}\small\label{eq_50}
\left\{
\begin{split}
& u_1(\mbs x) = g_1(\mbs x)+ \mbs z_1(x)\mbs M_3\mbs v_2(y), \\
& \mbs z_1(x) = \begin{bmatrix} 1 & \varphi_0(x) & \varphi_1(x) & \varphi_2(x) \end{bmatrix}^T, \\
& \varphi_0(x)= \left(\frac{c-x}{c-a_1} \right)^2, \quad
\varphi_1(x) = \frac{x-a_1}{c-a_1}\left(1+ \frac{c-x}{c-a_1}\right),
\quad
\varphi_2(x) = -\frac{(x-a_1)(c-x)}{c-a_1}, \\
& \mbs M_3 = \begin{bmatrix}
0 & C(x,a_2)-g_1(x,a_2) & C(x,b_2)-g_1(x,b_2)\\
C(a_1,y)-g_1(a_1,y) & -[C(a_1,a_2)-g_1(a_1,a_2)] & -[C(a_1,b_2)-g_1(a_1,b_2)]\\
\alpha(y)-g_1(c,y) & -[C(c,a_2)-g_1(c,a_2)] & -[C(c,b_2)-g_1(c,b_2)] \\
\beta(y)-g_{1x}(c,y) & -[C_x(c,a_2)-g_{1x}(c,a_2)] & -[C_x(c,b_2)-g_{1x}(c,b_2)]
\end{bmatrix}, \\
\end{split} \right.
\end{equation}
where $g_1(\mbs x)$ is an arbitrary (free) function, $g_{1x}=\frac{\partial g_1}{\partial x}$, $\alpha(y)$ is given by~\eqref{eq_41}, and $\beta(y)$ is given by~\eqref{eq_49}.
The conditions that $u_2(\mbs x)$ needs to satisfy consist of~\eqref{eq_43}, \eqref{eq_40}, and~\eqref{eq_48}. The general form for $u_2(\mbs x)$ is given by
\begin{equation}\small\label{eq_52}
\left\{
\begin{split}
& u_2(\mbs x) = g_2(\mbs x) + \mbs z_2(x)\mbs M_4\mbs v_2(y), \\
& \mbs z_2(x) = \begin{bmatrix}
1 & \hat\varphi_0(x) & \hat\varphi_1(x) & \hat\varphi_2(x)
\end{bmatrix}, \\
& \hat\varphi_0(x)= \frac{b_1-x}{b_1-c}\left(1+ \frac{x-c}{b_1-c}\right), \quad
\hat\varphi_1(x) = \left(\frac{x-c}{b_1-c} \right)^2, \quad
\hat\varphi_2(x) = \frac{(x-c)(b_1-x)}{b_1-c}, \\
& \mbs M_4 = \begin{bmatrix}
0 & C(x,a_2)-g_2(x,a_2) & C(x,b_2)-g_2(x,b_2)\\
\alpha(y)-\kappa(y)-g_2(c,y) & -[C(c,a_2)-g_2(c,a_2)] & -[C(c,b_2)-g_2(c,b_2)]\\
C(b_1,y)-g_2(b_1,y) & -[C(b_1,a_2)-g_2(b_1,a_2)] & -[C(b_1,b_2)-g_2(b_1,b_2)] \\
\beta(y)-\kappa_1(y)-g_{2x}(c,y) & -[C_x(c,a_2)-g_{2x}(c,a_2)] & -[C_x(c,b_2)-g_{2x}(c,b_2)]
\end{bmatrix}, 
\end{split} \right.
\end{equation}
where $g_2(\mbs x)$ is an arbitrary (free) function.
Equations~\eqref{eq_50} and~\eqref{eq_52} provide the constrained expressions for $u_1(\mbs x)$ and $u_2(\mbs x)$ that  satisfy the constraints~\eqref{eq_34} and~\eqref{eq_36}. 

\begin{theorem}\label{thm_3}
(i) The functions $u_1(\mbs x)$ and $u_2(\mbs x)$ given by~\eqref{eq_50} and~\eqref{eq_52} satisfy the conditions~\eqref{eq_34} and~\eqref{eq_36}, for arbitrary functions $g_1(\mbs x)$, $g_2(\mbs x)$, $\hat\alpha(y)$ and $\hat\beta(y)$ therein. (ii) Any two functions $u_1(\mbs x)$ and $u_2(\mbs x)$ satisfying~\eqref{eq_34} and~\eqref{eq_36} can be expressed into~\eqref{eq_50} and~\eqref{eq_52}, for some $g_1(\mbs x)$,  $g_2(\mbs x)$, $\hat\alpha(y)$ and $\hat\beta(y)$.
\end{theorem}

The proof of this theorem follows the same strategy as that of Theorem~\ref{thm_2}.
This theorem affirms that the expressions~\eqref{eq_50} and~\eqref{eq_52} are indeed the general forms for $u_1(\mbs x)$ and $u_2(\mbs x)$ that satisfy the conditions~\eqref{eq_34} and~\eqref{eq_36}.
Note that these forms contain four free functions $g_1(\mbs x)$, $g_2(\mbs x)$, $\hat\alpha(y)$, and $\hat\beta(y)$.

%One can verify that $u_1(\mbs x)$ and $u_2(\mbs x)$ given by the expressions~\eqref{eq_44} and~\eqref{eq_45} satisfy the conditions~\eqref{eq_34} and~\eqref{eq_35}, for any function $g_1(\mbs x)$, $g_2(\mbs x)$ and $\hat\alpha(y)$ contained therein.
%One can also verify that any two functions $u_1(\mbs x)$ and $u_2(\mbs x)$ that satisfy the conditions~\eqref{eq_34} and~\eqref{eq_35} can be expressed in the form of~\eqref{eq_44} and~\eqref{eq_45}, for some $g_1(\mbs x)$, $g_2(\mbs x)$ and $\hat\alpha(y)$.
%Similarly, one can verify that $u_1(\mbs x)$ and $u_2(\mbs x)$ given by the expressions~\eqref{eq_50} and~\eqref{eq_52} satisfy the conditions~\eqref{eq_34} and~\eqref{eq_36}, for any function $g_1(\mbs x)$, $g_2(\mbs x)$,  $\hat\alpha(y)$, and $\hat\beta(y)$ contained therein.
%One can also verify that any two functions $u_1(\mbs x)$ and $u_2(\mbs x)$ that satisfy the conditions~\eqref{eq_34} and~\eqref{eq_36} can be expressed in the form of~\eqref{eq_50} and~\eqref{eq_52}, for some $g_1(\mbs x)$, $g_2(\mbs x)$,  $\hat\alpha(y)$, and $\hat\beta(y)$.

\begin{remark}
% comment on the constrained expressions of coupled functions, the free functions, the spaces for these free functions
% how are they related to constructions of elements?

The expressions~\eqref{eq_44} and~\eqref{eq_45} provide the general forms for two functions respectively defined on two sub-domains that satisfy
%~\eqref{eq_34} on the domain boundary and 
the continuity condition~\eqref{eq_35} (of $C^0$ type) on the shared sub-domain boundary. Similarly, \eqref{eq_50} and~\eqref{eq_52} provide the general forms of the two functions that satisfy
%~\eqref{eq_34} on the domain boundary and 
the continuity condition~\eqref{eq_36} (of $C^1$ type) on the shared sub-domain boundary. These expressions  essentially define two $C^0$ or $C^1$ type ``elements''.  It is notable that these expressions contain a number of free functions, such as $g_1(\mbs x)$, $g_2(\mbs x)$, $\hat\alpha(y)$ and $\hat\beta(y)$, which can be arbitrary. By restricting these free functions to a particular  function space, one can attain a specific type of element. In particular, one has the freedom to choose function spaces other than the classical polynomial space,  leading to ``elements'' beyond the traditional sense.

\end{remark}

\subsection{Constructing Functionally Connected Elements (FCE)}

\subsubsection{Functionally Connected Elements in 1D}\label{sec_231}

Consider a 1D domain $\Omega=[a,b]$, which is partitioned into $N$ ($N\geqslant 1$) sub-domains or ``elements" by the points $X_i$ ($0\leqslant i\leqslant N$), with
$a=X_0<X_1<\dots<X_N=b$. We  refer to $\Omega_i=[X_i,X_{i+1}]$ ($0\leqslant i\leqslant N-1$) as the $i$-th element below. Consider a piece-wise function $u(x)$ ($x\in\Omega$) defined by,
\begin{equation}\label{eq_54}
u(x)=\left\{
\begin{array}{ll}
u_0(x), & x\in\Omega_0; \\
u_1(x), & x\in\Omega_1; \\
\cdots \\
u_{N-1}(x), & x\in\Omega_{N-1},
\end{array}
\right.
\end{equation}
where $u_i(x)$ is defined on $\Omega_i$ for $0\leqslant i\leqslant N-1$. We impose the following  conditions on $u(x)$ at the interior element boundaries:
\begin{equation}\label{eq_55}
u_i(X_{i+1}) = u_{i+1}(X_{i+1}), \quad 0\leqslant i\leqslant N-2;
\end{equation}
or
\begin{equation}\label{eq_56}
\left\{
\begin{split}
& u_i(X_{i+1}) = u_{i+1}(X_{i+1}), \\
& \left.\frac{du_i}{dx}\right|_{X_{i+1}} = \left.\frac{du_{i+1}}{dx}\right|_{X_{i+1}}, \quad 0\leqslant i\leqslant N-2.
\end{split}
\right.
\end{equation}
Equations~\eqref{eq_55} and~\eqref{eq_56} are the $C^0$ and $C^1$ continuity conditions, respectively. Those elements satisfying the $C^0$ (or $C^1$) continuity conditions will be referred to as $C^0$ (or $C^1$) elements.

%In addition, we consider the following conditions for $u(x)$,
%\begin{equation}\label{eq_57}
%u(a) = \alpha_0, \quad
%u(b) = \alpha_N,
%\end{equation}
%where $\alpha_0$ and $\alpha_N$ are given constants.

We are interested in the general forms of $u(x)$ that satisfy the $C^0$ or $C^1$ continuity conditions. These general forms are constructed  below using the ideas from Section~\ref{sec_22}. We refer to the elements equipped with such general forms as Functionally Connected Elements (FCE).

\paragraph{$C^0$ FCEs}

Let us now construct the general form of  $u(x)$ defined in~\eqref{eq_54} satisfying the $C^0$ condition~\eqref{eq_55}. 
Let $u(X_i)=\alpha_i$ ($0\leqslant i\leqslant N$), where $\alpha_i$ ($0\leqslant i\leqslant N$) denote a set of free parameters. Then on $\Omega_i$, $u_i(x)$ satisfies the  conditions,
\begin{equation}\label{eq_59}
u_i(X_i) = \alpha_i, \quad
u_i(X_{i+1}) = \alpha_{i+1},
\quad 0\leqslant i\leqslant N-1.
\end{equation}
The general form for $u_i(x)$ satisfying these conditions is given by, based on Section~\ref{sec_211},
\begin{equation}\label{eq_60}
u_i(x) = g_i(x) + [\alpha_i - g_i(X_i)]\phi_{0}(X_i,X_{i+1},x)
+ [\alpha_{i+1}-g_i(X_{i+1})]\phi_{1}(X_i,X_{i+1},x),
\ x\in\Omega_i,
\ 0\leqslant i\leqslant N-1,
\end{equation}
where $g_i(x)$ is an arbitrary (free) function defined on $\Omega_i$. 
%the switching functions are 
%\begin{equation}\label{eq_a60}
%\phi_0(\chi_1,\chi_2,x) = \frac{\chi_2-x}{\chi_2-\chi_1}, \quad
%\phi_1(\chi_1,\chi_2,x) = \frac{x-\chi_1}{\chi_2-\chi_1}, 
%\end{equation}
%and we have used  $\{1, x \}$ as the set of support functions for simplicity.

The function $u(x)$ defined by~\eqref{eq_54}, with $u_i(x)$ given by~\eqref{eq_60}, characterizes the 1D $C^0$ FCEs. It is notable that this function contains $N$ free functions $g_i(x)$ ($0\leqslant i\leqslant N-1$) and $(N+1)$ free parameters $\alpha_i$ ($0\leqslant i\leqslant N$). 

%It can be shown that this function satisfies the $C^0$ continuity condition~\eqref{eq_55} for arbitrary $g_i(x)$ and $\alpha_i$, and that any function $u(x)$ defined in~\eqref{eq_54} that satisfies~\eqref{eq_55} can be expressed in the form of~\eqref{eq_60}.  This is indeed the constrained expression for functions defined on  $C^0$ FCEs.

\begin{theorem}\label{thm_4}
(i) The functions $u_i(x)$ given by~\eqref{eq_60} satisfy the conditions~\eqref{eq_55}, for arbitrary $g_i(x)$ and arbitrary values of the parameters $\alpha_i$ therein. (ii) Any set of functions $u_i(x)$ ($0\leqslant i\leqslant N-1$) satisfying the conditions~\eqref{eq_55} can be expressed in terms of~\eqref{eq_60}, for some $g_i(x)$ ($0\leqslant i\leqslant N-1$) and some value of the parameters $\alpha_i$ ($0\leqslant i\leqslant N$).
\end{theorem}
\begin{proof}
(i) By straightforward verification.
(ii) Suppose $u_i(x)$ ($0\leqslant i\leqslant N-1$) are a set of functions satisfying~\eqref{eq_55}. They can be expressed in the form~\eqref{eq_60}, by setting $g_i=u_i(x)$ ($0\leqslant i\leqslant N-1$), $\alpha_i=u_i(X_i)$ ($0\leqslant i\leqslant N-1$) and $\alpha_N=u_{N-1}(b)$ in~\eqref{eq_60}.
\end{proof}

% free functions, bases

This theorem affirms that~\eqref{eq_60} provides the general form of $C^0$ piece-wise functions defined on the domain.
Note that $g_i(x)$ therein can be arbitrary functions. To devise a computational technique, we restrict $g_i(x)$ to a finite-dimensional function space $\mathcal F_m(\Omega_i)$ defined by
\begin{equation}\label{eq_61}
\mathcal F_m(\Omega_i)=\text{span}\left\{
\Phi_{ij}(x),\ x\in\Omega_i,\ 1\leqslant j\leqslant m
\right\}, 
\quad 0\leqslant i\leqslant N-1,
\end{equation}
where $m$ denotes the dimension of $\mathcal F_m(\Omega_i)$ and $\Phi_{ij}(x)$ ($1\leqslant j\leqslant m$) are its bases. We require that the set of functions
\begin{equation}\label{eq_c61}
\left\{
\phi_0(X_i,X_{i+1},x),\ 
\phi_1(X_i,X_{i+1},x),\ 
\Phi_{i1}(x),\ \dots,\ 
\Phi_{im}(x)
\right\}, \quad x\in\Omega_i,
\end{equation}
be linearly independent for  $0\leqslant i\leqslant N-1$.
Let 
\begin{equation}\label{eq_62}
g_i(x) = \sum_{j=1}^m \hat g_{ij}\Phi_{ij}(x), 
\quad x\in\Omega_i, \quad
0\leqslant i\leqslant N-1,
\end{equation}
where $\hat g_{ij}$ are the expansion coefficients.
Then $u_i(x)$ in~\eqref{eq_60} is transformed into
\begin{equation}\label{eq_64}
\begin{split}
u_i(x) =& \sum_{j=1}^m\hat g_{ij}\left[
\Phi_{ij}(x) - \Phi_{ij}(X_i)\phi_0(X_i,X_{i+1},x) - \Phi_{ij}(X_{i+1})\phi_1(X_i,X_{i+1},x)
\right] \\
& + \alpha_i \phi_0(X_i,X_{i+1},x)
+ \alpha_{i+1} \phi_1(X_i,X_{i+1},x),
\quad x\in\Omega_i,
\quad 0\leqslant i\leqslant N-1.
\end{split}
\end{equation}
The function $u(x)$ defined by~\eqref{eq_54}, with $u_i(x)$ given by~\eqref{eq_64}, satisfies the $C^0$ continuity condition~\eqref{eq_55} automatically, irrespective of the choice for the bases $\Phi_{ij}(x)$. Once the basis functions $\Phi_{ij}(x)$ are specified, $\hat g_{ij}$ ($0\leqslant i\leqslant N-1$, $1\leqslant j\leqslant m$) and $\alpha_i$ ($0\leqslant i\leqslant N$) are the unknown coefficients to be determined.

\paragraph{$C^1$ FCEs}

Let us next consider the general form of $u(x)$ defined in~\eqref{eq_54} satisfying the $C^1$ conditions in~\eqref{eq_56}.
Let
\begin{subequations}\label{eq_a63}
\begin{align}
& u(X_i) = \alpha_i, \quad 0\leqslant i\leqslant N; \\
& \frac{du}{dx}(X_i) = \beta_i, \quad
0\leqslant i\leqslant N,
\end{align}
\end{subequations}
where $\alpha_i$ ($0\leqslant i\leqslant N$) and $\beta_i$ ($0\leqslant i\leqslant N$) are free parameters. Then on $\Omega_i$ the function $u_i(x)$ satisfies the following conditions,
\begin{subequations}\label{eq_a64}
\begin{align}
& u_i(X_i) = \alpha_i, \quad u_i(X_{i+1}) = \alpha_{i+1}, \\
& \frac{du_i}{dx}(X_i) = \beta_i, \quad \frac{du_i}{dx}(X_{i+1}) = \beta_{i+1}, \quad 0\leqslant i\leqslant N-1.
\end{align}
\end{subequations}
In light of Section~\ref{sec_211}, the general form of $u_i(x)$ satisfying these conditions are given by ($0\leqslant i\leqslant N-1$),
\begin{equation}\label{eq_66}
\begin{split}
u_i(x) =&\ g_i(x) + [\alpha_i - g_i(X_i)]\psi_0(X_i,X_{i+1},x)
+ [\alpha_{i+1} - g_i(X_{i+1})]\psi_1(X_i,X_{i+1},x) \\
&+ [\beta_i - g'_i(X_i)]\varphi_0(X_i,X_{i+1},x)
+ [\beta_{i+1} - g'_i(X_{i+1})]\varphi_1(X_i,X_{i+1},x),
\quad x\in\Omega_i,
\end{split}
\end{equation}
where $g_i(x)$ is an arbitrary (free) function, and $\psi_0$, $\psi_1$, $\varphi_0$ and $\varphi_1$ are defined in~\eqref{eq_67}.
The following result holds.

\begin{comment}
the switching functions are
\begin{equation}\label{eq_67}
\left\{
\begin{split}
& \psi_0(\chi_1,\chi_2,x) = 3[\phi_0(\chi_1,\chi_2,x)]^2 - 2[\phi_0(\chi_1,\chi_2,x)]^3,  \\
& \psi_1(\chi_1,\chi_2,x) = 3[\phi_1(\chi_1,\chi_2,x)]^2 - 2[\phi_1(\chi_1,\chi_2,x)]^3,  \\
& \varphi_0(\chi_1,\chi_2,x) = \left[\left. \frac{d\phi_0(\chi_1,\chi_2,x)}{dx}\right|_{\chi_1}\right]^{-1} \left( [\phi_0(\chi_1,\chi_2,x)]^3 - [\phi_0(\chi_1,\chi_2,x)]^2\right),  \\
& \varphi_1(\chi_1,\chi_2,x) = \left[\left. \frac{d\phi_1(\chi_1,\chi_2,x)}{dx}\right|_{\chi_2}\right]^{-1} \left( [\phi_1(\chi_1,\chi_2,x)]^3 - [\phi_1(\chi_1,\chi_2,x)]^2\right).
\end{split} 
\right.
\end{equation}
Here $\phi_0(\chi_1,\chi_2,x)$ and $\phi_1(\chi_1,\chi_2,x)$ are defined in~\eqref{eq_a60}. Note that we have used the set $\{\ 1,x,x^2,x^3 \ \}$ as the support functions for deriving these switching functions.
\end{comment}

%One can verify that the function defined in~\eqref{eq_54}, with $u_i(x)$ given by~\eqref{eq_66}, satisfy the condition~\eqref{eq_56}, for arbitrary functions $g_i(x)$ ($0\leqslant i\leqslant N-1$) and arbitrary values of the parameters $\alpha_i$ ($0\leqslant i\leqslant N$) and $\beta_i$ ($0\leqslant i\leqslant N$). One can also verify that any piece-wise function $u(x)$ satisfying the condition~\eqref{eq_56} can be written in the form~\eqref{eq_66}. 

\begin{theorem}\label{thm_5}
(i) The functions $u_i(x)$ given by~\eqref{eq_66} satisfy the conditions in~\eqref{eq_56}, for arbitrary $g_i(x)$ and arbitrary values of the parameters $(\alpha_i,\beta_i)$ therein. (ii) Any set of functions $u_i(x)$ ($0\leqslant i\leqslant N-1$) satisfying~\eqref{eq_56} can be expressed in the form~\eqref{eq_66}, for some $g_i(x)$ ($0\leqslant i\leqslant N-1$) and some value of the parameters $(\alpha_i,\beta_i)$ ($0\leqslant i\leqslant N$).
\end{theorem}
\noindent The proof of this theorem follows the same strategy as that of Theorem~\ref{thm_4}.

To arrive at a computational technique, we restrict $g_i(x)$ to the function space $\mathcal F_m(\Omega_i)$ defined by~\eqref{eq_61}, where the basis functions are such that the set
\begin{equation}\label{eq_c67}
\{\  
\psi_0(X_i,X_{i+1},x), \psi_1(X_i,X_{i+1},x), \varphi_0(X_i,X_{i+1}, x), \varphi_1(X_i,X_{i+1},x), \Phi_{i1}(x), \dots, \Phi_{im}(x)
\ \}, \ x\in\Omega_i,
\end{equation}
is linearly independent for $0\leqslant i\leqslant N-1$.
 Then the expression~\eqref{eq_66} is transformed into
\begin{equation}\label{eq_68}
\begin{split}
 u_i(x) =& \sum_{j=1}^m \hat g_{ij}\left[ 
\Phi_{ij}(x) - \Phi_{ij}(X_i)\psi_0(X_i,X_{i+1},x) - \Phi_{ij}(X_{i+1})\psi_1(X_i,X_{i+1},x) \right. \\
& \qquad\quad
\left.
-\Phi'_{ij}(X_i)\varphi_0(X_i,X_{i+1},x) - \Phi'_{ij}(X_{i+1})\varphi_1(X_i,X_{i+1},x)
\right] \\
& + \alpha_i\psi_0(X_i,X_{i+1},x)
+ \alpha_{i+1} \psi_1(X_i,X_{i+1},x)
+ \beta_i\varphi_0(X_i,X_{i+1},x)
+ \beta_{i+1} \varphi_1(X_i,X_{i+1},x), \\
& x\in\Omega_i, \quad 0\leqslant i\leqslant N-1,
\end{split}
\end{equation}
where we have used the expansion~\eqref{eq_62}. 
In~\eqref{eq_68}, $\hat g_{ij}$ ($0\leqslant i\leqslant N-1$, $1\leqslant j\leqslant m$), $\alpha_i$ ($0\leqslant i\leqslant N$), and $\beta_i$ ($0\leqslant i\leqslant N$) are the unknown parameters to be determined.

\subsubsection{Functionally Connected Elements in 2D}\label{sec_232}

Consider the domain $\Omega = [a_1,b_1]\times[a_2,b_2]\subset\mbb R^2$, %with given constants $a_i$ and $b_i$ ($i=1,2$). 
and suppose it is partitioned into $N_x$ elements along the $x$ direction by the points $a_1=X_0<X_1<\cdots<X_{N_x}=b_1$,
and into $N_y$ elements along the $y$ direction by the points $a_2=Y_0<Y_1<\cdots<Y_{N_y}=b_2$. Let $\Omega_{ij}=[X_i,X_{i+1}]\times[Y_j,Y_{j+1}]$ denote the element with the index $e_{ij}=iN_y+j$, for $0\leqslant i\leqslant N_x-1$ and $0\leqslant j\leqslant N_y-1$.

Consider a piece-wise function $u(\mbs x)$ ($\mbs x=(x,y)\in\Omega$) defined by
\begin{equation}\label{eq_69}
u(\mbs x) = \left\{
\begin{array}{ll}
u_{0}(\mbs x), & \mbs x\in\Omega_{00}, \\
u_1(\mbs x), & \mbs x\in\Omega_{01},\\
\cdots \\
u_{e_{ij}}(\mbs x), & \mbs x\in\Omega_{ij}, \\
\cdots \\
u_{N-1}(\mbs x), & \mbs x\in\Omega_{N_x-1,N_y-1},
\end{array}
\right.
\end{equation}
where $N=N_xN_y$.
We aim for the general forms of $u(\mbs x)$ satisfying $C^0$ or $C^1$ continuity conditions across the elements.

\paragraph{$C^0$ FCEs} 

We impose the following $C^0$ continuity conditions across the elements,
\begin{subequations}\label{eq_71}
\begin{align}
& u_{e_{ij}}(X_{i+1},y) = u_{e_{i+1,j}}(X_{i+1},y), \quad
y\in[Y_j,Y_{j+1}], \quad
0\leqslant (i,j)\leqslant (N_x-2,N_y-1); \\
& u_{e_{ij}}(x,Y_{j+1}) = u_{e_{i,j+1}}(x,Y_{j+1}), \quad
x\in [X_i,X_{i+1}], \quad
0\leqslant (i,j)\leqslant (N_x-1,N_y-2).
\end{align}
\end{subequations}
Here ``$0\leqslant (i,j)\leqslant (n_1,n_2)$" or ``$(0,0)\leqslant (i,j)\leqslant (n_1,n_2)$" denotes $0\leqslant i\leqslant n_1$ and $0\leqslant j\leqslant n_2$. We will employ these and similar notations hereafter whenever it is convenient.
We are interested in the general form of $u(\mbs x)$ defined in~\eqref{eq_69} satisfying the conditions~\eqref{eq_71}. %and~\eqref{eq_70}.

% need to define parameters for u on the vertices, parameter functions on the edges of elements

Let
\begin{equation}
u(X_i,Y_j) = \hat\alpha_{ij}, \quad 0\leqslant (i,j)\leqslant (N_x,N_y),
\end{equation}
where $\hat\alpha_{ij}$ 
%for $1\leqslant (i,j)\leqslant (N_x-1,N_y-1)$ 
are parameters to be determined. 
%and the $\hat\alpha_{ij}=C(X_i,Y_j)$ when $i=0, N_x$ or $j=0,N_y$ in light of~\eqref{eq_70}.
Let
\begin{subequations}
\begin{align}
& u(X_i,y) = G_{ij}(y), \quad y\in[Y_j,Y_{j+1}], \quad 0\leqslant (i,j)\leqslant (N_x,N_y-1), \\
& u(x,Y_j) = H_{ij}(x), \quad x\in[X_i, X_{i+1}], \quad
0\leqslant (i,j)\leqslant (N_x-1,N_y).
\end{align}
\end{subequations}
Here $G_{ij}(y)$ are functions to be determined that satisfy
\begin{align}
& G_{ij}(Y_j) = \hat\alpha_{ij}, \quad G_{ij}(Y_{j+1}) = \hat\alpha_{i,j+1},
\quad 0\leqslant (i,j)\leqslant(N_x, N_y-1). \label{eq_74a} 
\end{align}
$H_{ij}(x)$ are functions to be determined that satisfy
\begin{align}
& H_{ij}(X_i) = \hat\alpha_{ij}, \quad H_{ij}(X_{i+1}) = \hat\alpha_{i+1,j}, \quad
0\leqslant (i,j)\leqslant (N_x-1,N_y). \label{eq_75a} 
\end{align}

% general forms for G_ij(y) and H_ij(x)

In light of Section~\ref{sec_211}, the general form of $G_{ij}(y)$  that satisfies~\eqref{eq_74a} is given by, for $0\leqslant (i,j)\leqslant (N_x,N_y-1)$,
\begin{equation}\label{eq_76}
G_{ij}(y) = \hat G_{ij}(y) + [\hat\alpha_{ij} - \hat G_{ij}(Y_j)]\phi_0(Y_j,Y_{j+1},y) + [\hat\alpha_{i,j+1}-\hat G_{ij}(Y_{j+1})]\phi_1(Y_j,Y_{j+1},y),
\ y\in[Y_j,Y_{j+1}],
\end{equation}
where $\hat G_{ij}(y)$ is an arbitrary (free) function.
Similarly, the general form of $H_{ij}(x)$ satisfying~\eqref{eq_75a} is given by, for $0\leqslant (i,j)\leqslant (N_x-1,N_y)$,
\begin{equation}\label{eq_77}
H_{ij}(x) = \hat H_{ij}(x) + [\hat\alpha_{ij} - \hat H_{ij}(X_i)]\phi_0(X_i,X_{i+1},x) + [\hat\alpha_{i+1,j}-\hat H_{ij}(X_{i+1})]\phi_1(X_i,X_{i+1},x),
\ x\in[X_i,X_{i+1}],
\end{equation}
where $\hat H_{ij}(x)$ is an arbitrary (free) function.

With the above settings, one can note that $u_{e_{ij}}(x,y)$ on $\Omega_{ij}$ satisfies the following conditions, for $ 0\leqslant (i,j)\leqslant (N_x-1,N_y-1)$,
\begin{equation}\left\{
\begin{split}
&
u_{e_{ij}}(X_i,y) = G_{ij}(y), \quad
u_{e_{ij}}(X_{i+1},y) = G_{i+1,j}(y), \quad y\in[Y_j,Y_{j+1}], \\
&
u_{e_{ij}}(x,Y_i) = H_{ij}(x), \quad u_{e_{ij}}(x,Y_{j+1}) = H_{i,j+1}(x), \quad x\in[X_i,X_{i+1}],
\end{split}
\right.
\end{equation}
where $G_{ij}(y)$ and $H_{ij}(x)$ are given by~\eqref{eq_76} and~\eqref{eq_77}, respectively.
In light of  Section~\ref{sec_212}, the general form of $u_{e_{ij}}(x,y)$ satisfying these conditions is given by,
\begin{equation}\label{eq_79}
u_{e_{ij}}(x,y) = g_{e_{ij}}(x,y)
- \mbs v_1^T(x) \mbs M_g \mbs v_2(y) + \mbs v_1^T(x)\mbs M_b \mbs v_2(y), \quad
0\leqslant (i,j)\leqslant (N_x-1,N_y-1),
\end{equation}
where $g_{e_{ij}}(x,y)$ is an arbitrary (free) function, and
\begin{align}
& \mbs M_g = \begin{bmatrix}
0 & g_{e_{ij}}(x,Y_j) & g_{e_{ij}}(x,Y_{j+1}) \\
g_{e_{ij}}(X_i,y) & -g_{e_{ij}}(X_i,Y_j) & -g_{e_{ij}}(X_i,Y_{j+1})\\
g_{e_{ij}}(X_{i+1},y) & -g_{e_{ij}}(X_{i+1},Y_j) & -g_{e_{ij}}(X_{i+1},Y_{j+1})
\end{bmatrix},
\quad \mbs v_1(x)=\begin{bmatrix}
1 \\ \phi_0(X_i,X_{i+1},x) \\ \phi_1(X_i,X_{i+1},x)
\end{bmatrix}, \\
& \mbs M_b = \begin{bmatrix}
0 & H_{ij}(x) & H_{i,j+1}(x) \\
G_{ij}(y) & -\hat\alpha_{ij} & -\hat\alpha_{i,j+1} \\
G_{i+1,j}(y) & -\hat\alpha_{i+1,j} & -\hat\alpha_{i+1,j+1}
\end{bmatrix},
\quad \mbs v_2(y)=\begin{bmatrix}
1 \\ \phi_0(Y_j,Y_{j+1},y) \\ \phi_1(Y_j,Y_{j+1},y)
\end{bmatrix}.
\end{align}

%One can verify that the function $u(\mbs x)$ defined by~\eqref{eq_69}, with $u_{e_{ij}}(\mbs x)$  given by~\eqref{eq_79},~\eqref{eq_76} and~\eqref{eq_77} for $0\leqslant (i,j)\leqslant (N_x-1,N_y-1)$, 
%for arbitrary functions $g_{e_{ij}}(\mbs x)$, $\hat G_{ij}(y)$ and $\hat H_{ij}(x)$, and arbitrary parameters $\hat\alpha_{ij}$ involved therein, satisfies the condition~\eqref{eq_71}.  
%Similarly, one can verify that any piece-wise function $u(\mbs x)$ defined by~\eqref{eq_69} that satisfies the condition~\eqref{eq_71}  
%can be expressed in  the form as given by~\eqref{eq_79},~\eqref{eq_76} and~\eqref{eq_77}.

\begin{theorem}\label{thm_6}
(i) The functions $u_{e_{ij}}(\mbs x)$ given by~\eqref{eq_79}, with $G_{ij}(y)$ and $H_{ij}(x)$ therein given by~\eqref{eq_76} and~\eqref{eq_77}, satisfy the conditions~\eqref{eq_71}, for arbitrary functions $g_{e_{ij}}(\mbs x)$, $\hat G_{ij}(y)$ and $\hat H_{ij}(x)$, and arbitrary values of the parameters $\hat\alpha_{ij}$ involved therein.
(ii) Any piece-wise function $u(\mbs x)$ defined by~\eqref{eq_69} that satisfies the conditions~\eqref{eq_71} 
can be expressed in  the form~\eqref{eq_79}, for some $g_{e_{ij}}(\mbs x)$, $\hat G_{ij}(y)$ and $\hat H_{ij}(x)$, and some value of the parameters $\hat\alpha_{ij}$ therein.
\end{theorem}
\begin{proof}
(i) The proof follows from a straightforward (albeit quite cumbersome) verification.
(ii) Suppose $u(\mbs x)$ is a piece-wise function that is defined by~\eqref{eq_69} and satisfies the conditions~\eqref{eq_71}. The $u_{e_{ij}}(\mbs x)$ functions in~\eqref{eq_69} can be written into the form~\eqref{eq_79}, by setting 
$g_{e_{ij}}(\mbs x) = u_{e_{ij}}(\mbs x)$, $\hat G_{ij}(y)=u(X_i,y)$ ($y\in[Y_i,Y_{j+1}]$), $\hat H_{ij}(x)=u(x,Y_j)$ ($x\in[X_i,X_{i+1}]$), and  $\hat\alpha_{ij}=u(X_i,Y_j)$ in~\eqref{eq_79},~\eqref{eq_76} and~\eqref{eq_77}.
\end{proof}

To arrive at a computational technique, we restrict each of the free functions $g_{e_{ij}}(\mbs x)$ 
 for $0\leqslant (i,j)\leqslant (N_x-1,N_y-1)$, $\hat G_{ij}(y)$ for $0\leqslant (i,j)\leqslant (N_x,N_y-1)$, and $\hat H_{ij}(x)$ for $0\leqslant (i,j)\leqslant (N_x-1,N_y)$ to a finite-dimensional function space.
% function space for g_eij
Define the function space
\begin{equation}\label{eq_82}
\mathcal{G}(\Omega_{ij},\mathcal M) = \text{span}\left\{\ 
\Psi_{ij,k}(\mbs x), \ \mbs x\in\Omega_{ij},\ 1\leqslant k\leqslant \mathcal M
\ \right\}, \quad 
\text{for}\ 0\leqslant (i,j)\leqslant (N_x-1,N_y-1)
\end{equation}
where $\mathcal M\in \mbb N$ is the dimension of $\mathcal G(\Omega_{ij},\mathcal M)$ and $\Psi_{ij,k}(x)$ ($0\leqslant k\leqslant \mathcal M$) are its bases.
We restrict the function $g_{e_{ij}}(\mbs x)$ to $\mathcal G(\Omega_{ij},\mathcal M)$ for $0\leqslant (i,j)\leqslant (N_x-1,N_y-1)$.
 We restrict $\hat G_{ij}(y)$ to the function space $\mathcal F_m([Y_j,Y_{j+1}])$ and $\hat H_{ij}(x)$ to the function space $\mathcal F_m([X_i,X_{i+1}])$, where $\mathcal F_m$ is defined in~\eqref{eq_61}. 
Let
\begin{subequations}\label{eq_83}
\begin{align}
& g_{e_{ij}}(\mbs x) = \sum_{k=1}^{\mathcal M} \hat g_{e_{ij},k} \Psi_{ij,k}(\mbs x), \quad \mbs x\in\Omega_{ij}, \quad 0\leqslant (i,j)\leqslant (N_x-1,N_y-1); \label{eq_83a} \\
& \hat G_{ij}(y) = \sum_{k=1}^m \mathscr{\hat G}_{ij,k} \Phi_{jk}(y), \quad y\in[Y_j,Y_{j+1}], \quad 0\leqslant (i,j)\leqslant (N_x,N_y-1); \label{eq_83b} \\
& \hat H_{ij}(x) = \sum_{k=1}^m \mathscr{\hat H}_{ij,k} \Phi_{ik}(x), \quad x\in[X_i,X_{i+1}], \quad 0\leqslant (i,j)\leqslant (N_x-1,N_y). \label{eq_83c}
\end{align}
\end{subequations}
Here $\hat g_{e_{ij},k}$, $\mathscr{\hat G}_{ij,k}$ and $\mathscr{\hat H}_{ij,k}$ are the expansion coefficients.

The equations~\eqref{eq_76} and~\eqref{eq_83b} characterize the functions $G_{ij}(y)$, which are formulated in terms of  the parameters $\hat\alpha_{ij}$ and $\mathscr{\hat G}_{ij,k}$. 
The equations~\eqref{eq_77} and~\eqref{eq_83c} characterize the functions $H_{ij}(x)$, which are formulated in terms of the parameters $\hat\alpha_{ij}$ and $\mathscr{\hat H}_{ij,k}$. Substitution of these parameterized forms of $G_{ij}(y)$ and $H_{ij}(x)$, together with equation~\eqref{eq_83a}, into equation~\eqref{eq_79} provide the parameterized form of $u_{e_{ij}}(\mbs x)$. In this parameterized form, $\hat\alpha_{ij}$ for $0\leqslant (i,j)\leqslant (N_x,N_y)$, $\hat g_{e_{ij},k}$ for $(0,0,1)\leqslant (i,j,k)\leqslant (N_x-1,N_y-1,\mathcal M)$, $\mathscr{\hat G}_{ij,k}$ for $(0,0,1)\leqslant (i,j,k)\leqslant(N_x,N_y-1,m)$, and $\mathscr{\hat H}_{ij,k}$ for $(0,0,1)\leqslant (i,j,k)\leqslant (N_x-1,N_y,m)$ are the unknown parameters to be determined.

\paragraph{$C^1$ FCEs}

Consider the following $C^1$ continuity conditions across the elements,
\begin{subequations}\label{eq_84}
\begin{align}
& u_{e_{ij}}(X_{i+1},y) = u_{e_{i+1,j}}(X_{i+1},y), \quad
y\in[Y_j,Y_{j+1}], \quad
0\leqslant (i,j)\leqslant (N_x-2,N_y-1); \\
& u_{e_{ij}}(x,Y_{j+1}) = u_{e_{i,j+1}}(x,Y_{j+1}), \quad
x\in [X_i,X_{i+1}], \quad
0\leqslant (i,j)\leqslant (N_x-1,N_y-2); \\
& \left.\frac{\partial u_{e_{ij}}}{\partial x} \right|_{(X_{i+1},y)} = \left.\frac{\partial u_{e_{i+1,j}}}{\partial x} \right|_{(X_{i+1},y)}, \quad y\in[Y_j,Y_{j+1}], \quad 0\leqslant (i,j)\leqslant (N_x-2,N_y-1); \\
& \left.\frac{\partial u_{e_{ij}}}{\partial y} \right|_{(x,Y_{j+1})} = \left.\frac{\partial u_{e_{i,j+1}}}{\partial y} \right|_{(x,Y_{j+1})}, \quad x\in[X_i,X_{i+1}], \quad 0\leqslant (i,j)\leqslant (N_x-1,N_y-2).
\end{align}
\end{subequations}
We are interested in the general form of $u(\mbs x)$ defined in~\eqref{eq_69} that satisfies the conditions~\eqref{eq_84}.
%and~\eqref{eq_70}.

Let
\begin{equation}\left\{
\begin{split}
& u(X_i,Y_j) = \hat \alpha_{ij}, \quad
 \left.\frac{\partial u}{\partial x}\right|_{(X_i,Y_j)} = \hat\beta_{ij}^{(1)}, \quad 
\left.\frac{\partial u}{\partial y}\right|_{(X_i,Y_j)} = \hat\beta_{ij}^{(2)}, \\
& \left.\frac{\partial^2 u}{\partial x\partial y}\right|_{(X_i,Y_j)} = \hat\gamma_{ij}, \quad
 0\leqslant (i,j)\leqslant (N_x,N_y), 
\end{split} \right.
\end{equation}
where $\hat\alpha_{ij}$, $\hat\beta^{(1)}_{ij}$, $\hat\beta^{(2)}_{ij}$ and $\hat\gamma_{ij}$ are  parameters to be determined. 

%%%%%%%%%%%%%%%%%%%%%%
\begin{comment}
and 
\begin{equation}\left\{
\begin{split}
& \hat\alpha_{0j} = C(a_1,Y_{j}), \quad
\hat\alpha_{N_x,j} = C(b_1,Y_j), \quad
0\leqslant j\leqslant N_y; \\
& \hat\alpha_{i0} = C(X_i,a_2), \quad
\hat\alpha_{i,N_y} = C(X_i,b_2), \quad 0\leqslant i\leqslant N_x; \\
& \hat\beta^{(1)}_{i0} = C_x(X_i,a_2), \quad \hat\beta^{(1)}_{iN_y} = C_x(X_i,b_2), \quad 0\leqslant i\leqslant N_x;
\\
& \hat\beta_{0j}^{(2)} = C_y(a_1,Y_j), \quad \hat\beta^{(2)}_{N_x,j} = C_y(b_1,Y_j), \quad 0\leqslant j\leqslant N_y.
\end{split} \right.
\end{equation}
Here $C_x(x,y)$ and $C_y(x,y)$ denote the partial derivatives of $C(x,y)$ with respect to $x$ and $y$, respectively.
\end{comment}
%%%%%%%%%%%%%%%%%%%

In light of the continuity conditions~\eqref{eq_84}, let
\begin{subequations}
\begin{align}
& u(X_i,y) = G_{ij}^{(1)}(y), \quad y\in[Y_j,Y_{j+1}], \quad 0\leqslant (i,j)\leqslant (N_x,N_y-1); \\
& \left.\frac{\partial u}{\partial x}\right|_{(X_i,y)} = G^{(2)}_{ij}(y), \quad y\in[Y_j,Y_{j+1}], \quad 0\leqslant (i,j)\leqslant (N_x,N_y-1); \\
& u(x,Y_j) = H^{(1)}_{ij}(x), \quad x\in[X_i,X_{i+1}], \quad 0\leqslant(i,j)\leqslant (N_x-1,N_y); \\
& \left.\frac{\partial u}{\partial y}\right|_{(x,Y_j)} = H^{(2)}_{ij}(x), \quad x\in[X_i,X_{i+1}], \quad 0\leqslant (i,j)\leqslant (N_x-1,N_y).
\end{align}
\end{subequations}
Here $G^{(1)}_{ij}(y)$, $G^{(2)}_{ij}(y)$, $H^{(1)}_{ij}(y)$, and $G^{(2)}_{ij}(y)$ are functions to be determined. In addition, $G^{(1)}_{ij}(y)$ satisfies the following conditions,
\begin{subequations}\label{eq_88}
\begin{align}
& G^{(1)}_{ij}(Y_j) = \hat\alpha_{ij},
\quad G^{(1)}_{ij}(Y_{j+1}) = \hat\alpha_{i,j+1},
\quad 0\leqslant(i,j)\leqslant(N_x,N_y-1); \label{eq_88a} \\
%& G^{(1)}_{0j}(y) = C(a_1,y), \quad G^{(1)}_{N_x,j}(y) = C(b_1,y), \quad y\in[Y_j,Y_{j+1}], \quad 0\leqslant j\leqslant N_y-1; \label{eq_88b} \\
& \left.\frac{dG^{(1)}_{ij}}{dy}\right|_{Y_j} = \hat\beta^{(2)}_{ij}, \quad \left.\frac{dG^{(1)}_{ij}}{dy}\right|_{Y_{j+1}} = \hat\beta^{(2)}_{i,j+1}, \quad 0\leqslant(i,j)\leqslant (N_x,N_y-1). \label{eq_88c}
\end{align}
\end{subequations}
$G^{(2)}_{ij}(y)$ satisfies the following conditions,
\begin{subequations}\label{eq_89}
\begin{align}
& G^{(2)}_{ij}(Y_j) = \hat\beta^{(1)}_{ij}, \quad G^{(2)}_{ij}(Y_{j+1}) = \hat\beta^{(1)}_{i,j+1}, \quad 0\leqslant (i,j)\leqslant(N_x,N_y-1); \\
& \left.\frac{dG^{(2)}_{ij}}{dy}\right|_{Y_j} = \hat\gamma_{ij}, \quad \left.\frac{dG^{(2)}_{ij}}{dy}\right|_{Y_{j+1}} = \hat\gamma_{i,j+1}, \quad 0\leqslant (i,j)\leqslant(N_x,N_y-1).
\end{align}
\end{subequations}
$H^{(1)}_{ij}(x)$ satisfies the following conditions,
\begin{subequations}\label{eq_90}
\begin{align}
& H^{(1)}_{ij}(X_i) = \hat\alpha_{ij}, \quad H^{(1)}_{ij}(X_{i+1}) = \hat\alpha_{i+1,j}, \quad 0\leqslant(i,j)\leqslant (N_x-1,N_y); \label{eq_90a} \\
%& H^{(1)}_{i0}(x) = C(x,a_2), \quad H^{(1)}_{i,N_y}(x) = C(x,b_2), \quad x\in[X_i,X_{i+1}], \quad 0\leqslant i\leqslant N_x-1; \label{eq_90b} \\
& \left.\frac{dH_{ij}^{(1)}}{dx}\right|_{X_i} = \hat\beta^{(1)}_{ij}, \quad \left.\frac{dH_{ij}^{(1)}}{dx}\right|_{X_{i+1}} = \hat\beta^{(1)}_{i+1,j}, \quad 0\leqslant (i,j)\leqslant (N_x-1,N_y). \label{eq_90c}
\end{align}
\end{subequations}
$H^{(2)}_{ij}(x)$ satisfies the following conditions,
\begin{subequations}\label{eq_91}
\begin{align}
& H^{(2)}_{ij}(X_i) = \hat\beta^{(2)}_{ij}, \quad H^{(2)}_{ij}(X_{i+1}) = \hat\beta^{(2)}_{i+1,j}, \quad 0\leqslant(i,j)\leqslant(N_x-1,N_y); \\
& \left.\frac{dH^{(2)}_{ij}}{dx} \right|_{X_i} = \hat\gamma_{ij}, \quad \left.\frac{dH^{(2)}_{ij}}{dx} \right|_{X_{i+1}} = \hat\gamma_{i+1,j}, \quad 0\leqslant(i,j)\leqslant(N_x-1,N_y).
\end{align}
\end{subequations}

% general forms for G_ij and H_ij

In light of the discussions in Sections~\ref{sec_211} and~\ref{sec_231}, the general form for $G^{(1)}_{ij}(y)$ that satisfies~\eqref{eq_88a} and~\eqref{eq_88c} is given by
\begin{equation}\label{eq_92}
\begin{split}
 G^{(1)}_{ij}(y) =& \hat G^{(1)}_{ij}(y) + \left[\hat\alpha_{ij} - \hat G_{ij}^{(1)}(Y_j)\right]\psi_0(Y_i,Y_{j+1},y) +
\left[\hat\alpha_{i,j+1} - \hat G^{(1)}_{ij}(Y_{j+1})\right]\psi_1(Y_j,Y_{j+1},y) \\
&+ \left[\hat\beta^{(2)}_{ij} - \hat G^{(1)}_{ij,y}(X_i)\right]\varphi_0(Y_j,Y_{j+1},y)
+ \left[\beta^{(2)}_{i,j+1} - \hat G^{(1)}_{ij,y}(Y_{j+1})\right]\varphi_1(Y_j,Y_{j+1},y), \\
& \text{for}\ (1,0)\leqslant(i,j)\leqslant (N_x-1,N_y-1),
\end{split}
\end{equation}
where $\hat G^{(1)}_{ij}(y)$ is an arbitrary (free) function, $\hat G^{(1)}_{ij,y}$ denotes its derivative, and $\psi_0$, $\psi_1$, $\varphi_0$ and $\varphi_1$ are defined in~\eqref{eq_67}.
The general form of $G^{(2)}_{ij}(y)$ that satisfies~\eqref{eq_89} is given by,
\begin{equation}\label{eq_93}
\begin{split}
 G^{(2)}_{ij}(y) =& \hat G^{(2)}_{ij}(y) + \left[\hat\beta^{(1)}_{ij} - \hat G_{ij}^{(2)}(Y_j)\right]\psi_0(Y_i,Y_{j+1},y) +
\left[\hat\beta^{(1)}_{i,j+1} - \hat G^{(2)}_{ij}(Y_{j+1})\right]\psi_1(Y_j,Y_{j+1},y) \\
&+ \left[\hat\gamma_{ij} - \hat G^{(2)}_{ij,y}(X_i)\right]\varphi_0(Y_j,Y_{j+1},y)
+ \left[\gamma_{i,j+1} - \hat G^{(2)}_{ij,y}(Y_{j+1})\right]\varphi_1(Y_j,Y_{j+1},y), \\
& \text{for}\ 0\leqslant(i,j)\leqslant (N_x,N_y-1),
\end{split}
\end{equation}
where $\hat G^{(2)}_{ij}(y)$ is an  arbitrary (free) function, and $\hat G^{(2)}_{ij,y}$ denotes its derivative.
The general form of $H^{(1)}_{ij}(x)$ that satisfies~\eqref{eq_90a} and~\eqref{eq_90c} is given by
\begin{equation}\label{eq_94}
\begin{split}
H^{(1)}_{ij}(x) =& \hat H^{(1)}_{ij}(x) + \left[\hat\alpha_{ij} - \hat H^{(1)}_{ij}(X_i) \right]\psi_0(X_i,X_{i+1},x)
+ \left[\hat\alpha_{i+1,j} - \hat H^{(1)}_{ij}(X_{i+1}) \right]\psi_1(X_i,X_{i+1},x) \\
&+ \left[\hat\beta^{(1)}_{ij} - \hat H^{(1)}_{ij,x}(X_i) \right]\varphi_0(X_i,X_{i+1},x)
+ \left[\hat\beta^{(1)}_{i+1,j} - \hat H^{(1)}_{ij,x}(X_{i+1}) \right]\varphi_1(X_i,X_{i+1},x), \\
& \text{for}\ (0,1)\leqslant(i,j)\leqslant(N_x-1,N_y-1),
\end{split}
\end{equation}
where $\hat H^{(1)}_{ij}(x)$ is an arbitrary (free) function and $\hat H^{(1)}_{ij,x}$ denotes its derivative.
The general form of $\hat H^{(2)}_{ij}(x)$ that satisfies~\eqref{eq_91} is given by
\begin{equation}\label{eq_95}
\begin{split}
H^{(2)}_{ij}(x) =\ & \hat H^{(2)}_{ij}(x) + \left[\hat\beta^{(2)}_{ij} - \hat H^{(2)}_{ij}(X_i) \right]\psi_0(X_i,X_{i+1},x) 
+ \left[\hat\beta^{(2)}_{i+1,j} - \hat H^{(2)}_{ij}(X_{i+1}) \right]\psi_1(X_i,X_{i+1},x) \\
&+ \left[\hat\gamma_{ij} - \hat H^{(2)}_{ij,x}(X_i) \right]\varphi_0(X_i,X_{i+1},x)
+ \left[\hat\gamma_{i+1,j}-\hat H^{(2)}_{ij,x}(X_{i+1}) \right]\varphi_1(X_i,X_{i+1},x), \\
& \text{for}\ 0\leqslant(i,j)\leqslant(N_x-1,N_y),
\end{split}
\end{equation}
where $\hat H^{(2)}_{ij}(x)$ is an arbitrary (free) function and $\hat H^{(2)}_{ij,x}$ denotes its derivative.

With the above settings, on  element $\Omega_{ij}$ the function $u_{e_{ij}}(\mbs x)$ satisfies the following conditions, for $0\leqslant(i,j)\leqslant(N_x-1,N_y-1)$,
\begin{equation}\label{eq_96}
\left\{
\begin{split}
& u_{e_{ij}}(X_i,y) = G^{(1)}_{ij}(y), \quad u_{e_{ij}}(X_{i+1},y) = G^{(1)}_{i+1,j}(y), \quad y\in[Y_j,Y_{j+1}]; \\
& u_{e_{ij}}(x,Y_j) = H^{(1)}_{ij}(x), \quad u_{e_{ij}}(x,Y_{j+1}) = H^{(1)}_{i,j+1}(x), \quad x\in[X_i,X_{i+1}]; \\
& \left.\frac{\partial u_{e_{ij}}}{\partial x}\right|_{(X_i,y)} = G^{(2)}_{ij}(y), \quad 
\left.\frac{\partial u_{e_{ij}}}{\partial x}\right|_{(X_{i+1},y)} = G^{(2)}_{i+1,j}(y) \quad y\in[Y_j,Y_{j+1}]; \\
& \left.\frac{\partial u_{e_{ij}}}{\partial y}\right|_{(x,Y_j)} = H^{(2)}_{ij}(x), \quad 
\left.\frac{\partial u_{e_{ij}}}{\partial y}\right|_{(x,Y_{j+1})} = H^{(2)}_{i,j+1}(x), \quad x\in[X_i,X_{i+1}],
\end{split} \right.
\end{equation}
where $G^{(1)}_{ij}$, $G^{(2)}_{ij}$, $H^{(1)}_{ij}$ and $H^{(2)}_{ij}$ are given by equations~\eqref{eq_92}--\eqref{eq_95}.
In light of the discussions in Section~\ref{sec_212}, the general form for $u_{e_{ij}}(\mbs x)$ satisfying~\eqref{eq_96} is given by,
\begin{equation}\label{eq_97}
u_{e_{ij}}(x,y) = g_{e_{ij}}(x,y) - \mbs w_1^T(x)\mbs Q_g \mbs w_2(y)
+ \mbs w_1^T(x)\mbs Q_b \mbs w_2(y),
\quad 0\leqslant (i,j)\leqslant (N_x-1,N_y-1),
\end{equation}
where $g_{e_{ij}}(\mbs x)$ is an arbitrary (free) function, and
\begin{subequations}\label{eq_98}
\begin{align}
& \mbs w_1(x)=\begin{bmatrix}
1 \\ \psi_0(X_i,X_{i+1},x) \\ \psi_1(X_i,X_{i+1},x) \\ \varphi_0(X_i,X_{i+1},x) \\ \varphi_1(X_i,X_{i+1},x)
\end{bmatrix}, \quad 
\mbs w_2(y)=\begin{bmatrix}
1 \\ \psi_0(Y_j,Y_{j+1},y) \\ \psi_1(Y_j,Y_{j+1},y) \\ \varphi_0(Y_j,Y_{j+1},y) \\ \varphi_1(Y_j,Y_{j+1},y)
\end{bmatrix}, \\
& \mbs Q_g = \begin{bmatrix}
0 & g_{e_{ij}}(x,Y_j) & g_{e_{ij}}(x,Y_{j+1}) & g_{e_{ij},y}(x,Y_j) & g_{e_{ij},y}(x,Y_{j+1}) \\
g_{e_{ij}}(X_i,y) & -g_{e_{ij}}(X_i,Y_j) & -g_{e_{ij}}(X_i,Y_{j+1}) & -g_{e_{ij},y}(X_i,Y_j) & -g_{e_{ij},y}(X_i,Y_{j+1})\\
g_{e_{ij}}(X_{i+1},y) & -g_{e_{ij}}(X_{i+1},Y_j) & -g_{e_{ij}}(X_{i+1},Y_{j+1}) & -g_{e_{ij},y}(X_{i+1},Y_j) & -g_{e_{ij},y}(X_{i+1},Y_{j+1}) \\
g_{e_{ij},x}(X_i,y) & -g_{e_{ij},x}(X_i,Y_j) & -g_{e_{ij},x}(X_i,Y_{j+1}) & -g_{e_{ij},xy}(X_i,Y_j) & -g_{e_{ij},xy}(X_i,Y_{j+1}) \\
g_{e_{ij},x}(X_{i+1},y) & -g_{e_{ij},x}(X_{i+1},Y_j) & -g_{e_{ij},x}(X_{i+1},Y_{j+1}) & -g_{e_{ij},xy}(X_{i+1},Y_j) & -g_{e_{ij},xy}(X_{i+1},Y_{j+1})
\end{bmatrix},
\quad  \\
& \mbs Q_b = \begin{bmatrix}
0 & H^{(1)}_{ij}(x) & H^{(1)}_{i,j+1}(x) & H^{(2)}_{ij}(x) & H^{(2)}_{i,j+1}(x) \\
G^{(1)}_{ij}(y) & -\hat\alpha_{ij} & -\hat\alpha_{i,j+1} & -\hat\beta^{(2)}_{ij} & -\hat\beta^{(2)}_{i,j+1} \\
G^{(1)}_{i+1,j}(y) & -\hat\alpha_{i+1,j} & -\hat\alpha_{i+1,j+1} & -\hat\beta^{(2)}_{i+1,j} & -\hat\beta^{(2)}_{i+1,j+1} \\
G^{(2)}_{ij}(y) & -\hat\beta^{(1)}_{ij} & - \hat\beta^{(1)}_{i,j+1} & - \hat\gamma_{ij} & -\hat\gamma_{i,j+1} \\
G^{(2)}_{i+1,j}(y) & -\hat\beta^{(1)}_{i+1,j} & -\hat\beta^{(1)}_{i+1,j+1} & -\hat\gamma_{i+1,j} & -\hat\gamma_{i+1,j+1}
\end{bmatrix}.
\end{align}
\end{subequations}
Here, $\psi_0$, $\psi_1$, $\varphi_0$ and $\varphi_1$ are defined in~\eqref{eq_67}. $g_{e_{ij},x}=\frac{\partial g_{e_{ij}}}{\partial x}$, $g_{e_{ij},y}=\frac{\partial g_{e_{ij}}}{\partial y}$, and $g_{e_{ij},xy}=\frac{\partial^2 g_{e_{ij}}}{\partial x\partial y}$.
The following result holds.

%One can verify that the function $u(\mbs x)$ defined in~\eqref{eq_69}, with $u_{e_{ij}}(\mbs x)$ given by~\eqref{eq_97} and by equations~\eqref{eq_98},~\eqref{eq_92}--\eqref{eq_95}, satisfies the conditions~\eqref{eq_84},
%for arbitrary functions $g_{e_{ij}}(\mbs x)$, $\hat G^{(1)}_{ij}(y)$, $G^{(2)}_{ij}(y)$, $H^{(1)}_{ij}(x)$ and $H^{(2)}_{ij}(x)$ and arbitrary parameter values $\hat\alpha_{ij}$, $\hat\beta^{(1)}_{ij}$, $\hat\beta^{(2)}_{ij}$ and $\hat\gamma_{ij}$ involved therein.
%One can also verify that any function $u(\mbs x)$ defined by~\eqref{eq_69} that satisfies the condition~\eqref{eq_84}
%can be expressed in the form as given by~\eqref{eq_97},~\eqref{eq_98} and~\eqref{eq_92}--\eqref{eq_95}.

\begin{theorem}\label{thm_7}
(i) The functions $u_{e_{ij}}(\mbs x)$ given by~\eqref{eq_97}, together with~\eqref{eq_98} and~\eqref{eq_92}--\eqref{eq_95}, satisfy the conditions~\eqref{eq_84},
for arbitrary functions $g_{e_{ij}}(\mbs x)$, $\hat G^{(1)}_{ij}(y)$, $\hat G^{(2)}_{ij}(y)$, $\hat H^{(1)}_{ij}(x)$ and $\hat H^{(2)}_{ij}(x)$, and for arbitrary parameter values $\hat\alpha_{ij}$, $\hat\beta^{(1)}_{ij}$, $\hat\beta^{(2)}_{ij}$ and $\hat\gamma_{ij}$ involved therein.
(ii) Any function $u(\mbs x)$ defined by~\eqref{eq_69} satisfying the conditions~\eqref{eq_84}
can be expressed in the form as given by~\eqref{eq_97},~\eqref{eq_98} and~\eqref{eq_92}--\eqref{eq_95}, for some $g_{e_{ij}}(\mbs x)$, $\hat G^{(1)}_{ij}(y)$, $\hat G^{(2)}_{ij}(y)$, $\hat H^{(1)}_{ij}(x)$ and $\hat H^{(2)}_{ij}(x)$, and  some value of the parameters $\hat\alpha_{ij}$, $\hat\beta^{(1)}_{ij}$, $\hat\beta^{(2)}_{ij}$ and $\hat\gamma_{ij}$ therein.
\end{theorem}
\begin{proof}
(i) By straightforward (albeit cumbersome) verification.
(ii) Suppose $u(\mbs x)$ is a function defined by~\eqref{eq_69} that satisfies~\eqref{eq_84}. The functions $u_{e_{ij}}(\mbs x)$ in $u(\mbs x)$ can be written in the form of~\eqref{eq_97} by setting $g_{e_{ij}}(\mbs x)=u_{e_{ij}}(\mbs x)$, $\hat G_{ij}^{(1)}(y)=u(X_i,y)$ ($y\in[Y_j,Y_{j+1}]$), $\hat G_{ij}^{(2)}(y)=\left.\frac{\partial u}{\partial x}\right|_{(X_i,y)}$ ($y\in[Y_j,Y_{j+1}]$), $\hat H_{ij}^{(1)}(x)=u(x,Y_j)$ ($x\in[X_i,X_{i+1}]$), $\hat H_{ij}^{(2)}(x)=\left.\frac{\partial u}{\partial y} \right|_{(x,Y_j)}$ ($x\in[X_i,X_{i+1}]$), $\hat\alpha_{ij}=u(X_i,Y_j)$, $\hat\beta_{ij}^{(1)}=\left.\frac{\partial u}{\partial x} \right|_{(X_i,Y_j)}$, $\hat\beta_{ij}^{(2)}=\left.\frac{\partial u}{\partial y} \right|_{(X_i,Y_j)}$, and $\hat\gamma_{ij}=\left.\frac{\partial^2 u}{\partial x\partial y} \right|_{(X_i,Y_j)}$ in~\eqref{eq_97},~\eqref{eq_98}, and~\eqref{eq_92}--\eqref{eq_95}.
\end{proof}

% restrict free functions to finite dimensional space

To derive a computational technique, we restrict the free functions $g_{e_{ij}}(\mbs x)$, for $0\leqslant(i,j)\leqslant(N_x-1,N_y-1)$, to the function space $\mathcal{G}(\Omega_{ij},\mathcal M)$ defined in~\eqref{eq_82}. We restrict the free functions $\hat G^{(1)}_{ij}(y)$ (for $0\leqslant(i,j)\leqslant(N_x,N_y-1)$) and $\hat G^{(2)}_{ij}(y)$ (for $0\leqslant(i,j)\leqslant(N_x,N_y-1)$) to the function space $\mathcal F_m([Y_j,Y_{j+1}])$, where $\mathcal F_m$ is defined in~\eqref{eq_61}.
We restrict the free functions $\hat H^{(1)}_{ij}(x)$ (for $0\leqslant(i,j)\leqslant(N_x-1,N_y)$) and $\hat H^{(2)}_{ij}(x)$ (for $0\leqslant(i,j)\leqslant(N_x-1,N_y)$) to the function space $\mathcal F_m([X_i,X_{i+1}])$. Let
\begin{subequations}\label{eq_99}
\begin{align}
& g_{e_{ij}}(\mbs x) = \sum_{k=1}^{\mathcal M} \hat g_{e_{ij},k} \Psi_{ij,k}(\mbs x), \quad \mbs x\in\Omega_{ij}, \quad 0\leqslant (i,j)\leqslant (N_x-1,N_y-1); \label{eq_99a} \\
& \hat G^{(1)}_{ij}(y) = \sum_{k=1}^m \mathscr{\hat G}^{(1)}_{ij,k} \Phi_{jk}(y), \quad y\in[Y_j,Y_{j+1}], \quad 0\leqslant (i,j)\leqslant (N_x,N_y-1); \label{eq_99b} \\
& \hat G^{(2)}_{ij}(y) = \sum_{k=1}^m \mathscr{\hat G}^{(2)}_{ij,k} \Phi_{jk}(y), \quad y\in[Y_j,Y_{j+1}], \quad 0\leqslant (i,j)\leqslant (N_x,N_y-1); \label{eq_99c} \\
& \hat H^{(1)}_{ij}(x) = \sum_{k=1}^m \mathscr{\hat H}^{(1)}_{ij,k} \Phi_{ik}(x), \quad x\in[X_i,X_{i+1}], \quad 0\leqslant (i,j)\leqslant (N_x-1,N_y); \label{eq_99d} \\
& \hat H^{(2)}_{ij}(x) = \sum_{k=1}^m \mathscr{\hat H}^{(2)}_{ij,k} \Phi_{ik}(x), \quad x\in[X_i,X_{i+1}], \quad 0\leqslant (i,j)\leqslant (N_x-1,N_y). \label{eq_99e}
\end{align}
\end{subequations}
Here $\hat g_{e_{ij},k}$, $\mathscr{\hat G}^{(1)}_{ij,k}$, $\mathscr{\hat G}^{(2)}_{ij,k}$, $\mathscr{\hat H}^{(1)}_{ij,k}$, and $\mathscr{\hat H}^{(2)}_{ij,k}$ are the expansion coefficients. Substitution of the expressions~\eqref{eq_99a}--\eqref{eq_99e} into~\eqref{eq_97} leads to the parameterized form of $u_{e_{ij}}(\mbs x)$. In this parameterized form, the unknown coefficients (parameters) to be determined include $\hat g_{e_{ij},k}$ (for $(0,0,1)\leqslant(i,j,k)\leqslant(N_x-1,N_y-1,\mathcal M)$), $\mathscr{\hat G}^{(1)}_{ij,k}$ (for $(0,0,1)\leqslant(i,j,k)\leqslant(N_x,N_y-1,m)$), $\mathscr{\hat G}^{(2)}_{ij,k}$ (for $(0,0,1)\leqslant(i,j,k)\leqslant(N_x,N_y-1,m)$), $\mathscr{\hat H}^{(1)}_{ij,k}$ (for $(0,0,1)\leqslant(i,j,k)\leqslant(N_x-1,N_y,m)$), $\mathscr{\hat H}^{(2)}_{ij,k}$ (for $(0,0,1)\leqslant(i,j,k)\leqslant(N_x-1,N_y,m)$), $\hat\alpha_{ij}$ (for $0\leqslant(i,j)\leqslant(N_x,N_y)$), $\hat\beta^{(1)}_{ij}$ (for $0\leqslant(i,j)\leqslant(N_x,N_y)$), $\hat\beta^{(2)}_{ij}$ (for $0\leqslant(i,j)\leqslant(N_x,N_y)$), and $\hat\gamma_{ij}$ (for $0\leqslant (i,j)\leqslant(N_x,N_y)$).

\paragraph{Mixed FCEs} 

We next consider the case of enforcing $C^0$ continuity  across the elements in one direction and $C^1$ continuity  in the other direction. We refer to elements with such continuity properties as mixed FCEs. 

To be more concrete, let us consider $C^1$ continuity along  $x$  and $C^0$ continuity along  $y$, 
\begin{subequations}\label{eq_a97}
\begin{align}
& u_{e_{ij}}(X_{i+1},y) = u_{e_{i+1,j}}(X_{i+1},y), \quad
y\in[Y_j,Y_{j+1}], \quad
0\leqslant (i,j)\leqslant (N_x-2,N_y-1); \\
& u_{e_{ij}}(x,Y_{j+1}) = u_{e_{i,j+1}}(x,Y_{j+1}), \quad
x\in [X_i,X_{i+1}], \quad
0\leqslant (i,j)\leqslant (N_x-1,N_y-2); \\
& \left.\frac{\partial u_{e_{ij}}}{\partial x} \right|_{(X_{i+1},y)} = \left.\frac{\partial u_{e_{i+1,j}}}{\partial x} \right|_{(X_{i+1},y)}, \quad y\in[Y_j,Y_{j+1}], \quad 0\leqslant (i,j)\leqslant (N_x-2,N_y-1).
\end{align}
\end{subequations}
Let
\begin{align}
u(X_i,Y_j) = \hat\alpha_{ij}, \quad
\left.\frac{\partial u}{\partial x}\right|_{(X_i,Y_j)} = \hat\beta_{ij}, \quad 0\leqslant(i,j)\leqslant (N_x,N_y),
\end{align}
where $\hat\alpha_{ij}$ and $\hat\beta_{ij}$ are  parameters to be determined.

In light of the continuity conditions in~\eqref{eq_a97}, let
\begin{subequations}
\begin{align}
& u(X_i,y) = G_{ij}^{(1)}(y), \quad y\in[Y_j,Y_{j+1}], \quad 0\leqslant (i,j)\leqslant (N_x,N_y-1); \\
& \left.\frac{\partial u}{\partial x}\right|_{(X_i,y)} = G^{(2)}_{ij}(y), \quad y\in[Y_j,Y_{j+1}], \quad 0\leqslant (i,j)\leqslant (N_x,N_y-1); \\
& u(x,Y_j) = H_{ij}(x), \quad x\in[X_i,X_{i+1}], \quad 0\leqslant(i,j)\leqslant (N_x-1,N_y),
\end{align}
\end{subequations}
where $G^{(1)}_{ij}(y)$, $G^{(2)}_{ij}(y)$ and $H_{ij}(x)$ are functions to be determined. In addition, $G^{(1)}_{ij}(y)$ satisfies the following conditions,
\begin{align}\label{eq_b100}
G^{(1)}_{ij}(Y_j) = \hat\alpha_{ij}, \quad G^{(1)}_{ij}(Y_{j+1}) = \hat\alpha_{i,j+1}, \quad 0\leqslant(i,j)\leqslant(N_x,N_y-1).
\end{align}
$G^{(2)}_{ij}(y)$ satisfies the following conditions,
\begin{align}\label{eq_b101}
G^{(2)}_{ij}(Y_j) = \hat\beta_{ij}, \quad G^{(2)}_{ij}(Y_{j+1}) = \hat\beta_{i,j+1}, \quad 0\leqslant(i,j)\leqslant(N_x,N_y-1).
\end{align}
$H_{ij}(x)$ satisfies the following conditions,
\begin{subequations}\label{eq_b102}
\begin{align}
& H_{ij}(X_i) = \hat\alpha_{ij}, \quad H_{ij}(X_{i+1}) = \hat\alpha_{i+1,j}, \quad 0\leqslant(i,j)\leqslant (N_x-1,N_y); \label{eq_b102a} \\
& \left.\frac{dH_{ij}}{dx}\right|_{X_i} = \hat\beta_{ij}, \quad \left.\frac{dH_{ij}}{dx}\right|_{X_{i+1}} = \hat\beta_{i+1,j}, \quad 0\leqslant (i,j)\leqslant (N_x-1,N_y). \label{eq_b102b}
\end{align}
\end{subequations}

The general form of $G^{(1)}_{ij}(y)$ satisfying~\eqref{eq_b100} is given by,
\begin{equation}\label{eq_b103}
\begin{split}
G^{(1)}_{ij}(y) =\ & \hat G_{ij}^{(1)}(y) + \left[\hat\alpha_{ij}-\hat G^{(1)}_{ij}(Y_j) \right]\phi_0(Y_j,Y_{j+1},y) + \left[\hat\alpha_{i,j+1}-\hat G^{(1)}_{ij}(Y_{j+1}) \right]\phi_1(Y_j,Y_{j+1},y), \\
& \text{for}\ 0\leqslant(i,j)\leqslant(N_x,N_y-1),
\end{split}
\end{equation}
where $\hat G^{(1)}_{ij}(y)$ is an arbitrary (free) function, and $\phi_0$ and $\phi_1$ are defined in~\eqref{eq_a60}.
The general form of $G^{(2)}_{ij}(y)$ satisfying~\eqref{eq_b101} is given by
\begin{equation}
\begin{split}
G^{(2)}_{ij}(y) =\ & \hat G_{ij}^{(2)}(y) + \left[\hat\beta_{ij}-\hat G^{(2)}_{ij}(Y_j) \right]\phi_0(Y_j,Y_{j+1},y) + \left[\hat\beta_{i,j+1}-\hat G^{(2)}_{ij}(Y_{j+1}) \right]\phi_1(Y_j,Y_{j+1},y), \\
& \text{for}\ 0\leqslant(i,j)\leqslant(N_x,N_y-1),
\end{split}
\end{equation}
where $\hat G^{(2)}_{ij}(y)$ is an arbitrary (free) function. 
The general form of $H_{ij}(x)$ satisfying~\eqref{eq_b102} is given by,
\begin{equation}\label{eq_b105}
\begin{split}
H_{ij}(x) =\ & \hat H_{ij}(x) + \left[\hat\alpha_{ij} - \hat H_{ij}(X_i) \right]\psi_0(X_i,X_{i+1},x)
+ \left[\hat\alpha_{i+1,j} - \hat H_{ij}(X_{i+1}) \right]\psi_1(X_i,X_{i+1},x) \\
&+ \left[\hat\beta_{ij} - \hat H_{ij,x}(X_i) \right]\varphi_0(X_i,X_{i+1},x)
+ \left[\hat\beta_{i+1,j} - \hat H_{ij,x}(X_{i+1}) \right]\varphi_1(X_i,X_{i+1},x), \\
& \text{for}\ 0\leqslant(i,j)\leqslant(N_x-1,N_y),
\end{split}
\end{equation}
where $\hat H_{ij}(x)$ is an arbitrary (free) function.

With the above settings, on $\Omega_{ij}$ the function $u_{e_{ij}}(\mbs x)$ satisfies the following conditions, for $0\leqslant(i,j)\leqslant(N_x-1,N_y-1)$,
\begin{equation}\label{eq_b106}
\left\{
\begin{split}
& u_{e_{ij}}(X_i,y) = G^{(1)}_{ij}(y), \quad u_{e_{ij}}(X_{i+1},y) = G^{(1)}_{i+1,j}(y), \quad y\in[Y_j,Y_{j+1}]; \\
& u_{e_{ij}}(x,Y_j) = H_{ij}(x), \quad u_{e_{ij}}(x,Y_{j+1}) = H_{i,j+1}(x), \quad x\in[X_i,X_{i+1}]; \\
& \left.\frac{\partial u_{e_{ij}}}{\partial x}\right|_{(X_i,y)} = G^{(2)}_{ij}(y), \quad 
\left.\frac{\partial u_{e_{ij}}}{\partial x}\right|_{(X_{i+1},y)} = G^{(2)}_{i+1,j}(y) \quad y\in[Y_j,Y_{j+1}],
\end{split} \right.
\end{equation}
where $G^{(1)}_{ij}(y)$, $G^{(2)}_{ij}(y)$ and $H_{ij}(y)$ are given by~\eqref{eq_b103}--\eqref{eq_b105}.
The general form of $u_{e_{ij}}(\mbs x)$ satisfying~\eqref{eq_b106} is given by,
\begin{equation}\label{eq_b107}
u_{e_{ij}}(x,y) = g_{e_{ij}}(x,y) - \mbs z_1^T(x)\mbs R_g \mbs z_2(y)
+ \mbs z_1^T(x)\mbs R_b \mbs z_2(y),
\quad 0\leqslant (i,j)\leqslant (N_x-1,N_y-1),
\end{equation}
where $g_{e_{ij}}(\mbs x)$ is an arbitrary (free) function, and
\begin{subequations}\label{eq_b108}
\begin{align}
& \mbs z_1(x)=\begin{bmatrix}
1 \\ \psi_0(X_i,X_{i+1},x) \\ \psi_1(X_i,X_{i+1},x) \\ \varphi_0(X_i,X_{i+1},x) \\ \varphi_1(X_i,X_{i+1},x)
\end{bmatrix}, \quad 
\mbs z_2(y)=\begin{bmatrix}
1 \\ \phi_0(Y_j,Y_{j+1},y) \\ \phi_1(Y_j,Y_{j+1},y) 
\end{bmatrix}, \\
& \mbs R_g = \begin{bmatrix}
0 & g_{e_{ij}}(x,Y_j) & g_{e_{ij}}(x,Y_{j+1})  \\
g_{e_{ij}}(X_i,y) & -g_{e_{ij}}(X_i,Y_j) & -g_{e_{ij}}(X_i,Y_{j+1}) \\
g_{e_{ij}}(X_{i+1},y) & -g_{e_{ij}}(X_{i+1},Y_j) & -g_{e_{ij}}(X_{i+1},Y_{j+1})  \\
g_{e_{ij},x}(X_i,y) & -g_{e_{ij},x}(X_i,Y_j) & -g_{e_{ij},x}(X_i,Y_{j+1}) &  \\
g_{e_{ij},x}(X_{i+1},y) & -g_{e_{ij},x}(X_{i+1},Y_j) & -g_{e_{ij},x}(X_{i+1},Y_{j+1}) 
\end{bmatrix},
\quad  \\
& \mbs R_b = \begin{bmatrix}
0 & H_{ij}(x) & H_{i,j+1}(x)  \\
G^{(1)}_{ij}(y) & -\hat\alpha_{ij} & -\hat\alpha_{i,j+1}  \\
G^{(1)}_{i+1,j}(y) & -\hat\alpha_{i+1,j} & -\hat\alpha_{i+1,j+1}  \\
G^{(2)}_{ij}(y) & -\hat\beta_{ij} & - \hat\beta_{i,j+1}  \\
G^{(2)}_{i+1,j}(y) & -\hat\beta_{i+1,j} & -\hat\beta_{i+1,j+1} 
\end{bmatrix}.
\end{align}
\end{subequations}

%One can verify that the function $u(\mbs x)$ defined in~\eqref{eq_69}, with $u_{e_{ij}}(\mbs x)$ given by~\eqref{eq_b107} and by equations~\eqref{eq_b103}--\eqref{eq_b105}, satisfies the conditions~\eqref{eq_a97},
%for arbitrary functions $g_{e_{ij}}(\mbs x)$, $\hat G^{(1)}_{ij}(y)$, $\hat G^{(2)}_{ij}(y)$, and $\hat H_{ij}(x)$ and arbitrary parameter values $\hat\alpha_{ij}$ and $\hat\beta_{ij}$ involved therein.
%One can also verify that any function $u(\mbs x)$ defined by~\eqref{eq_69} that satisfies the conditions in~\eqref{eq_a97} can be expressed in the form as given by~\eqref{eq_b107} and~\eqref{eq_b103}--\eqref{eq_b105}.

\begin{theorem}\label{thm_8}
(i) The functions $u_{e_{ij}}(\mbs x)$ given by~\eqref{eq_b107}, together with~\eqref{eq_b108} and~\eqref{eq_b103}--\eqref{eq_b105}, satisfy the conditions~\eqref{eq_a97}, for arbitrary  $g_{e_{ij}}(\mbs x)$, $\hat G^{(1)}_{ij}(y)$, $\hat G^{(2)}_{ij}(y)$ and $\hat H_{ij}(x)$, and arbitrary parameter values $\hat\alpha_{ij}$ and $\hat\beta_{ij}$ involved therein.
(ii) Any function $u(\mbs x)$ defined by~\eqref{eq_69} that satisfies the conditions~\eqref{eq_a97}
can be expressed in the form as given by~\eqref{eq_b107},~\eqref{eq_b108} and~\eqref{eq_b103}--\eqref{eq_b105}, for some $g_{e_{ij}}(\mbs x)$, $\hat G^{(1)}_{ij}(y)$, $\hat G^{(2)}_{ij}(y)$ and $\hat H_{ij}(x)$, and some value of the parameters $\hat\alpha_{ij}$ and $\hat\beta_{ij}$ therein.
\end{theorem}
\noindent The proof of this theorem follows the same strategy as that of Theorem~\ref{thm_7}.

For mixed FCEs with $C^1$  conditions along the $y$ direction and $C^0$ conditions along the $x$ direction, one can construct the general forms of $u_{e_{ij}}(\mbs x)$ for $0\leqslant(i,j)\leqslant(N_x-1,N_y-1)$ in an analogous way.
 
% Remark on other types of BCs: Neumann

\section{Solving Boundary Value Problems with FCEs: Least Squares Collocation Formulation}
\label{sec_3}

% strong form of PDE, max 2nd order
% domain decomposition, C^k continuity
% collocation points
% least squares solution

We next discuss how to  solve boundary value problems (BVP) using FCEs constructed in the previous section together with  a least squares collocation approach. We restrict our attention to 2D second-order linear PDEs in the following discussions. We defer the discussions on how to use this method to solve nonlinear PDEs, first-order PDEs, and 1D problems to a few remarks at the end of this section.

Consider the domain $\Omega=[a_1,b_1]\times[a_2,b_2]\subset\mbb R^2$ and the following BVP on $\Omega$,
\begin{subequations}\label{eq_100}
\begin{align}
& \mathcal L u(\mbs x) = S(\mbs x), \quad \mbs x=(x,y)\in\Omega, \\
& \mathcal B u(\mbs x) = C(\mbs x), \quad \mbs x\in\partial\Omega,
\end{align}
\end{subequations}
where $\mathcal L$ is a second-order linear differential operator with respect to both $x$ and $y$, $u(\mbs x)$ is the solution field to be sought, $\mathcal B$ is a linear algebraic or differential operator representing the boundary condition (BC), and $S(\mbs x)$ and $C(\mbs x)$ are prescribed functions on $\Omega$ and $\partial\Omega$, respectively.
We focus on Dirichlet BCs (i.e.~$\mathcal B=\mbs I$, the identity operator) here, and outline the ideas on how to deal with other types of BCs using FCEs in a remark at the end of this section. 

Suppose $\Omega$ is partitioned into $N_x$ ($N_x\geqslant 1$) and $N_y$ ($N_y\geqslant 1$) elements along the $x$ and $y$ directions, respectively. Let $\{\ X_i, 0\leqslant i\leqslant N_x\ :\ a_1=X_0<X_1<\cdots<X_{N_x}=b_1 \ \}$ and $\{\ Y_i, 0\leqslant i\leqslant N_y\ :\ a_2=Y_0<Y_1<\cdots<Y_{N_y}=b_2 \ \}$ denote the element boundaries in $x$ and $y$. Let $\Omega_{ij}=[X_i,X_{i+1}]\times[Y_j,Y_{j+1}]$ denote the element with index $e_{ij}=iN_y+j$ for $0\leqslant(i,j)\leqslant(N_x-1,N_y-1)$. 

We solve the problem~\eqref{eq_100} by seeking a piece-wise function $u(\mbs x)$ defined by equation~\eqref{eq_69} that satisfies the following system,
\begin{subequations}\label{eq_101}
\begin{align}
& \mathcal L u_{e_{ij}}(\mbs x) = S(\mbs x), \quad \mbs x\in\Omega_{ij}, \quad 0\leqslant(i,j)\leqslant(N_x-1,N_y-1); \label{eq_101a} \\
& u_{e_{ij}}(X_{i+1},y) = u_{e_{i+1,j}}(X_{i+1},y), \quad y\in[Y_j,Y_{j+1}], \quad (1,0)\leqslant(i,j)\leqslant(N_x-2,N_y-1); \label{eq_101b} \\
& u_{e_{ij}}(x,Y_{j+1}) = u_{e_{i,j+1}}(x,Y_{j+1}), \quad x\in[X_i,X_{i+1}], \quad (0,1)\leqslant(i,j)\leqslant(N_x-1,N_y-2); \label{eq_101c} \\
& \left.\frac{\partial u_{e_{ij}}}{\partial x} \right|_{(X_{i+1},y)} = \left.\frac{\partial u_{e_{i+1,j}}}{\partial x} \right|_{(X_{i+1},y)}, \quad y\in[Y_j,Y_{j+1}], \quad (1,0)\leqslant(i,j)\leqslant(N_x-2,N_y-1); \label{eq_101d} \\
& \left.\frac{\partial u_{e_{ij}}}{\partial y} \right|_{(x,Y_{j+1})} = \left.\frac{\partial u_{e_{i,j+1}}}{\partial y} \right|_{(x,Y_{j+1})}, \quad x\in[X_i,X_{i+1}],  \quad (0,1)\leqslant(i,j)\leqslant(N_x-1,N_y-2); \label{eq_101e} \\
& u_{e_{0j}}(a_1,y) = C(a_1,y), 
\quad y\in[Y_j,Y_{j+1}], \quad 0\leqslant j\leqslant N_y-1; \label{eq_101f} \\
& u_{e_{N_x-1,j}}(b_1,y) = C(b_1,y),
\quad y\in[Y_j,Y_{j+1}], \quad 0\leqslant j\leqslant N_y-1; \label{eq_101g} \\
& u_{e_{i0}}(x,a_2) = C(x,a_2), \quad x\in[X_i,X_{i+1}], \quad 0\leqslant i\leqslant N_x-1; \label{eq_101h} \\
& u_{e_{i,N_y-1}}(x,b_2) = C(x,b_2), \quad x\in[X_i,X_{i+1}], \quad 0\leqslant i\leqslant N_x-1. \label{eq_101i}
\end{align}
\end{subequations}
Equation~\eqref{eq_101a} represents the requirement that $u(\mbs x)$ should satisfy the PDE on each element. Equations~\eqref{eq_101b}--\eqref{eq_101e} represent the requirement that $u(\mbs x)$ should satisfy the $C^1$ continuity  across the interior element boundaries in $x$ and $y$ directions. Equations~\eqref{eq_101f}--\eqref{eq_101i} are the boundary conditions imposed on the elements adjacent to the domain boundary $\partial\Omega$.

In what follows we consider three approaches for numerically solving the system~\eqref{eq_101} based on the least squares collocation idea. The first approach is based on $C^1$ FCEs, which automatically and exactly satisfy the $C^0$ and $C^1$  conditions in~\eqref{eq_101b}--\eqref{eq_101e}. The second approach is based on $C^0$ FCEs, which automatically and exactly satisfy the $C^0$  conditions~\eqref{eq_101b}--\eqref{eq_101c}, but not the $C^1$ conditions~\eqref{eq_101d}--\eqref{eq_101e}. The $C^1$  conditions are imposed in the least squares sense instead with this approach. In the third approach, the piece-wise function $u(\mbs x)$ does not  satisfy any continuity constraint across the element boundaries, and the $C^0$ and $C^1$ conditions~\eqref{eq_101b}--\eqref{eq_101e} are all imposed in the least squares sense.

\paragraph{Solution by $C^1$ FCEs}

Let us first consider using the $C^1$ FCEs constructed in Section~\ref{sec_232} to solve the system~\eqref{eq_101}. In this case $u_{e_{ij}}(\mbs x)$ is given by~\eqref{eq_97} for $0\leqslant(i,j)\leqslant(N_x-1,N_y-1)$, which satisfies the continuity conditions~\eqref{eq_101b}--\eqref{eq_101e} automatically. 

In light of the boundary conditions~\eqref{eq_101f}--\eqref{eq_101i}, the functions $G^{(1)}_{ij}(y)$ and $H^{(1)}_{ij}(x)$ in the $u_{e_{ij}}(\mbs x)$ expression~\eqref{eq_97} satisfy the following conditions,
\begin{subequations}\label{eq_a99}
\begin{align}
& G^{(1)}_{0j}(y) = C(a_1,y), \quad y\in[Y_j,Y_{j+1}], \quad 0\leqslant j\leqslant N_y-1; \label{eq_a99a} \\
& G^{(1)}_{N_x,j}(y) = C(b_1,y), \quad y\in[Y_j,Y_{j+1}], \quad 0\leqslant j\leqslant N_y-1; \label{eq_a99b} \\
& H^{(1)}_{i0}(x) = C(x,a_2), \quad x\in[X_i,X_{i+1}], \quad 0\leqslant i\leqslant N_x-1; \label{eq_a99c} \\
& H^{(1)}_{i,N_y}(x) = C(x,b_2), \quad x\in[X_i,X_{i+1}], \quad 0\leqslant i\leqslant N_x-1. \label{eq_a99d}
\end{align}
\end{subequations}
In addition, the following conditions on the parameters $\hat\alpha_{ij}$, $\hat\beta^{(1)}_{ij}$ and $\hat\beta^{(2)}_{ij}$ involved in~\eqref{eq_97} hold,
\begin{subequations}\label{eq_a100}
\begin{align}
& \hat\alpha_{0j} = C(a_1,Y_j), \quad \hat\alpha_{N_x,j} = C(b_1,Y_j), \quad 0\leqslant j\leqslant N_y; \\
& \hat\alpha_{i0} = C(X_i,a_2), \quad \hat\alpha_{i,N_y} = C(X_i,b_2), \quad 0\leqslant i\leqslant N_x; \\
& \hat\beta^{(1)}_{i0} = C_x(X_i,a_2), \quad \hat\beta^{(1)}_{i,N_y} = C_x(X_i,b_2), \quad 0\leqslant i\leqslant N_x; \\
& \hat\beta^{(2)}_{0j} = C_y(a_1,Y_j), \quad \hat\beta^{(2)}_{N_x,j} = C_y(b_1,Y_j), \quad 0\leqslant j\leqslant N_y,
\end{align}
\end{subequations}
where $C_x(x,y)$ and $C_y(x,y)$ denote the partial derivatives of $C(x,y)$ with respect to $x$ and $y$, respectively.

Therefore, in~\eqref{eq_97}, the functions $G^{(1)}_{ij}(y)$ are given by~\eqref{eq_a99a}--\eqref{eq_a99b}, and by~\eqref{eq_92} and~\eqref{eq_99b} for $(1,0)\leqslant(i,j)\leqslant(N_x-1,N_y-1)$. The functions $H^{(1)}_{ij}(x)$ are given by~\eqref{eq_a99c}--\eqref{eq_a99d}, and by~\eqref{eq_94} and~\eqref{eq_99d} for $(0,1)\leqslant(i,j)\leqslant(N_x-1,N_y-1)$. The functions $G^{(2)}_{ij}(y)$ ($0\leqslant(i,j)\leqslant(N_x,N_y-1)$) and $H^{(2)}_{ij}(x)$ ($0\leqslant(i,j)\leqslant(N_x-1,N_y)$) are the same as those given in Section~\ref{sec_232}. In the parameterized form of $u_{e_{ij}}(\mbs x)$ for $0\leqslant(i,j)\leqslant(N_x-1,N_y-1)$, let $\bm\Theta$ denote the vector of the unknown coefficients to be determined.
Then $\bm\Theta$ is given by,
\begin{equation}\label{eq_a101}
\begin{split}
\bm\Theta = \{\ &\hat g_{e_{ij},k},\ \text{for}\ (0,0,1)\leqslant(i,j,k)\leqslant(N_x-1,N_y-1,\mathcal M));  \\
& \mathscr{\hat G}^{(1)}_{ij,k}\ \text{for}\ (1,0,1) \leqslant(i,j,k)\leqslant(N_x-1,N_y-1,m)); \\
& \mathscr{\hat G}^{(2)}_{ij,k}\ \text{for}\ 
(0,0,1)\leqslant(i,j,k)\leqslant(N_x,N_y-1,m); \\ 
& \mathscr{\hat H}^{(1)}_{ij,k}, \ \text{for}\ 
(0,1,1)\leqslant(i,j,k)\leqslant(N_x-1,N_y-1,m); \\
& \mathscr{\hat H}^{(2)}_{ij,k}, \ \text{for}\ 
(0,0,1)\leqslant(i,j,k)\leqslant(N_x-1,N_y,m); \\ 
& \hat\alpha_{ij}, \ \text{for}\ 
1\leqslant(i,j)\leqslant(N_x-1,N_y-1); \quad
 \hat\beta^{(1)}_{ij}, \ \text{for}\ 
(0,1)\leqslant(i,j)\leqslant(N_x,N_y-1); \\ 
 & \hat\beta^{(2)}_{ij}, \ \text{for}\ 
(1,0)\leqslant(i,j)\leqslant(N_x-1,N_y); \quad 
 \hat\gamma_{ij}, \ \text{for}\ 
0\leqslant (i,j)\leqslant(N_x,N_y) \ \}.
\end{split}
\end{equation}
With this parametric form  for $u_{e_{ij}}(\mbs x)$,  the conditions~\eqref{eq_101b}--\eqref{eq_101i} are automatically satisfied, and the equations in~\eqref{eq_101a} are the only ones that need to be solved.

% collocation points

Let 
\begin{equation}\label{eq_102}
\mbb X = \{\ \mbs x_p^{e_{ij}}\in\Omega_{ij}\ |\ 0\leqslant(i,j,p)\leqslant(N_x-1,N_y-1,Q-1)  \ \}
\end{equation}
denote a set of $Q$ collocation points from each sub-domain. While a variety of distributions for the collocation points can be considered, we employ either Gauss-Lobatto-Legendre quadrature points or a uniform set of grid points on each sub-domain as the collocation points in numerical simulations. We enforce equation~\eqref{eq_101a} on the collocation points. This leads to
\begin{equation}\label{eq_103}
\mathcal L u_{e_{ij}}(\mbs x_p^{e_{ij}}) = S(\mbs x_p^{e_{ij}}), \quad \mbs x_p^{e_{ij}}\in\mbb X, \quad 0\leqslant(i,j,p)\leqslant (N_x-1,N_y-1,Q-1),
\end{equation}
where $u_{e_{ij}}(\mbs x)$ is the parametric form of~\eqref{eq_97} with the unknown parameters $\bm\Theta$ defined in~\eqref{eq_a101}. This is a linear algebraic system of equations about the unknown parameters $\bm\Theta$. The system involves terms $\mathcal L f(\mbs x_p^{e_{ij}})$, where $f(\mbs x)$ is a known function consisting of the basis functions $\Psi_{ij,k}$ and $\Phi_{ik}$. These terms are well-defined and can be computed once the basis functions are specified. The system~\eqref{eq_103} is rectangular, in which the number of equations and the number of unknowns are not equal in general. Symbolically we re-write this system into,
\begin{equation}\label{eq_104}
\mbs H\bm\Theta = \mbs S,
\end{equation}
where $\mbs H$ denotes the rectangular coefficient matrix, $\bm\Theta$ is the vector of unknown parameters, and $\mbs S$ denotes the right hand side (RHS).
We seek a least squares solution to~\eqref{eq_104}, and solve this system for $\bm\Theta$ by the linear least squares method~\cite{Bjorck1996}.
In our implementation, the least squares solution to~\eqref{eq_104} is obtained by the ``lsqr'' routine in Matlab. Upon attaining  $\bm\Theta$, the solution to the original system~\eqref{eq_101} is given by~\eqref{eq_97}.

\paragraph{Solution by $C^0$ FCEs}

We next consider solving the system~\eqref{eq_101} using the $C^0$ FCEs constructed in Section~\ref{sec_232}. In this case, $u_{e_{ij}}(\mbs x)$ is given by~\eqref{eq_79} for $0\leqslant(i,j)\leqslant(N_x-1,N_y-1)$, which satisfy the conditions~\eqref{eq_101b}--\eqref{eq_101c} automatically.  

In light of the conditions~\eqref{eq_101f}--\eqref{eq_101i}, the functions $G_{ij}(y)$ and $H_{ij}(x)$ involved in the $u_{e_{ij}}$ expression~\eqref{eq_79} satisfy,
\begin{subequations}\label{eq_105}
\begin{align}
& G_{0j}(y) = C(a_1,y), \quad y\in[Y_j,Y_{j+1}], \quad 0\leqslant j\leqslant N_y-1; \label{eq_105a} \\
& G_{N_x,j}(y) = C(b_1,y), \quad y\in[Y_j,Y_{j+1}], \quad 0\leqslant j\leqslant N_y-1; \label{eq_105b} \\
& H_{i0}(x) = C(x,a_2), \quad x\in[X_i,X_{i+1}], \quad 0\leqslant i\leqslant N_x-1; \label{eq_105c} \\
& H_{i,N_y}(x) = C(x,b_2), \quad x\in[X_i,X_{i+1}], \quad 0\leqslant i\leqslant N_x-1. \label{eq_105d}
\end{align}
\end{subequations}
In addition, the parameters $\hat\alpha_{ij}$ involved in~\eqref{eq_79} satisfy the conditions,
\begin{subequations}\label{eq_106}
\begin{align}
& \hat\alpha_{0j} = C(a_1,Y_j), \quad \hat\alpha_{N_x,j} = C(b_1,Y_j), \quad 0\leqslant j\leqslant N_y; \\
& \hat\alpha_{i0} = C(X_i,a_2), \quad \hat\alpha_{i,N_y} = C(X_i,b_2), \quad 0\leqslant i\leqslant N_x.
\end{align}
\end{subequations}

Therefore, in~\eqref{eq_79}, the functions $G_{ij}(y)$ are given by~\eqref{eq_105a}--\eqref{eq_105b}, and by~\eqref{eq_76} for $(1,0)\leqslant(i,j)\leqslant(N_x-1,N_y-1)$.  The functions $H_{ij}(x)$ are given by~\eqref{eq_105c}--\eqref{eq_105d}, and by~\eqref{eq_77} for $(0,1)\leqslant(i,j)\leqslant(N_x-1,N_y-1)$. Let $\bm\Theta$ denote the unknown parameters to be determined in the parametric form of~\eqref{eq_79}, which is given by,
\begin{equation}\label{eq_107}
\begin{split}
\bm\Theta = \{\ &\hat g_{e_{ij},k},\ \text{for}\ (0,0,1)\leqslant(i,j,k)\leqslant(N_x-1,N_y-1,\mathcal M);  \\
& \mathscr{\hat G}_{ij,k},\ \text{for}\ (1,0,1) \leqslant(i,j,k)\leqslant(N_x-1,N_y-1,m); \\
& \mathscr{\hat H}_{ij,k}, \ \text{for}\ 
(0,1,1)\leqslant(i,j,k)\leqslant(N_x-1,N_y-1,m); \\
& \hat\alpha_{ij}, \ \text{for}\ 
1\leqslant(i,j)\leqslant(N_x-1,N_y-1)
 \ \}.
\end{split}
\end{equation}
With this parametric form of $u_{e_{ij}}(\mbs x)$, the equations~\eqref{eq_101b}--\eqref{eq_101c} and~\eqref{eq_101f}--\eqref{eq_101i} are automatically satisfied. The equations in~\eqref{eq_101a} and~\eqref{eq_101d}--\eqref{eq_101e} are the only ones that need to be solved.

% collocation points

We choose a set of collocation points from each element as defined in~\eqref{eq_102}, with the requirements that (i) a non-empty subset of the collocation points should reside on each of the boundaries of each element, and (ii) the collocation points from two neighboring elements that reside on their shared element boundary should match. Let
\begin{equation}\label{eq_108}
\left\{
\begin{split}
& \mbb Y_{ij}^b = \mbb X\cap\Omega_{ij}\cap\Omega_{i+1,j},
\quad 0\leqslant(i,j)\leqslant(N_x-2,N_y-1); \\
& \mbb Z_{ij}^b = \mbb X\cap\Omega_{ij}\cap\Omega_{i,j+1}, \quad 0\leqslant(i,j)\leqslant(N_x-1,N_y-2),
\end{split}
\right.
\end{equation}
which denote the sets of collocation points residing on the vertical and horizontal element boundaries, respectively. Uniform grid points or Gauss-Lobatto-Legendre quadrature points on each element are employed as the collocation points in our  simulations.

% enforce equations

We enforce the equations~\eqref{eq_101a},~\eqref{eq_101d} and~\eqref{eq_101e} on the collocation points as follows,
\begin{subequations}\label{eq_109}
\begin{align}
& \mathcal L u_{e_{ij}}(\mbs x_p^{e_{ij}}) = S(\mbs x_p^{e_{ij}}), \quad \mbs x_{p}^{e_{ij}}\in\mbb X, \quad 0\leqslant(i,j,p)\leqslant(N_x-1,N_y-1,Q-1); \\
& \left.\frac{\partial u_{e_{ij}}}{\partial x}\right|_{\mbs y_p}
= \left.\frac{\partial u_{e_{i+1,j}}}{\partial x}\right|_{\mbs y_p}, \quad \mbs y_p=(X_{i+1},y_p)\in\mbb Y^b_{ij}, \quad 0\leqslant(i,j,p)\leqslant(N_x-2,N_y-1,|\mbb Y^b_{ij}|-1); \\
& \left.\frac{\partial u_{e_{ij}}}{\partial y}\right|_{\mbs z_p}
= \left.\frac{\partial u_{e_{i,j+1}}}{\partial y}\right|_{\mbs z_p}, \quad \mbs z_p=(x_p,Y_{j+1})\in\mbb Z^b_{ij}, \quad 0\leqslant(i,j,p)\leqslant(N_x-1,N_y-2,|\mbb Z^b_{ij}|-1),
\end{align}
\end{subequations}
where $|\mbb Y_{ij}^b|$ and $|\mbb Z_{ij}^b|$ denote the size of $\mbb Y_{ij}^b$ and $\mbb Z_{ij}^b$, and $u_{e_{ij}}(\mbs x)$ is given by~\eqref{eq_79}.
This is a system linear algebraic system of equations about the parameters $\bm\Theta$ as defined in~\eqref{eq_107}, which can symbolically be written into the form in~\eqref{eq_104}. We seek a least squares solution to this system, and solve this system by the linear least squares method. Upon finding $\bm\Theta$, the solution to the original system~\eqref{eq_101} is given by~\eqref{eq_79}.

\paragraph{Solution by FCEs with No Continuity (FCE-NC)}

Employing the least squares collocation approach, we can also solve the system~\eqref{eq_101} using piece-wise functions $u_{e_{ij}}(\mbs x)$ for $0\leqslant(i,j)\leqslant(N_x-1,N_y-1)$ satisfying no intrinsic continuity  across the element boundaries. In this case, all the conditions in~\eqref{eq_101b}--\eqref{eq_101i} will be enforced in the least squares sense. We will refer to such elements as FCE-NC (i.e.~FCE with No Continuity) hereafter.

We restrict $u_{e_{ij}}(\mbs x)$ to the function space $\mathcal G(\Omega_{ij},\mathcal M)$ defined in~\eqref{eq_82}. Let
\begin{align}\label{eq_110}
u_{e_{ij}}(\mbs x)=\sum_{k=1}^{\mathcal M} \hat u_{e_{ij},k}\Psi_{ij,k}(\mbs x), \quad \mbs x\in\Omega_{ij}, \quad 0\leqslant(i,j)\leqslant(N_x-1,N_y-1),
\end{align}
where $\hat u_{e_{ij},k}$ are the unknown coefficients to be determined.
Define 
\begin{equation}\label{eq_111}
\bm\Theta = \{\ 
\hat u_{e_{ij},k}, \ (0,0,1)\leqslant(i,j,k)\leqslant(N_x-1,N_y-1,\mathcal M)
\ \}. 
\end{equation}

% collocation points

We again choose a set of collocation points from each element as defined in~\eqref{eq_102}, with the same requirements as for the $C^0$ FCEs. In other words, for each boundary of each element there is a non-empty subset of collocation points from $\mbb X$ should reside on this boundary, and those collocation points from two adjacent elements that reside on their shared element boundary should match. Define
\begin{equation}
\left\{
\begin{split}
& \mbb Y^{L}_{j} = \{\ \mbs x_p=(a_1,y_p)\in\mbb X\ |\ \mbs x_p\in\Omega_{0j} \ \}, \quad \quad 0\leqslant j\leqslant N_y-1;
\\ 
& \mbb Y^{R}_j = \{\ \mbs x_p=(b_1,y_p)\in\mbb X\ |\ \mbs x_p\in\Omega_{N_x-1,j} \ \},  \quad 0\leqslant j\leqslant N_y-1; \\
& \mbb Z^{B}_i = \{\ \mbs x_p=(x_p,a_2)\in\mbb X\ |\ \mbs x_p\in\Omega_{i0} \ \}, \quad 0\leqslant i\leqslant N_x-1; \\
& \mbb Z^{T}_i = \{\ \mbs x_p=(x_p,b_2)\in\mbb X\ |\ \mbs x_p\in\Omega_{i,N_y-1} \ \}, \quad 0\leqslant i\leqslant N_x-1.
\end{split}
\right.
\end{equation}

We enforce the equations~\eqref{eq_101a}--\eqref{eq_101i} on the collocation points as follows,
\begin{subequations}\label{eq_113}
\begin{align}
& \mathcal L u_{e_{ij}}(\mbs x_p^{e_{ij}}) = S(\mbs x_p^{e_{ij}}), \quad \mbs x_p^{e_{ij}}\in\mbb X, \quad 0\leqslant(i,j,p)\leqslant(N_x-1,N_y-1,Q-1); \label{eq_113a} \\
& u_{e_{ij}}(\mbs y_p) = u_{e_{i+1,j}}(\mbs y_p), \quad \mbs y_p=(X_{i+1},y_p)\in\mbb Y_{ij}^b, \ 0\leqslant(i,j,p)\leqslant(N_x-2,N_y-1,|\mbb Y_{ij}^b|-1); \label{eq_113b} \\
& u_{e_{ij}}(\mbs z_p) = u_{e_{i,j+1}}(\mbs z_p), \quad \mbs z_p=(x_p,Y_{j+1})\in\mbb Z_{ij}^b, \quad 0\leqslant(i,j,p)\leqslant(N_x-1,N_y-2,|\mbb Z_{ij}^b|-1); \label{eq_113c} \\
& \left.\frac{\partial u_{e_{ij}}}{\partial x} \right|_{\mbs y_p} = \left.\frac{\partial u_{e_{i+1,j}}}{\partial x} \right|_{\mbs y_p}, \ \mbs y_p=(X_{i+1},y_p)\in\mbb Y_{ij}^b, \ 0\leqslant(i,j,p)\leqslant(N_x-2,N_y-1,|\mbb Y_{ij}^b|-1); \label{eq_113d} \\
& \left.\frac{\partial u_{e_{ij}}}{\partial y} \right|_{\mbs z_p} = \left.\frac{\partial u_{e_{i,j+1}}}{\partial y} \right|_{\mbs z_p}, \quad \mbs z_p=(x_p,Y_{j+1})\in\mbb Z_{ij}^b,  \quad 0\leqslant(i,j,p)\leqslant(N_x-1,N_y-2,|\mbb Z_{ij}^b|-1); \label{eq_113e} \\
& u_{e_{0j}}(a_1,y_p) = C(a_1,y_p), 
\quad (a_1,y_p)\in\mbb Y_j^L, \quad 0\leqslant (j,p)\leqslant (N_y-1,|\mbb Y_j^L|-1); \label{eq_113f} \\
& u_{e_{N_x-1,j}}(b_1,y_p) = C(b_1,y_p),
\quad (b_1,y_p)\in\mbb Y_j^R, \quad 0\leqslant (j,p)\leqslant (N_y-1,|\mbb Y_j^R|-1); \label{eq_113g} \\
& u_{e_{i0}}(x_p,a_2) = C(x_p,a_2), \quad (x_p,a_2)\in\mbb Z_i^B, \quad 0\leqslant (i,p)\leqslant (N_x-1,|\mbb Z_i^B|-1); \label{eq_113h} \\
& u_{e_{i,N_y-1}}(x_p,b_2) = C(x_p,b_2), \quad(x_p,b_2)\in\mbb Z_i^T, \quad 0\leqslant (i,p)\leqslant (N_x-1,|\mbb Z_i^B|-1). \label{eq_113i}
\end{align}
\end{subequations}
In these equations $u_{e_{ij}}(\mbs x)$ is given by~\eqref{eq_110}, and $\mbb Y_{ij}^b$ and $\mbb Z_{ij}^b$ are defined in~\eqref{eq_108}.  
The equations in~\eqref{eq_113} form a linear algebraic system of equations about the unknown parameters $\bm\Theta$ defined in~\eqref{eq_111}. This system can be solved by the linear least squares method to yield a least squares solution for $\bm\Theta$. The solution to the original system~\eqref{eq_101} is then given by~\eqref{eq_110}.

\begin{remark}\label{rem_4}
% nonlinear PDEs

The method  presented in this section can be extended to solve nonlinear boundary value problems. Consider the following modified BVP on domain $\Omega$,
\begin{subequations}\label{eq_114}
\begin{align}
& \mathcal L u(\mbs x) + \mathcal N(u(\mbs x)) = S(\mbs x), \quad \mbs x\in\Omega, \\
& u(\mbs x) = C(\mbs x), \quad \mbs x\in\partial\Omega,
\end{align}
\end{subequations}
where $\mathcal N(u)$ denotes a nonlinear operator on $u(\mbs x)$. We can use the same procedure to solve this problem, leading to
\begin{align}\label{eq_115}
& \mathcal L u_{e_{ij}}(\mbs x_p^{e_{ij}}) + \mathcal N(u_{e_{ij}}(\mbs x_p^{e_{ij}})) = S(\mbs x_p^{e_{ij}}), \quad \mbs x_p^{e_{ij}}\in\mbb X, \quad 0\leqslant(i,j,p)\leqslant(N_x-1,N_y-1,Q-1),
\end{align}
together with the equations~\eqref{eq_101b}--\eqref{eq_101i}. By substituting the parametric form of $u_{e_{ij}}(\mbs x)$ for  $C^1$ or $C^0$ FCEs, or for FCEs with no continuity, into~\eqref{eq_115}, we attain an algebraic system about $\bm\Theta$ (unknown coefficients to be determined). This algebraic system, however, is nonlinear with respect to $\bm\Theta$.  
%The number of equations and the number of unknowns in this system are not equal in general. 
We seek a least squares solution to this nonlinear algebraic system, and solve this system by the nonlinear least squares method~\cite{Bjorck1996,Bjorck2015}. In our implementation, we employ the routine ``lsqnonlin'' in Matlab 
for the nonlinear least squares method. Upon finding $\bm\Theta$, the solution to the original nonlinear BVP is computed based on the parametric form of  $u_{e_{ij}}(\mbs x)$ for different types of FCEs. 

\end{remark}

\begin{remark}
% PDEs with 1st order for 1 variable and 2nd order for another variable

In the above discussions we have assumed  that $\mathcal L$ is a second-order differential operator with respect to both $x$ and $y$. The formulation can be modified to deal with PDEs in which $\mathcal L$ is of first order with respect to one or both variables. 

Suppose $\mathcal L$ is of second order with respect to one variable but  first order with respect to the other. Without loss of generality let us assume that it is second-order in $x$ and first-order in $y$. One such example is the heat equation, $\mathcal L=\frac{\partial}{\partial t} - \frac{\partial^2}{\partial x^2}$ (with $t\equiv y$). Upon partition of $\Omega$ by the elements, the system to be solved consists of equations~\eqref{eq_101a},~\eqref{eq_101f}--\eqref{eq_101h}, and the continuity conditions
\begin{subequations}\label{eq_116}
\begin{align}
& u_{e_{ij}}(X_{i+1},y) = u_{e_{i+1,j}}(X_{i+1},y), \quad y\in[Y_j,Y_{j+1}], \quad 0\leqslant(i,j)\leqslant(N_x-2,N_y-1); \label{eq_116a} \\
& u_{e_{ij}}(x,Y_{j+1}) = u_{e_{i,j+1}}(x,Y_{j+1}), \quad x\in[X_i,X_{i+1}], \quad 0\leqslant(i,j)\leqslant(N_x-1,N_y-2); \label{eq_116b} \\
& \left.\frac{\partial u_{e_{ij}}}{\partial x} \right|_{(X_{i+1},y)} = \left.\frac{\partial u_{e_{i+1,j}}}{\partial x} \right|_{(X_{i+1},y)}, \quad y\in[Y_j,Y_{j+1}], \quad 0\leqslant(i,j)\leqslant(N_x-2,N_y-1). \label{eq_116c} 
%& \left.\frac{\partial u_{e_{ij}}}{\partial y} \right|_{(x,Y_{j+1})} = \left.\frac{\partial u_{e_{i,j+1}}}{\partial y} \right|_{(x,Y_{j+1})}, \quad x\in[X_i,X_{i+1}],  \quad (0,1)\leqslant(i,j)\leqslant(N_x-1,N_y-2); \label{eq_116d}
\end{align}
\end{subequations}
Note that here we impose  $C^1$ continuity  across the element boundaries along the $x$ direction since $\mathcal L$ is second-order in $x$, and impose $C^0$ continuity  only across the element boundaries along the $y$ direction since $\mathcal L$ is  first-order in $y$. In general if the PDE operator is of $k$-th order along some direction, we would impose $C^{k-1}$ continuity conditions along that direction. 
If the PDE order is higher than two, one can reformulate the PDE into a system of PDEs of at most second order by introducing appropriate auxiliary variables.
Note that in the above we impose the ``boundary'' condition only on one side in the $y$ direction, which is in fact the initial condition.
The resultant system can be solved using  mixed FCEs or $C^0$ FCEs from Section~\ref{sec_232} or using FCEs with no continuity based on the least squares collocation approach in a way analogous to what has been discussed above.

If $\mathcal L$ is first order with respect to both $x$ and $y$, the system to be solved consists of equations~\eqref{eq_101a}, \eqref{eq_101f}, \eqref{eq_101h}, and the continuity conditions
\begin{subequations}\label{eq_117}
\begin{align}
& u_{e_{ij}}(X_{i+1},y) = u_{e_{i+1,j}}(X_{i+1},y), \quad y\in[Y_j,Y_{j+1}], \quad 0\leqslant(i,j)\leqslant(N_x-2,N_y-1); \label{eq_117a} \\
& u_{e_{ij}}(x,Y_{j+1}) = u_{e_{i,j+1}}(x,Y_{j+1}), \quad x\in[X_i,X_{i+1}], \quad 0\leqslant(i,j)\leqslant(N_x-1,N_y-2). \label{eq_117b} 
\end{align}
\end{subequations}
Here we impose $C^0$ continuity  across the element boundaries in both $x$ and $y$, and impose only one ``boundary'' condition in both directions.
The resultant system can be solved by using  $C^0$ FCEs from Section~\ref{sec_232} or using FCEs with no continuity based on the least squares collocation approach in an analogous way.

These discussions signify that the method developed here can be used for solving not only boundary value problems, but also for initial/boundary value problems. In the latter case,  the time variable ``t'' will be treated on the same footing as the spatial variables, and the problem is handled by a space-time approach.

\end{remark}

\begin{remark}\label{rem_6}
% 1D problems

Boundary value problems in 1D can be solved using 1D FCEs ($C^1$ FCEs, $C^0$ FCEs, or FCEs with no continuity) and the least squares collocation approach in the same manner. 
%
%%%%%%%%%%%%%%%%%%
%\begin{comment}
%
We use the following BVP as an example to outline the idea,
\begin{subequations}
\begin{align}
& \mathcal L u(x) = f(x), \quad x\in\Omega=[a,b]; \\
& u(a) = C_a, \quad u(b) = C_b, 
\end{align}
\end{subequations}
where $\mathcal L$ is a 1D  second-order linear differential operator, and $C_a$ and $C_b$ are prescribed values. Now partition $\Omega$ into $N$ elements by the points $a=X_0<X_1<\cdots<X_N=b$, and define
\begin{equation}\label{eq_131}
u(x) = \left\{
\begin{array}{ll}
u_0(x), & x\in\Omega_0=[X_0,X_1]; \\
\cdots \\
u_i(x), & x\in\Omega_i=[X_i,X_{i+1}]; \\
\cdots \\
u_{N-1}(x), & x\in\Omega_{N-1}=[X_{N-1},X_N].
\end{array}
\right.
\end{equation}
The system to be solved becomes,
\begin{subequations}\label{eq_132}
\begin{align}
& \mathcal Lu_i(x) = f(x), \quad x\in\Omega_i, \quad 0\leqslant i\leqslant N-1; \label{eq_132a} \\
& u_i(X_{i+1}) = u_{i+1}(X_{i+1}), \quad 0\leqslant i\leqslant N-2; \label{eq_132b} \\
& \left. \frac{d u_i}{\partial x} \right|_{X_{i+1}} = \left. \frac{d u_{i+1}}{\partial x} \right|_{X_{i+1}}, \quad 0\leqslant i\leqslant N-2; \label{eq_132c} \\
& u_0(a) = C_a; \label{eq_132d} \\
& u_{N-1}(b) = C_b. \label{eq_132e}
\end{align}
\end{subequations}
This system can be solved using 1D $C^1$ and $C^0$ FCEs or using 1D FCE-NCs, in a way analogous to the 2D problems. We next  use $C^1$ FCEs from Section~\ref{sec_231} for illustration. 

Using the definitions in~\eqref{eq_a63} and in light of the BCs~\eqref{eq_132d}--\eqref{eq_132e}, we have 
$ %\begin{equation}
\alpha_0 = C_a$ and $\alpha_N = C_b.
$ %\end{equation}
Therefore, in the parametric form~\eqref{eq_68} for $u_i(x)$, the unknown parameters  are
$
\bm\Theta = \{\ \hat g_{ij}\ \text{for}\ (0,1)\leqslant(i,j)\leqslant(N-1,m); \ 
\alpha_i\ \text{for}\ 1\leqslant i\leqslant N-1; \ 
\beta_i,\ \text{for}\ 0\leqslant i\leqslant N
\ \}.
$
With this parametric form, the equations~\eqref{eq_132b}--\eqref{eq_132e} are automatically satisfied. The equations in~\eqref{eq_132a} are the only ones that need to be solved. By choosing a set of collocation points on each element $\Omega_i$ and enforcing~\eqref{eq_132a} on these collocation points, we obtain a rectangular linear algebraic system about the parameters $\bm\Theta$, which can be solved by the linear least squares method. The solution to the original problem is than given by~\eqref{eq_131} and~\eqref{eq_68}.

When using 1D $C^0$ FCEs to solve~\eqref{eq_132}, the conditions~\eqref{eq_132b} and~\eqref{eq_132d}--\eqref{eq_132e} are automatically satisfied. So only the equations~\eqref{eq_132a} and~\eqref{eq_132c} need to be solved by the least squares method. On the other hand, when using 1D FCE-NCs,  all the equations~\eqref{eq_132a}--\eqref{eq_132e} need to be solved by least squares.

%\end{comment}
%%%%%%%%%%%%%%%%%%%%%

\end{remark}

\begin{remark}
% other types of BCs

We have used Dirichlet BCs  when describing the method. The method can be readily extended to accommodate other types of BCs, by modifying the parameters  that need to be computed in the FCE formulation or by modifying the FCE form for the elements adjacent to the domain boundary. We next briefly discuss how to handle Neumann and Robin type BCs.

\underline{Neumann BC.} With $C^0$ FCEs and FCEs with no continuity (FCE-NC), one can enforce the Neumann BC in the least squares sense, by adding the BC to those equations that need to be solved. With $C^1$ FCEs, one can enforce the Neumann BC exactly by modifying the set of parameters that need to be computed in the FCE parametric form. On those elements adjacent to the boundary, the $\hat\beta^{(1)}_{ij}$ or $\hat\beta_{ij}^{(2)}$ parameters and the $G_{ij}^{(2)}$ or $H_{ij}^{(2)}$ functions in the $C^1$ FCE form corresponding to the Neumann boundary will be known and set according to the Neumann BC data. This modifies the set of unknown coefficients that need to be computed in the parametric form of the $C^1$ FCE. 

\underline{Robin BC.} With $C^0$ FCEs and FCE-NCs, one can similarly enforce the Robin BC in the least squares sense, by adding the BC to those equations to be solved. With $C^1$ FCEs, one can enforce the Robin BC exactly as follows by modifying the set of parameters that needs to be computed in the FCE parametric form. On those elements adjacent to the Robin boundary, the parameters $\hat\alpha_{ij}$ and $\hat\beta_{ij}^{(1)}$ or $\hat\beta_{ij}^{(2)}$ in the FCE formulation corresponding to this boundary are not independent due to the Robin BC. The functions $G^{(1)}_{ij}(y)$ and $G^{(2)}_{ij}(y)$, or $H^{(1)}_{ij}(x)$ and $H^{(2)}_{ij}(x)$, corresponding to the Robin boundary are similarly not independent. In these pairs, one can choose one parameter (or function) as the unknown to be computed. Then the other parameter (or function) will be determined through the Robin BC. This modifies the set of unknown coefficients that need to be computed in the parametric form of the $C^1$ FCE. 

\end{remark}

\begin{remark}
% basis functions

To arrive at a specific computational technique, the basis functions for the spaces $\mathcal F_m([\xi_1, \xi_2])$ in~\eqref{eq_61} and $\mathcal G([\xi_1,\xi_2]\times[\eta_1,\eta_2],\mathcal M)$ in~\eqref{eq_82} still need to be specified. In principle one can  employ any function space that is sufficiently expressive for this purpose. 
In the majority of numerical simulations, we have employed Legendre polynomials as the bases for  $\mathcal F_m([\xi_1, \xi_2])$ (by first mapping $[\xi_1, \xi_2]$ to the standard interval $[-1,1]$), and employed the tensor products of Legendre polynomials as the bases for $\mathcal G([\xi_1,\xi_2]\times[\eta_1,\eta_2],\mathcal M)$ (by first mapping $[\xi_1,\xi_2]\times[\eta_1,\eta_2]$ to the standard domain $[-1,1]\times[-1,1]$).
We have also implemented non-polynomial bases with quasi-random sinusoidal functions for $\mathcal F_m([\xi_1, \xi_2])$ for  1D problems. These non-polynomial bases will be specified when presenting numerical tests in the next section.
 
\end{remark}

\section{Numerical Examples}
\label{sec_4}

In this section we test the FCE method developed in Sections~\ref{sec_2} and~\ref{sec_3} using several boundary and initial value problems in one and two dimensions. We consider polynomial bases (Legendre polynomials), as well as non-polynomial bases, for representing the free functions in the FCE formulations. Both linear and nonlinear test problems are considered, involving linear/nonlinear differential equations or boundary conditions. In particular, we show several test problems involving linear and nonlinear relative boundary conditions, which can be computed by the FCE method without  difficulty but would be much more challenging to  conventional element-based methods.

\subsection{Free Functions Represented by Polynomial Bases}

We first  represent the free functions involved in the FCE formulation by polynomial bases. 
Specifically, we employ Legendre polynomials $P_n(\xi)$ ($\xi\in[-1,1]$, $n$ denoting the degree) as the bases for representing the free functions in 1D FCEs (such as the ~$g_i(x)$ in Section~\ref{sec_231}), and for representing the free functions associated with the element boundaries in 2D FCEs (such as the~$\hat G_{ij}(y)$, $\hat H_{ij}(x)$, $\hat G^{(1)}_{ij}(y)$, $\hat G^{(2)}_{ij}(y)$, $\hat H^{(1)}_{ij}(y)$, $\hat H^{(2)}_{ij}(x)$ in Section~\ref{sec_232}). We employ the tensor product of  Legendre polynomials as  bases for representing the free functions $g_{e_{ij}}(\mbs x)$ in 2D FCEs.

Specifically, for 1D $C^0$ FCEs the bases for the free functions are
\begin{equation}\label{eq_135}
\mathcal F_p([a,b]) = \{\ P_{n+2}(\xi),\ \xi=2\phi_1(a,b,x)-1,\ x\in[a,b], \ n=0,1,\dots, p-2 \ \},
\end{equation}
where $\phi_1(\chi_1,\chi_2,x)$ is defined in~\eqref{eq_a60} and $p$ denotes the highest polynomial degree. For 1D $C^1$ FCEs the bases for the free functions are
\begin{equation}\label{eq_136}
\mathcal F_p([a,b]) = \{\ P_{n+4}(\xi),\ \xi=2\phi_1(a,b,x)-1,\ x\in[a,b], \ n=0,1,\dots, p-4 \ \}.
\end{equation}
We employ Legendre polynomials with orders $P_2(\xi)$ or $P_4(\xi)$ and higher due to the linear independence requirement for the sets in~\eqref{eq_c61} and~\eqref{eq_c67}.
For 2D $C^0$ FCEs, we employ the bases~\eqref{eq_135} for the free functions $\hat G_{ij}(y)$ and $\hat H_{ij}(x)$, and the following bases for the free functions $g_{e_{ij}}(\mbs x)$,
\begin{equation}
\begin{split}
\mathcal G([a,b]\times[c,d],p) = \{& \ 
P_{m+2}(\xi)P_{n+2}(\eta),\ \xi=2\phi_1(a,b,x)-1,\ \eta=2\phi_1(c,d,y)-1,\\
&\ x\in[a,b],\ y\in[c,d],\ m,n=0,1,\dots,p-2
\ \},
\end{split}
\end{equation}
where $p$ denotes  the highest polynomial degree in each direction. For 2D $C^1$ FCEs, we employ the bases~\eqref{eq_136} for the free functions $\hat G^{(1)}_{ij}(y)$, $\hat G^{(2)}_{ij}(y)$, $\hat H^{(1)}_{ij}(x)$ and $\hat H^{(2)}_{ij}(x)$, and the following bases for the free functions $g_{e_{ij}}(\mbs x)$, \begin{equation}
\begin{split}
\mathcal G([a,b]\times[c,d],p) = \{& \ 
P_{m+4}(\xi)P_{n+4}(\eta),\ \xi=2\phi_1(a,b,x)-1,\ \eta=2\phi_1(c,d,y)-1,\\
&\ x\in[a,b],\ y\in[c,d],\ m,n=0,1,\dots,p-4
\ \}.
\end{split}
\end{equation}

\subsubsection{1D Examples}

\begin{figure}[tb]
\centering 
\subfigure[h-refinement]{
\includegraphics[width=0.35\textwidth]{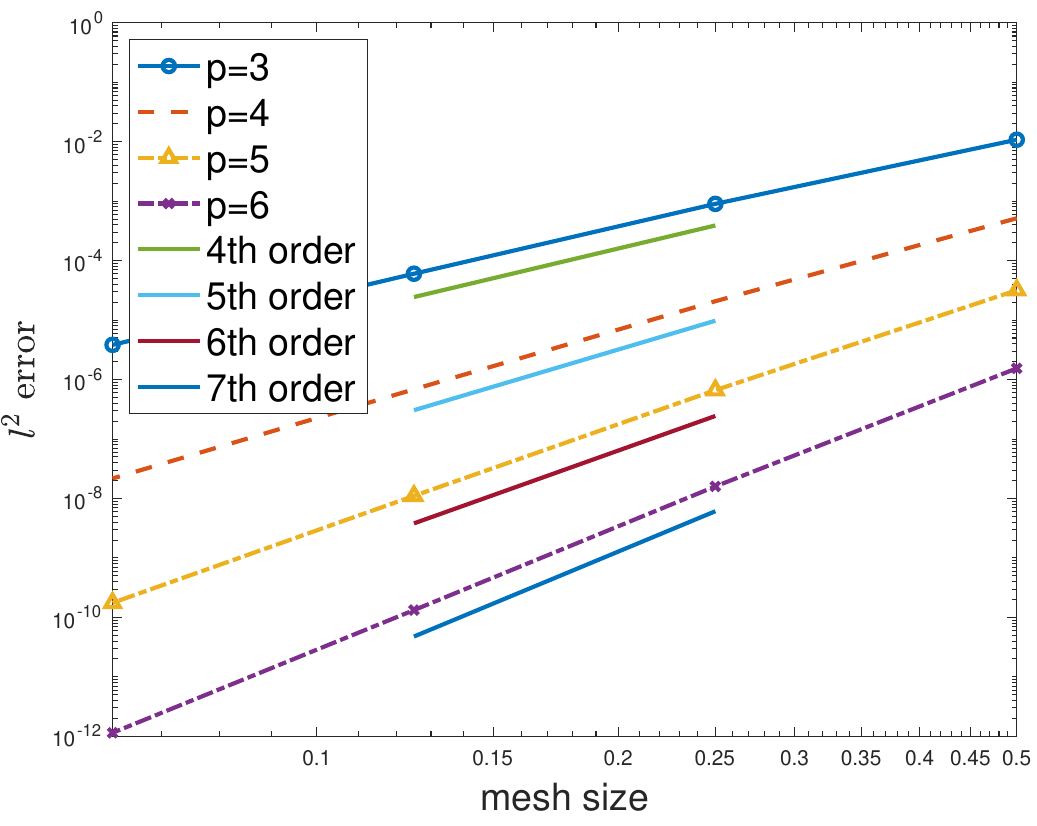}}
\subfigure[p-refinement]{
\includegraphics[width=0.35\textwidth]{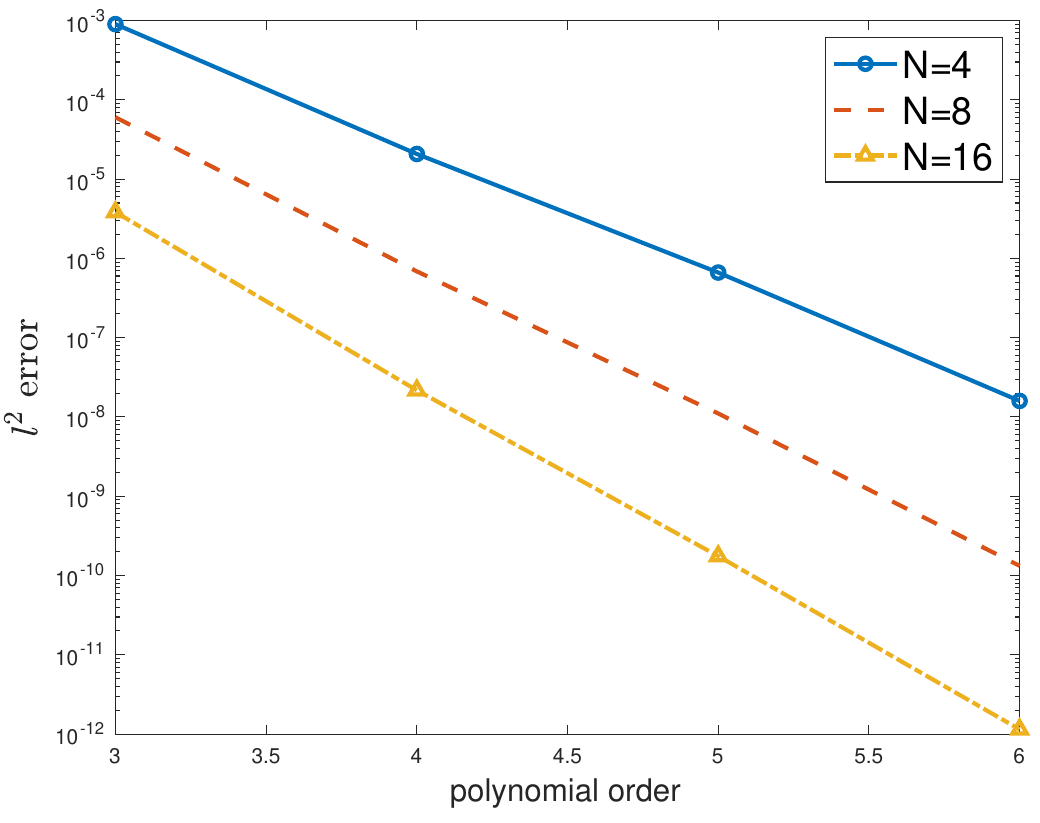}}
\caption{1D Helmholtz equation: $l^{2}$ errors of $C^1$ FCEs as a function of (a) the element size, and (b) the polynomial order. In (a), the polynomial order ($p$) is fixed while the element size is varied (h-refinement), showing a convergence rate of ($p+1$). In (b), the number of elements ($N$) is fixed while the polynomial order $p$ is varied (p-refinement), showing an exponential convergence rate. $(p+2)$ collocation points per element (Gauss-Lobatto-Legendre points).
} 
\label{fg_a1}
\end{figure}

\begin{table}[tb]
\centering
\begin{tabular}{l|r| lll}
\hline
 collocation points & polynomial order & FCE-$C^1$ &  FCE-$C^0$  & FCE-NC  \\  \hline
 Gauss-Lobatto-Legendre & $p=3$   & 9.01E-4  & 1.15E-3  & 3.35E-2  \\
  points & $4$  & 2.08E-5  & 2.15E-5 &  2.09E-5 \\
 & $5$  & 6.63E-7  & 6.63E-7  & 6.63E-7 \\
 & $6$  & 1.59E-8  & 1.59E-8  & 1.59E-8 \\ \hline
 uniform points & $p=3$  & 1.91E-2  & 2.10E-2 &  3.18E-2 \\
   & $4$  & 3.73E-4  & 3.83E-4 &  3.78E-4 \\
 & $5$ &  4.02E-5&  4.02E-5  & 4.02E-5 \\
 & $6$ &  5.01E-7  & 5.01E-7  & 5.01E-7 \\
\hline
\end{tabular}
\caption{1D Helmholtz equation:  $l^2$ errors computed by $C^1$ and $C^0$ FCEs and FCE-NCs, corresponding to several polynomial orders ($p$) with different types of collocation points. $N=4$ uniform elements in domain, and $q=p+2$ collocation points per element.
}
\label{tab_a1}
\end{table}

\paragraph{Helmholtz Equation}

Consider the domain $\Omega = [0,1]$ and the following BVP with the Helmholtz equation on $\Omega$,
\begin{subequations}\label{eq_142}
\begin{align}
& \frac{d^2u}{dx^2} - u = -(1+\pi^2)\cos(\pi x), \quad x\in\Omega, \\
& u(0) = 1, \quad
\left.\frac{du}{dx}\right|_{x=1} = 0,
\end{align}
\end{subequations}
where $u(x)$ is to be computed. This problem has an exact solution $u_{ex}(x) = \cos(\pi x)$.

To solve this problem, we partition $\Omega$ into $N$ uniform elements and impose $C^1$ continuity  across the element boundaries (see Remark~\ref{rem_6}). We compute the solution by using 1D $C^1$ and $C^0$ FCEs (denoted by ``FCE-$C^1$'' and ``FCE-$C^0$''), as well as FCEs with no  continuity (denoted by ``FCE-NC''). The free functions involved in the FCE formulations are represented by Legendre polynomials in each element. We use $p$ to denote the highest degree of the Legendre polynomial, which will be referred to as the polynomial order hereafter. Within each element we employ $(p+2)$ collocation points with two types of distributions: (i) uniform grid points, and (ii) Gauss-Lobatto-Legendre quadrature points.  

After the problem is solved, we evaluate the computed solution $u(x)$ on a finer set of $Q_u$ uniform grid points per element, as well as on a finer set of $Q_q$ Gauss-Lobatto-Legendre quadrature points in each element. Then we compute the $l^{\infty}$  and the $l^2$ errors of $u(x)$ as follows,
\begin{equation}
l^{\infty}\text{-error} = \max_{\substack{1\leqslant i\leqslant Q_u,\\1\leqslant e\leqslant N}}\{ |u(x_i^e)-u_{ex}(x_i^e)| \}, \quad
l^2\text{-error} = \sqrt{
\sum_{e=1}^N\sum_{i=1}^{Q_{q}}w_{i}[u(y_i^e)-u_{ex}(y_i^e)]^2
},
\end{equation}
where $x_i^e$ ($1\leqslant i\leqslant Q_u$) denote the finer uniform grid points in element $e$ ($1\leqslant e\leqslant N$), $y_i^e$ ($1\leqslant i\leqslant Q_q$) denote the finer Gauss-Lobatto-Legendre points (with quadrature weight $w_i$) within element $e$, and $u_{ex}(x)$ denotes the exact solution. We have employed $Q_u=20$ and $Q_q=12$ in the simulations.
We will refer to the resolution increase due to the decrease in the element size (while the polynomial order is fixed) as the h-refinement, and that due to the increase in the polynomial order (while the element size is fixed) as p-refinement, following conventions from the traditional spectral or hp finite element techniques.

Figure~\ref{fg_a1} illustrates the convergence behavior of the method with $C^1$ FCEs during the h-refinement  and the p-refinement. The two plots show the $l^{2}$ errors as a function of the element size (h-refinement) and the polynomial order in each element (p-refinement), respectively. Note that with $C^1$ FCEs, the continuity of the computed solution $u(x)$ and its derivative $u'(x)$ across the element boundaries is exactly satisfied. For the h-refinement we observe  $(p+1)$-th order convergence rate with increasing number of elements. For the p-refinement, we observe the exponential convergence rate with increasing polynomial order.

In Table~\ref{tab_a1} we compare the $l^{\infty}$ and $l^2$ errors obtained using $C^1$ and $C^0$ FCEs and FCEs with no intrinsic continuity (FCE-NC), corresponding to $N=4$ uniform elements and several polynomial orders per element, with quadrature points and uniform grid points as the collocation points. Note that with $C^0$ FCEs the continuity of the computed solution $u(x)$ across the element boundaries is satisfied exactly, while the continuity of its derivative $u'(x)$ is only enforced in the least squares sense across the element boundaries. With FCE-NCs, on other hand, the continuity of both $u(x)$ and $u'(x)$ is only enforced in the least squares sense across the element boundaries. We observe that the errors decrease exponentially as the polynomial order increases, with both uniform collocation points and quadrature collocation points, for all three types of FCEs. The error values obtained using the three types of FCEs are comparable for this problem. The error levels obtained using quadrature collocation points are notably lower than those obtained using uniform collocation points.

\begin{table}[tb]
\centering
\begin{tabular}{l | ll|ll}
\hline
collocation & quadrature & points & uniform & points \\ \cline{2-5}
points ($q$) & FCE-$C^0$ & FCE-NC & FCE-$C^0$ & FCE-NC \\ \hline
$3$ & 6.93E-1 & 7.18E-1 & 6.93E-1 & 7.18E-1 \\
$4$ & 1.53E-3 & 1.53E-3 & 6.64E-1 & 6.41E-1 \\
$5$ & 2.84E-7 & 2.84E-7 & 3.60E-7 & 3.60E-7 \\
$6$ & 1.08E-7 & 1.08E-7 & 2.85E-7 & 2.85E-7 \\
$7$ & 1.08E-7 & 1.08E-7 & 2.54E-7 & 2.54E-7 \\
\hline
\end{tabular}
\caption{IVP: $l^2$ errors obtained with FCE-$C^0$ and FCE-NC as a function of the number of collocation points ($q$), with Gauss-Lobatto-Legendre quadrature points and uniform grid points as the collocation points. $N=4$ uniform elements, and polynomial order $p=5$ for each element.
}
\label{tab_2}
\end{table}

\begin{figure}[tb]
\centering 
\subfigure[h-refinement]{
\includegraphics[width=0.35\textwidth]{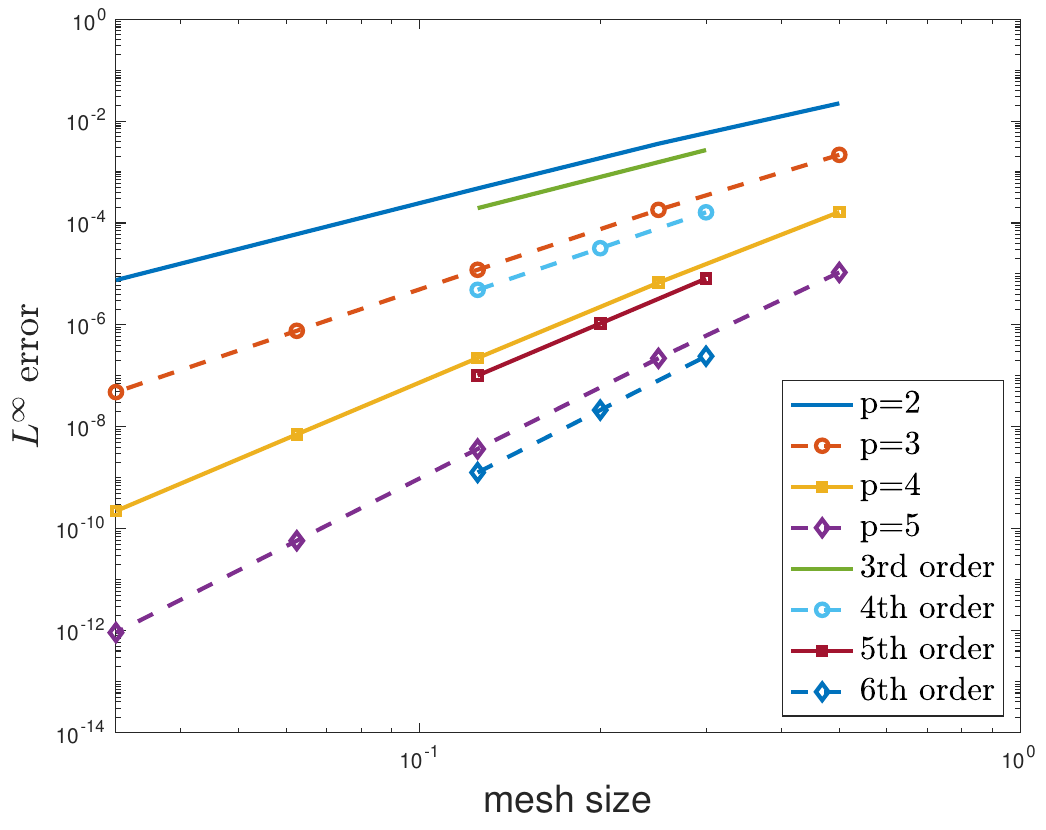}}
\subfigure[p-refinement]{
\includegraphics[width=0.35\textwidth]{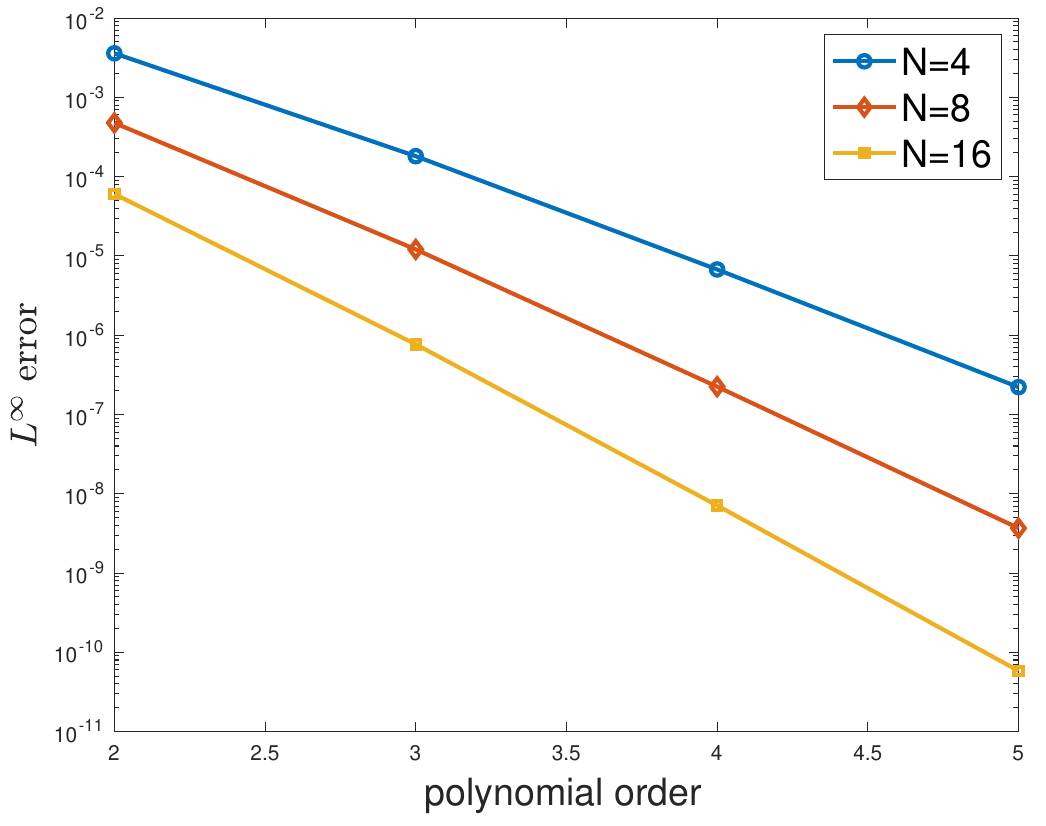}}
\caption{IVP: $l^{\infty}$ errors of  FCE-$C^0$ as a function of (a) the element size, and (b) the polynomial order. $q=p+2$ Gauss-Lobatto-Legendre  collocation points per element.
} 
\label{fg_2}
\end{figure}

\begin{table}[tb]
\centering
\begin{tabular}{l|l| ll| ll}
\hline
 collocation & polynomial &  FCE-$C^0$ & & FCE-NC &  \\ \cline{3-6}
 points & order & $l^{\infty}$-error & $l^2$-error & $l^{\infty}$-error & $l^2$-error  \\ \hline
 quadrature & $2$ & 1.95E-3 & 3.59E-3 & 1.95E-3 & 3.60E-3 \\
 points & $3$ &  8.86E-5 & 1.81E-4 & 8.86E-5 & 1.81E-4 \\
  & $4$ &  3.36E-6 & 6.74E-6 & 3.36E-6 & 6.74E-6 \\
 & $5$ &  1.08E-7 & 2.21E-7 & 1.08E-7 & 2.21E-7\\
  \hline
 uniform & $2$ &  9.22E-3 & 1.41E-2 & 8.70E-3 & 1.68E-2 \\
 points & $3$ &  2.34E-4 & 5.48E-4 & 2.41E-4 & 5.62E-4 \\
 & $4$ &  1.65E-5 & 2.86E-5 & 1.65E-5 & 2.86E-5 \\
 & $5$ &  2.54E-7 & 6.03E-7 & 2.54E-7 & 6.03E-7 \\
\hline
\end{tabular}
\caption{IVP: $l^{\infty}$ and $l^2$ errors computed by $C^0$ FCEs and FCE-NCs, corresponding to several polynomial orders ($p$) with different types of collocation points. $N=4$ uniform elements, and $q=p+2$ collocation points per element.
%{\color{red}$l^{\infty}$ errors computed on 20 uniform grid points (per direction) in each element. $l^2$ errors computed using $12$ Gauss-Lobatto-Legendre quadrature points for the integrals for all problems.}
%{\color{red}[Need $l^{\infty}$ errors for this table.]}
}
\label{tab_3}
\end{table}

\paragraph{An Initial Value Problem}

We consider the initial value problem (IVP),
\begin{subequations}\label{eq_143}
\begin{align}
& \frac{du}{dt} + u = e^{\sin(\pi t)}[1 + \pi\cos(\pi t)], \quad t\in\Omega = [0,1], 
\label{eq_143a} \\
& u(0) = 1, \label{eq_143b}
\end{align}
\end{subequations}
where $u(t)$ is the solution to be computed. This problem has an exact solution $u_{ex}(t)=e^{\sin(\pi t)}$.

% how to solve the problem

To solve the problem~\eqref{eq_143}, we treat it as a ``boundary'' value problem, with the condition~\eqref{eq_143b}  applied to the boundary $t=0$. We partition the domain $\Omega=[0,1]$ into $N$ uniform elements, and impose $C^0$ continuity  on $u(t)$ across the element boundaries. We employ 1D $C^0$ FCEs and FCE-NCs to solve the problem. 
Note that with $C^0$ FCEs the $C^0$ continuity across the element boundaries is enforced exactly and with FCE-NCs it is enforced only in the least squares sense (i.e.~approximately). 
We have used both the Gauss-Lobatto-Legendre quadrature points and the uniform grid points as the collocation points within each element. The number of collocation points and the polynomial order within each element is denoted by $q$ and $p$, respectively.

% tests on collocation points

Table~\ref{tab_2} illustrates the effect of the number of collocation points on the accuracy. It shows the $l^2$ errors of FCE-$C^0$ and FCE-NC corresponding to a range of collocation points, with both Gauss-Lobatto-Legendre collocation points and unform collocation points. In these tests the number of elements is fixed to $N=4$, and the polynomial order is fixed to $p=5$. As $q<p$ the  result is observed to be not accurate or less accurate. With $q>p$ we observe an accurate result, and the accuracy remains essentially the same as $q$ further increases. These characteristics appear common to the test problems  when polynomial bases are employed to represent the free functions in FCE. Hereafter we will generally employ $q=p+2$ collocation points with the FCE method.

Figure~\ref{fg_2} illustrates the convergence of the FCE method with $C^0$ FCEs for solving~\eqref{eq_143} during the $h$-refinement and $p$-refinement. The plots show the $l^{\infty}$ errors of FCE-$C^0$ as a function of the element size (Figure~\ref{fg_2}(a)) and the polynomial order (Fgiure~\ref{fg_2}(b)). We observe the $(p+1)$-th order convergence rate with respect to the element size and the exponential rate with respect to the polynomial order. 

Table~\ref{tab_3} lists the $l^{\infty}$ and $l^2$ errors obtained with FCE-$C^0$ and FCE-NC corresponding to a range of polynomial orders, using Gauss-Lobatto-Legendre quadrature points and uniform grid points as the collocation points. 
The accuracy  obtained with these two types of FCEs is comparable.
The results obtained using uniform collocation points are generally less accurate than those with quadrature collocation points.

\begin{figure}[tb]
\centering 
\subfigure[h-refinement]{
\includegraphics[width=0.35\textwidth]{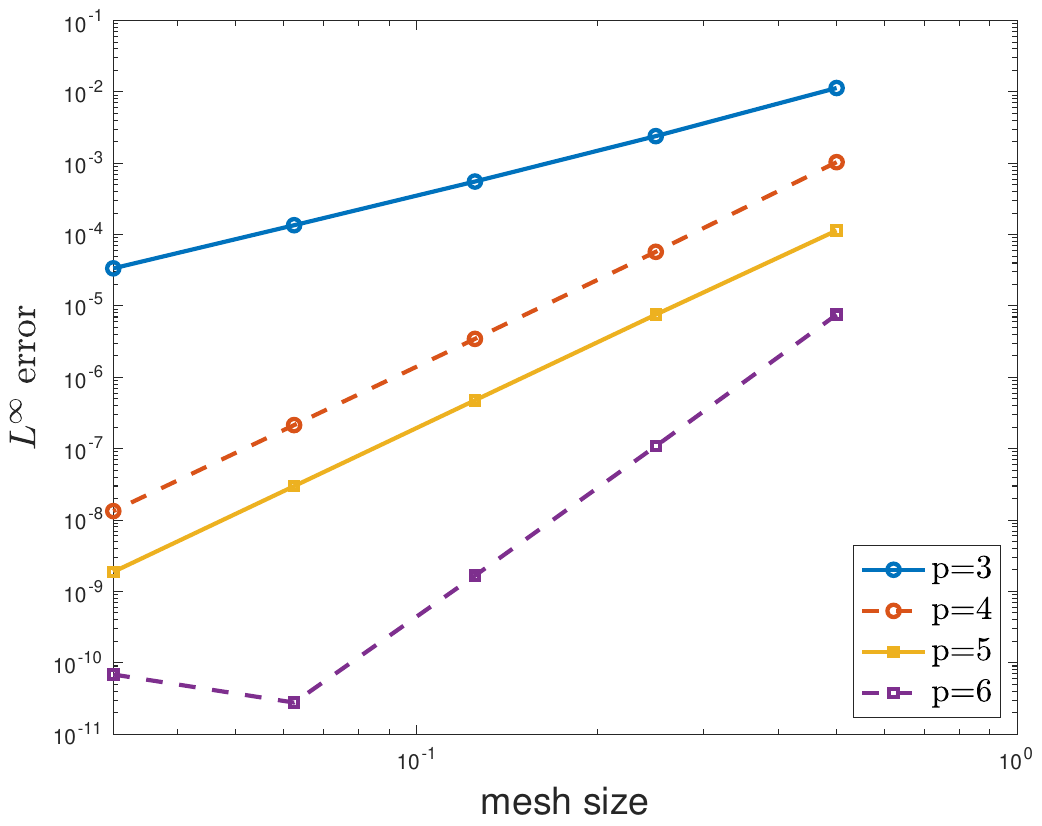}}
\subfigure[p-refinement]{
\includegraphics[width=0.35\textwidth]{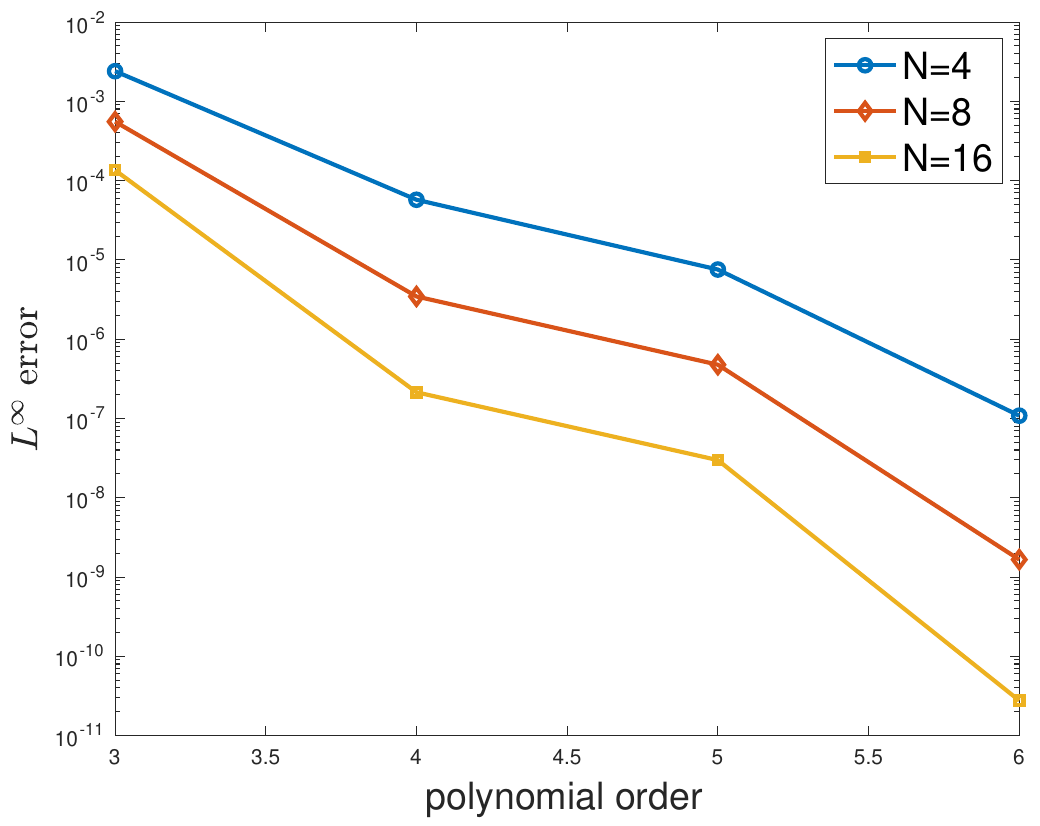}}
\caption{1D nonlinear Helmholtz equation: $l^{\infty}$ errors of FCE-$C^1$ as a function of (a) the element size, and (b) the polynomial order. $N$ uniform elements, and $q=p+2$ uniform collocation points per element.
} 
\label{fg_3}
\end{figure}

\paragraph{Nonlinear Helmholtz Equation}

We next consider the   BVP with the nonlinear Helmholtz equation,
\begin{subequations}
\begin{align}
& \frac{d^2u}{dx^2} - u + \sin(u) = f(x), \quad x\in\Omega=[0,1], \\
& u(0) = u(1) = 1,
\end{align}
\end{subequations}
where the source term $f(x)$ is chosen such that this BVP has the exact solution $u_{ex}(x)=1+\frac12\sin(\pi x)$.

We partition $\Omega$ into $N$ uniform elements, and impose $C^1$ continuity  across the element boundaries. We employ 1D $C^1$ FCEs to solve the problem, and use $q=p+2$ uniform collocation points in each element, where $p$ denotes the  order of the Legendre polynomials for representing the FCE free functions. After enforcing the problem on the  collocation points, the resultant nonlinear algebraic system is solved by the nonlinear least squares method (see Remark~\ref{rem_4}).

Figure~\ref{fg_3} illustrates the convergence behavior of the FCE method for solving this nonlinear problem. The two plots show the $l^{\infty}$ errors of FCE-$C^1$ as a function of the element size (h-refinement) and the polynomial order (p-refinement). The FCE method has evidently captured the solution to this nonlinear problem accurately.

\subsubsection{2D Examples}

\begin{table}
\centering
\begin{tabular}{llll} \hline 
  $p$  & FCE-$C^1$ & FCE-$C^0$ & FCE-NC  \\ \hline
 $3$& 1.50E-2& 5.04E-2  & 1.36E-1  \\ 
 $5$&2.30E-4& 6.24E-4  & 1.84E-3  \\ 
 $7$&1.90E-6& 5.39E-6  & 3.27E-5  \\ 
 $9$&1.10E-8& 3.07E-8  & 1.92E-7  \\ 
 $11$&4.05E-11& 1.18E-10  & 7.30E-10  \\ \hline
\end{tabular}\par\smallskip
\caption{2D Helmholtz equation: $l^{\infty}$ errors of 2D FCE-$C^1$, FCE-$C^0$ and FCE-NC corresponding to a range of polynomial orders. $(N_x,N_y)=(2,1)$ elements, and $q=p+2$ Gauss-Lobatto-Legendre collocation points in each direction.
%(least squares computed by lscov)
}
\label{tab_4}
\end{table}

%%%%%%%%%%%%%%%%%%%%%%
\begin{comment}
\begin{table}
\centering
\begin{tabular}{llll} \hline 
  $p$  & FCE-$C^1$ & FCE-$C^0$ & FCE-NC  \\ \hline
 $3$& 7.68E-3& 1.32E-2  & 2.89E-2 \\ 
 %$p=4$&2.354628e-2& 1.655190e-2  & 1.367321e-1  \\ 
 $5$&1.36E-4& 1.64E-4  & 1.64E-4  \\ 
 $7$&1.09E-6& 1.23E-6  & 1.10E-6  \\ 
 $9$&6.23E-9& 6.25E-9  & 6.04E-9  \\ 
 $11$&2.24E-11& 2.58E-11  & 2.49E-11  \\ \hline
\end{tabular}\par\smallskip
\caption{\color{red}2D Helmholtz equation: $l^{\infty}$ errors obtained by 2D FCE-$C^1$, FCE-$C^0$ and FCE-NC corresponding to different polynomial orders. $(N_x,N_y)=(2,1)$; $q=p+2$ Gauss-Lobatto-Legendre collocation points in each direction.(lsqr)
%For problem\eqref{eq 47}. $L_{\infty}$ error comparison between 2D hard $C^1$ element,2D mix $C^1$ element, 2D soft $C^1$ element. We fix the elements as $n_x=2,n_y=1$, and use Guass Lobatto collocation points in all cases, and number of collocation points is $p+2$.
}
\label{tab_4b}
\end{table}
\end{comment}
%%%%%%%%%%%%%%%%%%%%%%%%%

\begin{figure}[tb]
\centering 
\subfigure[h-refinement]{
\includegraphics[width=0.35\textwidth]{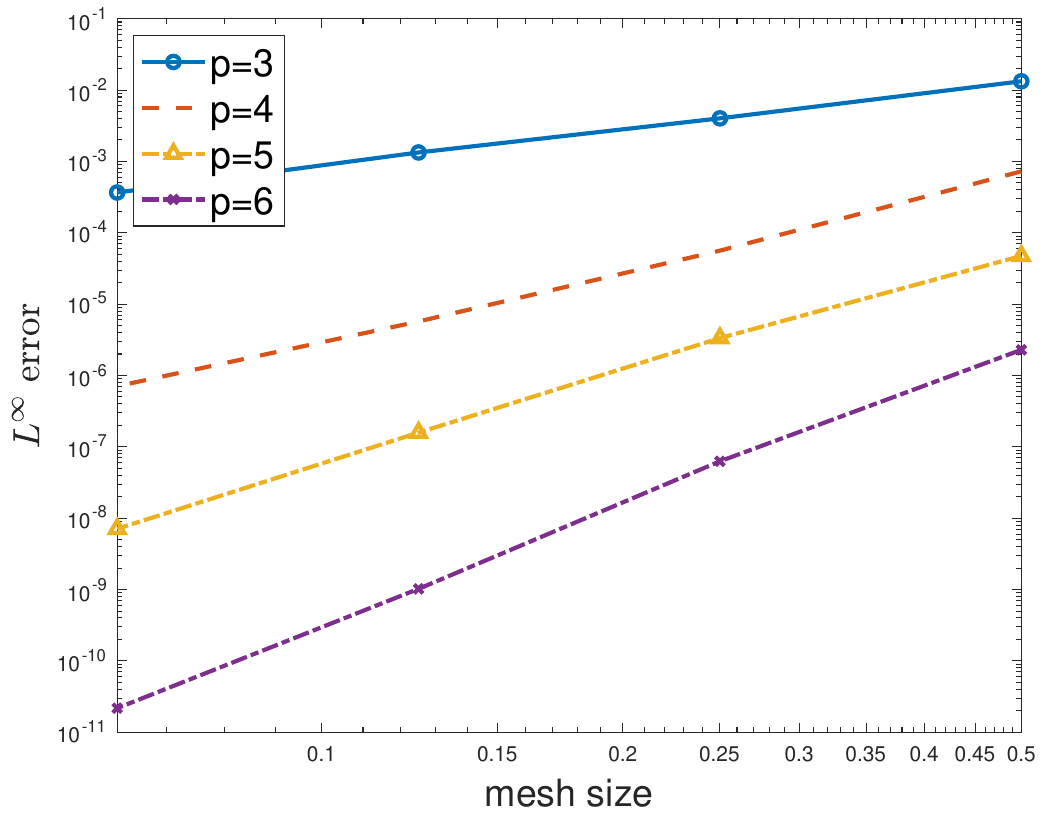}}
\subfigure[p-refinement]{
\includegraphics[width=0.35\textwidth]{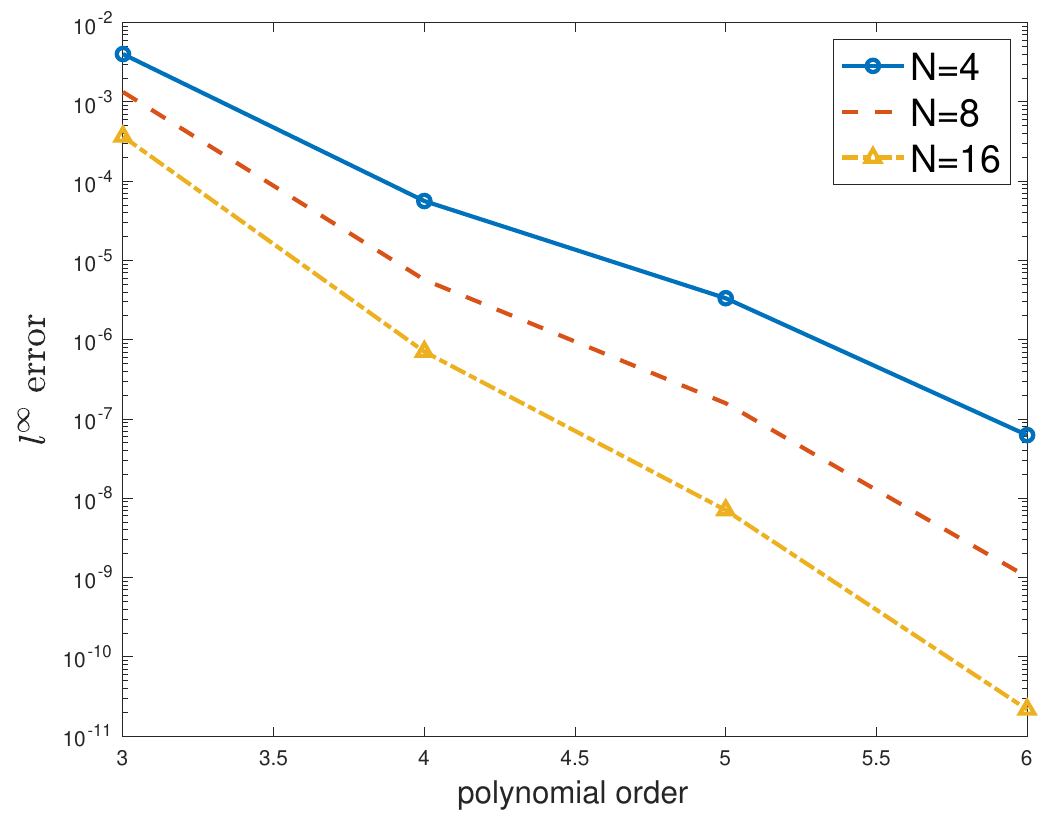}}
\caption{2D Helmholtz equation: $l^{\infty}$ errors of FCE-$C^0$ versus (a) the element size in each direction, and (b) the polynomial order. $N$ uniform elements per direction. 
%and $q=p+2$ Gauss-Lobatto-Legendre collocation points per direction.
%For problem\eqref{eq 47}, 2D mix $C^1$ element convergence.  In all cases, we use Guass Lobatto collocation points with number equal to $p + 2$.
} 
\label{fg_4}
\end{figure}

\begin{table}[tb]
\centering
\begin{tabular}{l|ll|ll} \hline 
polynomial & $(N_x,N_y)=(4,4)$ & & $(N_x,N_y)=(8,8)$ & \\ \cline{2-5}
 order ($p$)  & FCE-$C^0$ & FCE-NC & FCE-$C^0$ & FCE-NC  \\ \hline
 3 & 4.02E-3 & 2.27E-3 & 1.34E-3 & 3.68E-3  \\ 
 4 & 5.63E-5 & 1.29E-3 & 5.71E-6 & 8.30E-5 \\
 5 & 3.34E-6 & 2.18E-5 & 1.58E-7 & 7.88E-7 \\
 6 & 6.27E-8 & 1.22E-6 & 1.02E-9 & 1.91E-8 \\
 \hline
\end{tabular}\par\smallskip
\caption{2D Helmholtz equation: $l^{\infty}$ errors of FCE-$C^0$ and FCE-NC corresponding to several polynomial orders. 
%$(N_x,N_y)=(4,4)$ and $(8,8)$; $q=p+2$ Gauss-Lobatto-Legendre collocation points in each direction.
}
\label{tab_5}
\end{table}

\paragraph{2D Helmholtz Equation}

Consider the 2D domain $\Omega=[0,1]\times[0,1]$, and the following BVP on $\Omega$ with the 2D Helmholtz equation,
\begin{subequations}\label{eq_146}
\begin{align}
& \nabla^2 u - u = f(x,y), \quad
(x,y)\in\Omega, \\
& u(0,y) = g_1(y), \quad u(1,y) = g_2(y), \quad u(x,0) = g_3(x), \quad u(x,1) = g_4(x),
\end{align}
\end{subequations}
where $u(x,y)$ is the field function to solve, and $f$ and $g_i$ ($1\leqslant i\leqslant 4$) are prescribed source term or boundary distributions. We choose $f$ and $g_i$ such that this problem has the exact solution $u_{ex}(x,y)=\sin(\pi x)\cos(\pi y)$. 

We partition $\Omega$ into $N_x$ and $N_y$ uniform elements in the $x$ and $y$ directions, respectively.  We impose $C^1$ continuity  across the element boundaries along both $x$ and $y$ directions. 2D $C^1$ and $C^0$ FCEs and FCE-NCs are employed to solve the problem. 
%The free functions involved in the FCE formulation are represented in terms of the Legendre polynomials. 
Let $p$  denote the polynomial order of the Legendre polynomials in each direction, and $q$  denote the number of collocation points in each direction. Gauss-Lobatto-Legendre quadrature points are used as the collocation points. 

After the problem is solved, we evaluate the FCE solution on a finer set of $Q_u\times Q_u$  uniform grid points within each element, and compare the FCE solution with the exact solution on these points to obtain the $l^{\infty}$ error on $\Omega$. $Q_u=20$ is used in the errors reported below.

Table~\ref{tab_4} lists the $l^{\infty}$ errors on a mesh of $(N_x,N_y)=(2,1)$ elements obtained using 2D FCE-$C^0$, FCE-$C^1$ and FCE-NC, corresponding to a range of polynomial orders. 
%The number of collocation points is $q=p+2$ in the simulations. 
The errors decrease exponentially with respect to the polynomial order ($p$) for all three types of elements.  
The FCE-$C^1$ results are more accurate than those of FCE-$C^0$, which in turn are more than accurate than those of FCE-NC.
%The FCE-$C^0$ and FCE-$C^1$ results are notably more accurate than those of FCE-NC. A comparison between  FCE-$C^0$ and FCE-$C^1$  indicates that the FCE-$C^1$ results are generally better than those of FCE-$C^0$. 

Simulation results obtained on larger meshes are illustrated by Figure~\ref{fg_4} and Table~\ref{tab_5}. Figure~\ref{fg_4} shows the $l^{\infty}$ errors of FCE-$C^0$ as a function of the element size (plot (a)), and as a function of the polynomial order (plot (b)). The mesh size ranges from $(N_x,N_y)=(2,2)$ to $(N_x,N_y)=(16,16)$ in Figure~\ref{fg_4}(a). One can observe a rapid convergence in the simulation results as the element size or the polynomial order increases.
Table~\ref{tab_5} lists the $l^{\infty}$ errors obtained by FCE-$C^0$ and FCE-NC on two meshes, with $(N_x,N_y)=(4,4)$ and $(8,8)$, corresponding to a range of polynomial orders. One can observe the exponential decrease in the errors with increasing polynomial order, with
 FCE-$C^0$  generally more accurate than  FCE-NC.

\paragraph{Advection Equation}

\begin{figure}[tb]
\centering 
\subfigure[h-refinement]{
\includegraphics[width=0.35\textwidth]{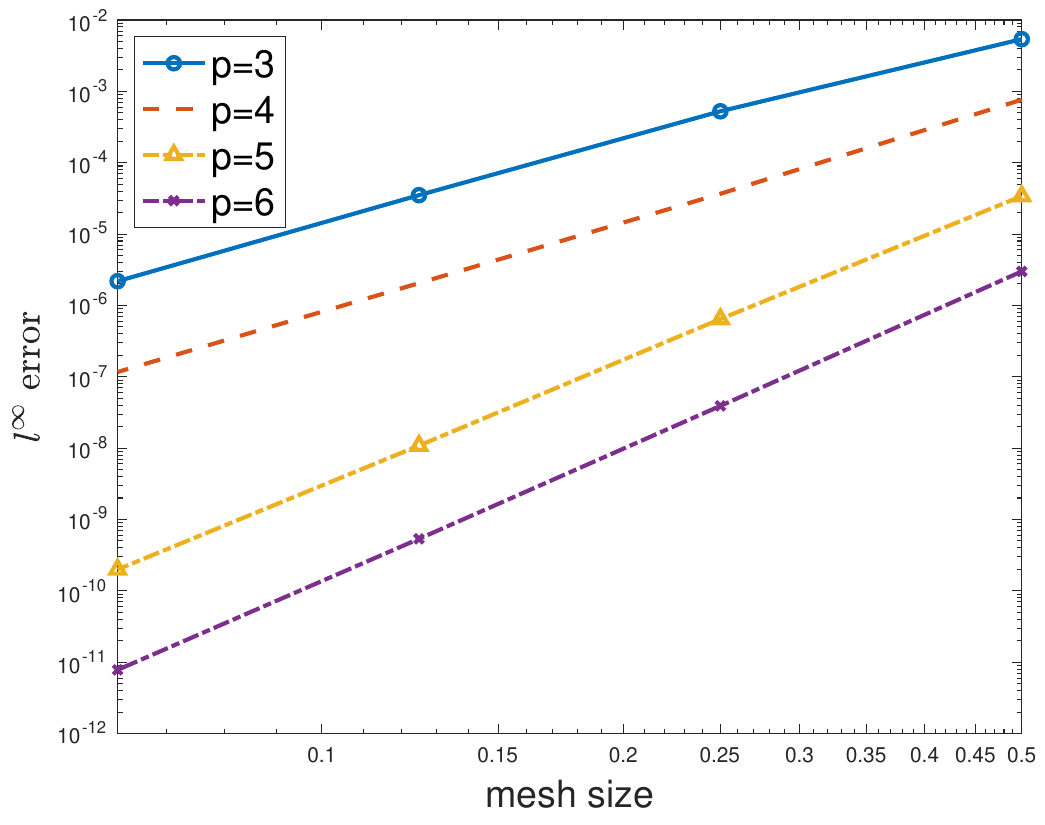}}
\subfigure[p-refinement]{
\includegraphics[width=0.35\textwidth]{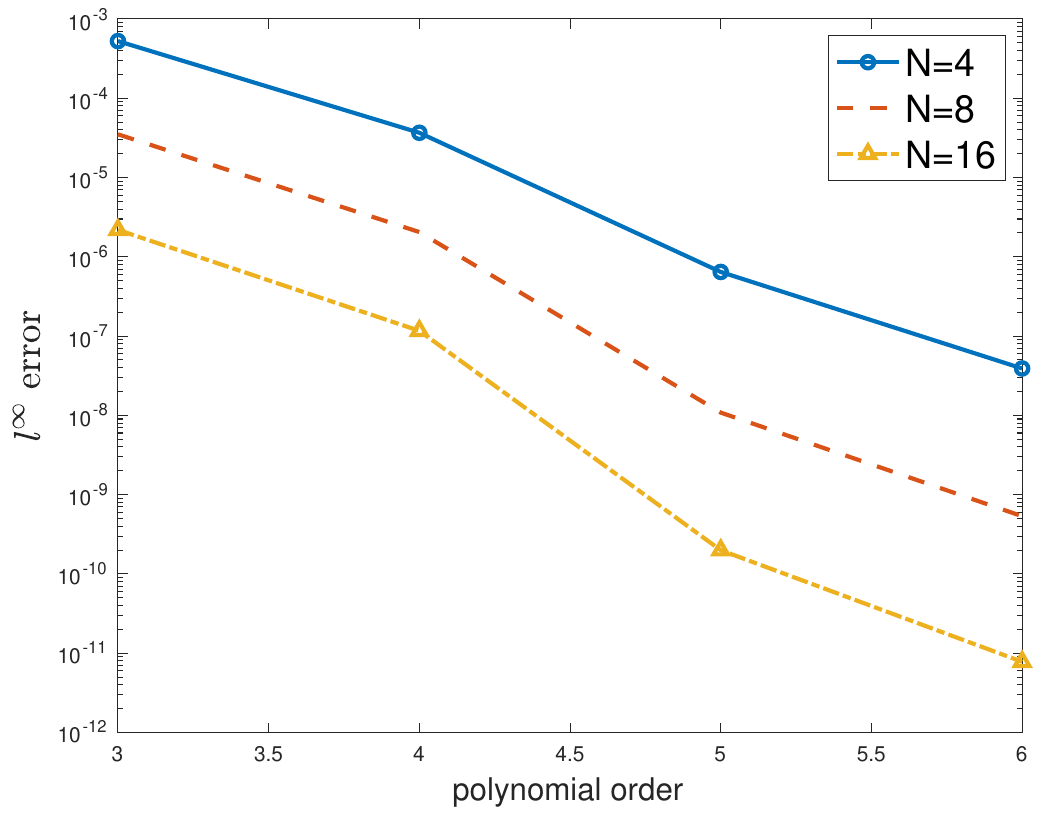}}
\caption{Advection equation: $l^{\infty}$ errors of FCE-$C^0$ versus (a) the element size in each direction, and (b) the polynomial order. $q=p+2$ Gauss-Lobatto-Legendre collocation points; $N$ denotes the number of elements per direction.
%in each direction within each element.
} 
\label{fg_5}
\end{figure}

\begin{table}[tb]
\centering
\begin{tabular}{l|l|ll|ll} \hline 
collocation & polynomial & $(N_x,N_t)=(4,4)$ & & $(N_x,N_t)=(8,8)$ & \\ \cline{3-6}
points & order ($p$)  & FCE-$C^0$ & FCE-NC & FCE-$C^0$ & FCE-NC  \\ \hline
Gauss-Lobatto- & 3 & 5.26E-4 & 1.12E-3 & 3.51E-5 & 8.23E-5  \\ 
Legendre points & 4 & 3.67E-5 & 4.84E-5 & 2.05E-6 & 2.46E-6 \\
& 5 & 6.45E-7 & 1.22E-6 & 1.08E-8 & 3.56E-8 \\
& 6 & 3.90E-8 & 4.88E-8 & 5.35E-10 & 9.05E-10 \\
\hline
uniform points & 3 & 9.25E-4 & 1.10E-3 & 6.24E-5 & 7.45E-5 \\
& 4 & 7.26E-5 & 9.89E-5 & 3.77E-6 & 5.10E-6 \\
& 5 & 1.27E-6 & 2.58E-6 & 2.52E-8 & 4.75E-8 \\
& 6 & 6.80E-8 & 1.22E-7 & 9.64E-10 & 1.64E-9 \\
 \hline
\end{tabular}\par\smallskip
\caption{Advection equation: $l^{\infty}$ errors of 2D FCE-$C^0$ and FCE-NC versus the polynomial orders, with Gauss-Lobatto-Legendre and uniform collocation points. 
%$(N_x,N_y)=(4,4)$ or $(8,8)$ elements; 
%$q=p+2$ collocation points in each direction within each element.
%{\color{red}[Need data for $p=6$ for this table.]}
}
\label{tab_6}
\end{table}

%We next test the FCE method with the advection equation in one spacial dimension. 
Consider the spatial-temporal domain $(x,t)\in\Omega=[0,1]\times [0,1]$ and the initial/boundary value problem,
\begin{subequations}\label{eq_147}
\begin{align}
& \frac{\partial u}{\partial t} + 2\frac{\partial u}{\partial x} = f(x,t), \quad (x,t)\in\Omega \label{eq_147a} \\
& u(0,t) = g_1(t), \quad t\in[0,1], \\
& u(x,0) = g_2(x), \quad x\in[0,1],
\end{align}
\end{subequations}
where $u(x,t)$ is the field to be computed.  The source term $f(x,t)$, and the boundary and initial data $g_1(t)$ and $g_2(x)$, are chosen such that this problem has an exact solution $u_{ex}(x,t) = e^{\cos(\pi x)}\sin(\pi t)$. 

% how to solve problem?

We solve~\eqref{eq_147} by the space-time approach, and treat the time variable $t$ in the same way as the spatial coordinate $x$. So this is effectively a 2D problem with respect to $(x,t)$. We partition $\Omega$ into uniform elements, with $N_x$ elements in $x$ and $N_t$ elements in $t$. Since equation~\eqref{eq_147a} is first order in both $x$ and $t$, we impose $C^0$ continuity  across the element boundaries in both  directions. We use 2D $C^0$ FCEs  and FCE-NCs  to solve the resultant system. 
%With FCE-$C^0$ the $C^0$ continuity conditions across the elements are satisfied exactly, while with FCE-NC these continuity conditions are satisfied only in the least squares sense. 
Both Gauss-Lobatto-Legendre type and uniform collocation points are tested, and the number of collocation points is set to be $q=p+2$, where $p$ denotes the polynomial order.  

Figure~\ref{fg_5} illustrates the convergence behavior of the FCE-$C^0$ method for solving the advection equation. Figure~\ref{fg_5}(a) shows the $l^{\infty}$ errors  as a function of the element size corresponding to several element orders. 
We observe a convergence with decreasing element size. The rate of convergence appears not quite regular. The order of convergence in the h-refinement is effectively $4$ with both polynomial orders $p=3$ and $p=4$, and it is effectively $6$ with both polynomial orders $p=5$ and $p=6$.
Figure~\ref{fg_5}(b) shows the $l^{\infty}$ errors of FCE-$C^0$ as a function of the polynomial order $p$ for several mesh sizes. The errors are observed to decrease essentially exponentially with increasing polynomial order.

Table~\ref{tab_6} lists the $l^{\infty}$ errors obtained with  FCE-$C^0$ and FCE-NC on two meshes, $(N_x,N_t)=(4,4)$ and $(8,8)$, corresponding to several polynomial orders. The results computed with both Gauss-Lobatto-Legendre and uniform collocation points are provided in the table. The FCE-$C^0$ errors and the FCE-NC errors decrease exponentially with increasing polynomial order, obtained by both Gauss-Lobatto-Legendre and uniform collocation points. FCE-$C^0$ is   in general   more accurate than  FCE-NC, and the results obtained with  Gauss-Lobatto-Legendre  points are  more accurate than those obtained with uniform collocation points.

\subsection{Free Functions Represented by Non-Polynomial Bases}

The FCE formulation in Section~\ref{sec_3} ensures that the $C^0$ or $C^1$ continuity  across the element boundaries is exactly satisfied, irrespective of the free functions or  free parameters involved therein. This enables the use of  a function space other than the  polynomial space for representing the free functions in  FCE. In this subsection we show  simulation results obtained by representing the FCE free functions with sinusoidal functions within each element for solving boundary and initial value problems. We would like to emphasize that Dirichlet type BCs are involved in  the overall problem, and yet the free functions in each element are expanded in terms of a set of sinusoidal functions. We use 1D problems in this subsection for simplicity. 

Specifically, we employ the following bases for the free functions in 1D $C^0$ and $C^1$ FCEs,
\begin{equation}\label{eq_148}
\mathcal F_p([a,b]) = \text{span}\{\ 
\sin(\xi_i\phi_1(a,b,x) + \eta_i),\ 
\text{where}\ \xi_i = 2\sqrt{i+1},\ 
\eta_i = \sin(i+1)+0.1,\ 
0\leqslant i\leqslant p-1
\ \},
\end{equation}
where $\phi_1(\chi_1,\chi_2,x)$ is defined in~\eqref{eq_a60}, and $p$ denotes the number of basis functions and will be referred to as the element order in this subsection.

\subsubsection{1D Poisson Equation}

\begin{figure}[tb]
\centering 
\subfigure[h-refinement]{
\includegraphics[height=0.3\textwidth]{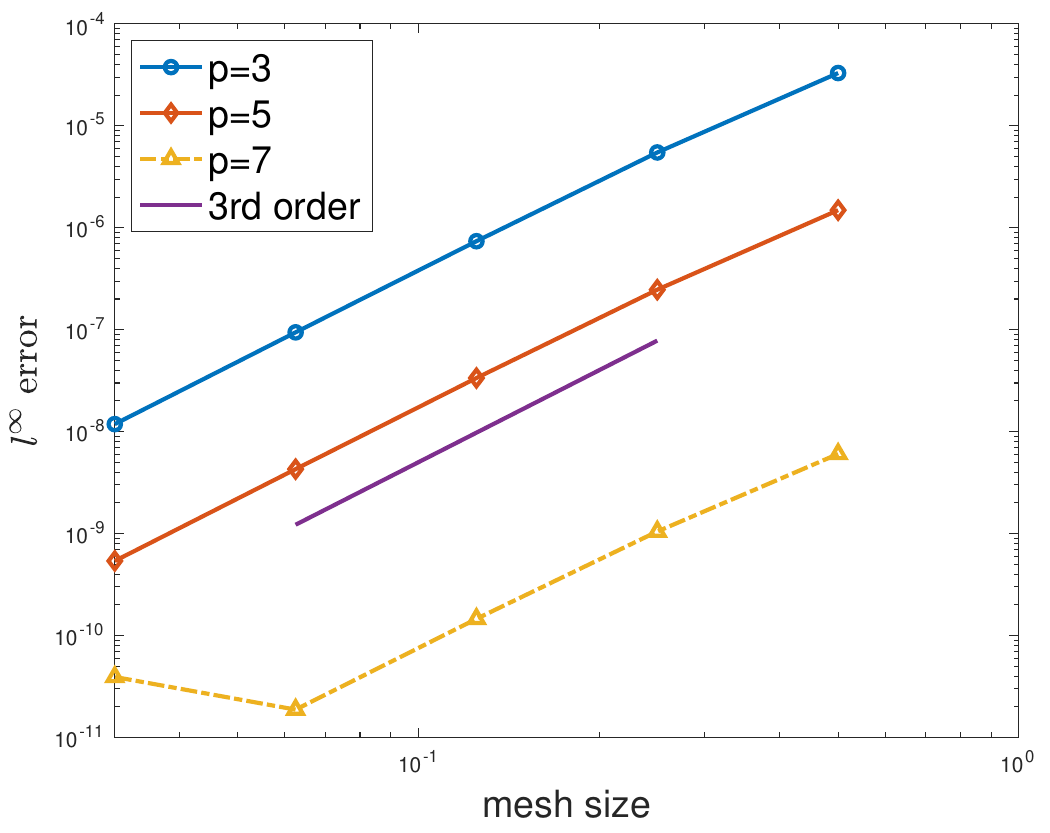}}
\subfigure[p-refinement]{
\includegraphics[height=0.3\textwidth]{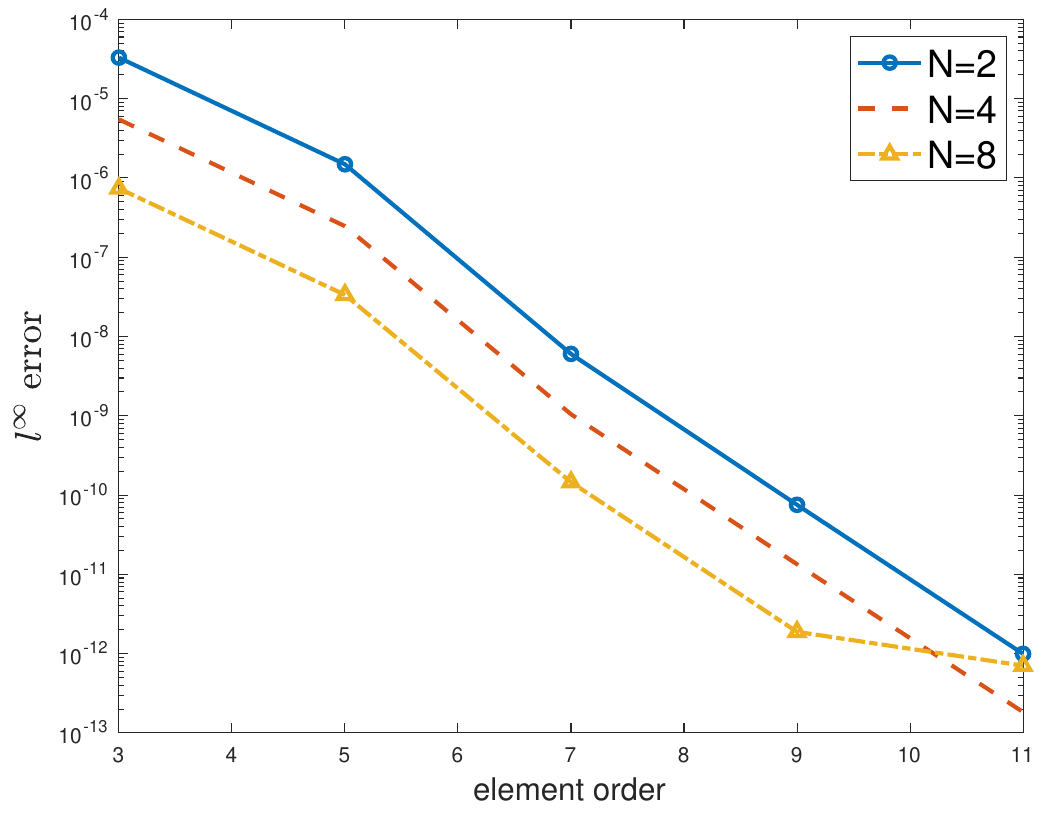}}
\caption{1D Poisson equation (non-polynomial bases): $l^{\infty}$ errors of FCE-$C^1$ versus (a) the element size, and (b) the element order $p$. 
$q=p+3$ Gauss-Lobatto-Legendre collocation points. 
A reference line for third-order convergence is shown in (a).
} 
\label{fg_6}
\end{figure}

\begin{table}[tb]
\centering
\begin{tabular}{llll}
\hline
$p$ & FCE-$C^1$ & FCE-$C^0$ & FCE-NC \\ \hline
3 & 2.91E-6 &3.25E-4 & 1.16E+0 \\
5 & 1.32E-7 &1.52E-5 & 7.44E-1 \\
7 & 5.63E-10 &8.25E-7 & 6.83E-2 \\
9 & 7.25E-12 &2.82E-8 & 2.53E-4 \\
11 & 9.93E-14 &5.96E-10 & 6.46E-7 \\
13 & 6.44E-14 &1.86E-11 & 1.23E-9 \\
15 & 1.55E-15 &4.83E-13 & 1.89E-11 \\
\hline
\end{tabular}
\caption{1D Poisson eqaution (non-polynomial bases): $l^{\infty}$ errors of FCE-$C^1$, FCE-$C^0$ and FCE-NC corresponding to a range of element orders $p$. $N=5$ uniform elements, and $q=p+3$ Gaus--Lobatto-Legendre collocation points in each element.
%{\color{red}[Need FCE-$C^0$ results for this table.]}
}
\label{tab_7}
\end{table}

We consider the domain $\Omega=[0,1]$ and the following BVP on $\Omega$,
\begin{subequations}\label{eq_149}
\begin{align}
& \frac{d^2u}{dx^2} = f(x), \quad x\in\Omega, \\
& u(0) = g_1, \quad u(1) = g_2,
\end{align}
\end{subequations}
where $u(x)$ is to be computed, 
$f(x)$ is a source term, and $g_1$ and $g_2$ are the boundary data. 
The source term and the boundary data are chosen such that this problem has the exact solution $u_{ex}(x)=\tanh (x)$.

We partition $\Omega$ into $N$ uniform elements and impose $C^1$ continuity conditions across the element boundaries. We employ 1D $C^1$ and $C^0$ FCEs and FCE-NCs to solve this problem. The free functions are expanded in terms of the sinusoidal bases given in~\eqref{eq_148}. We employ $q=p+3$ Gauss-Lobatto-Legendre quadrature points as the collocation points within each element, where $p$ is the element order.
When solving this problem using FCE-NC (FCEs with no continuity) with the sinusoidal bases, we find that it is necessary to scale those equations corresponding to the boundary conditions and the $C^0$ and $C^1$ continuity conditions by a constant factor and use the re-scaled equations. In other words, we scale the equations~\eqref{eq_132b}--\eqref{eq_132e} (in Remark~\ref{rem_6}) as follows,
\begin{subequations}\label{eq_150}
\begin{align}
& \sigma_0 u_i(X_{i+1}) - \sigma_0u_{i+1}(X_{i+1})=0, \quad 0\leqslant i\leqslant N-2; \label{eq_150b} \\
& \sigma_1\left. \frac{d u_i}{\partial x} \right|_{X_{i+1}} - \sigma_1\left. \frac{d u_{i+1}}{\partial x} \right|_{X_{i+1}}=0, \quad 0\leqslant i\leqslant N-2; \label{eq_150c} \\
& \sigma u_0(a) - \sigma C_a=0; \label{eq_150d} \\
& \sigma u_{N-1}(b) - \sigma C_b=0, \label{eq_150e}
\end{align}
\end{subequations}
where $\sigma$, $\sigma_0$ and $\sigma_1$ are prescribed positive constants. Note that while the re-scaled equations are mathematically equivalent to the original equations, the scaling will influence the least squares solution to the system.
We have employed $\sigma=\sigma_0 = \frac{1}{h^4}$ and $\sigma_1=\frac{1}{h^2}$ with FCE-NC for solving~\eqref{eq_149}, where $h=\frac{1}{N}$ denotes the element size.
Without the scaling (i.e.~$\sigma=\sigma_0=\sigma_1=1$), the FCE-NC results obtained with the sinusoidal bases would be markedly poorer (by around two orders of magnitude) for this problem. No scaling is performed when using 1D FCE-$C^0$ and FCE-$C^1$ for solving the problem.

% simulation results

Figure~\ref{fg_6} shows the $l^{\infty}$ errors obtained by FCE-$C^1$ as a function of the element size (plot (a)) and as a function of the element order (plot (b)). As the element size decreases, we observe a third-order convergence rate with FCE-$C^1$ (for different element order $p$). This behavior is a little different from that of the method when  Legendre polynomials are used to   represent the free functions in FCE. For a given mesh size, as the element order $p$ increases, the error decreases exponentially (Figure~\ref{fg_6}(b)). It is notable that the FCE method achieves an exponential convergence  with increasing number of sinusoidal basis functions within each element, while with Dirichlet BCs  imposed on the  domain.

Table~\ref{tab_7} shows a comparison of the $l^{\infty}$ errors corresponding to a range of element orders obtained with FCE-$C^1$, FCE-$C^0$ and FCE-NC for solving~\eqref{eq_149}. Five uniform elements are employed in these tests ($N=5$). The errors are observed to decrease exponentially with increasing element order $p$ for all three types of FCEs. The accuracy of the three  FCEs is quite different under the same or comparable element order. The FCE-$C^1$ results are significantly more accurate than those of FCE-$C^0$, which in turn are significantly more accurate than those of FCE-NC.

\begin{figure}[tb]
\centering 
\subfigure[h-refinement]{
\includegraphics[height=0.3\textwidth]{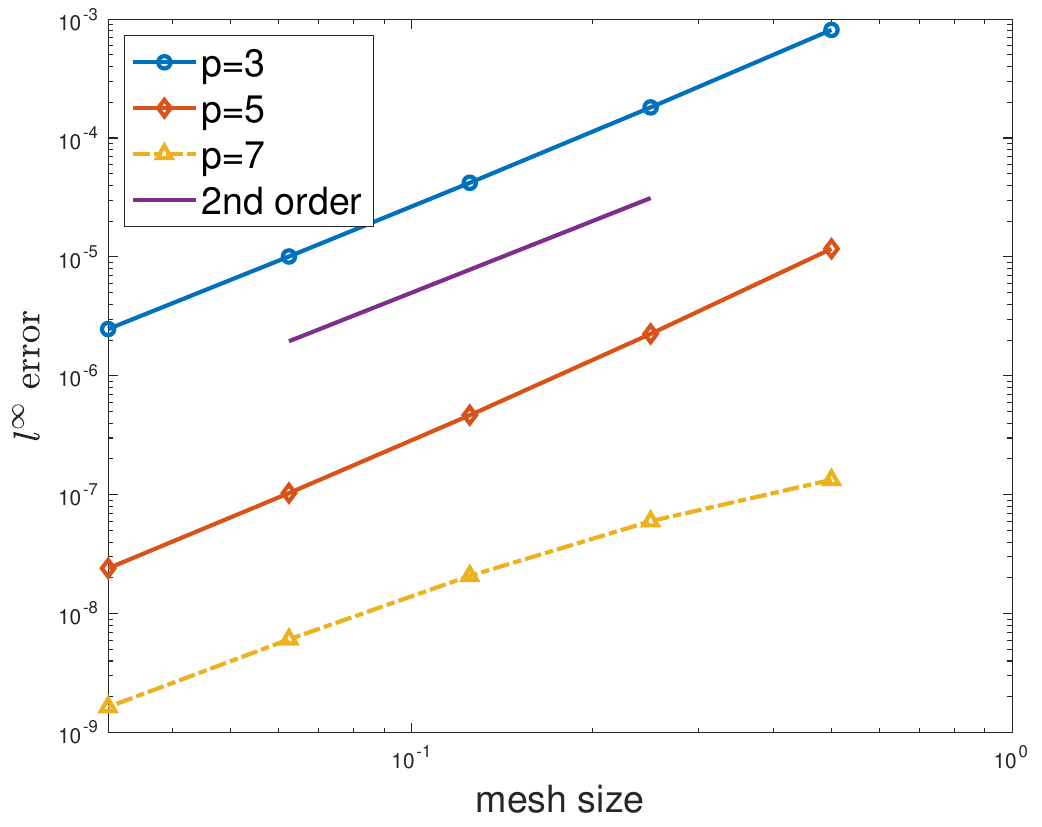}}
\subfigure[p-refinement]{
\includegraphics[height=0.3\textwidth]{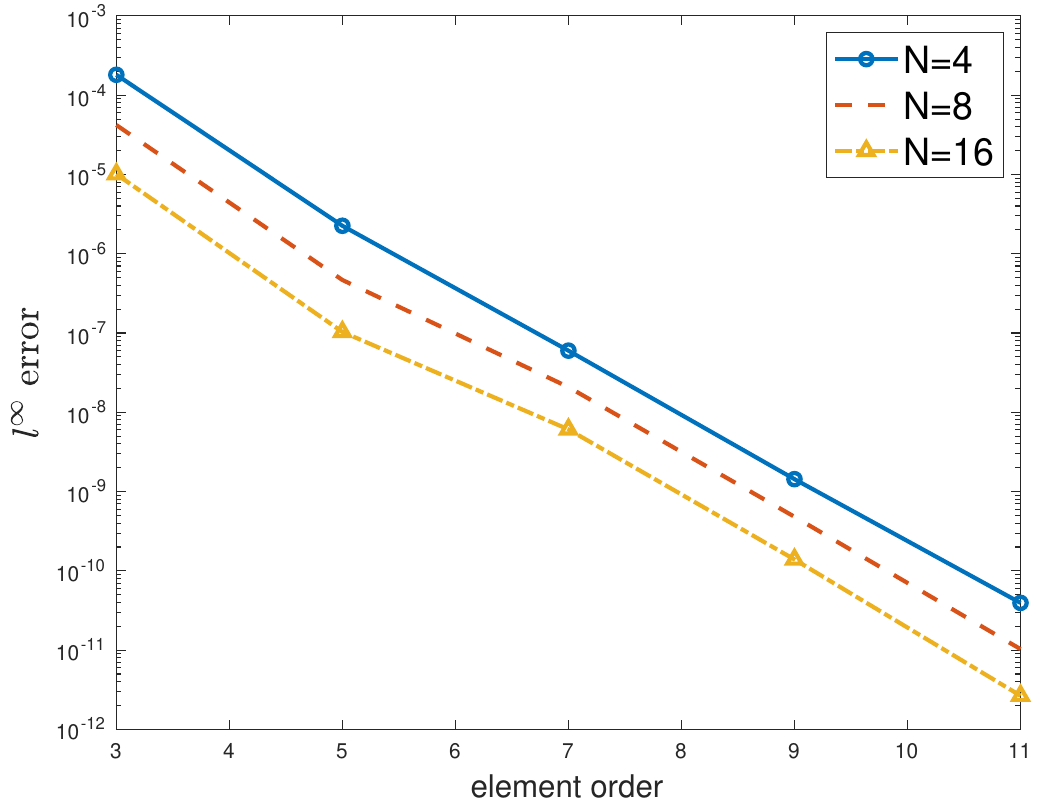}}
\caption{Initial value problem (non-polynomial bases): $l^{\infty}$ errors of FCE-$C^0$ as a function of (a) the element size, and (b) the element order. 
%$q=p+3$ Gauss-Lobatto-Legendre collocation points in each element.
} 
\label{fg_7}
\end{figure}

\begin{table}[tb]
\centering
\begin{tabular}{lll} \hline 
  $p$  & FCE-$C^0$ & FCE-NC\\ \hline
 $3$ & 1.12E-4  & 8.71E-1\\ 
 $5$ &	1.35E-6  & 5.22E-3 \\ 
 $7$ & 4.38E-8  & 2.85E-4 \\ 
 $9$ & 1.03E-9  & 6.76E-6\\ 
 $11$ & 2.56E-11 & 1.66E-7\\ 
 \hline
\end{tabular}
\caption{Initial value problem (non-polynomial bases): $l^{\infty}$ errors of FCE-$C^0$ and FCE-NC corresponding to a range of element orders $p$. $N=5$ uniform elements in mesh.
}
\label{tab_8}
\end{table}

\subsubsection{An Initial Value Problem}

We consider the domain $\Omega=[0,1]$, and test the FCE method with sinusoidal bases using the following IVP,
\begin{subequations}\label{eq_151}
\begin{align}
& \frac{du}{dt} + u = f(t), \quad t\in\Omega, \label{eq_151a} \\
& u(0) = u_0, \label{eq_151b}
\end{align}
\end{subequations}
where $u(t)$ is to be computed, $f(t)$ is a prescribed source term, and $u_0$ denotes the initial condition. We choose $f$ and $u_0$ such that this problem has the exact solution $u_{ex}(t)=\tanh(t^2)$. 

To solve~\eqref{eq_151} with FCE we treat it as a ``boundary" value problem, with the condition~\eqref{eq_151b} imposed  at $t=0$. We partition $\Omega$ into $N$ uniform elements, and impose $C^0$ continuity on $u(t)$ across the element boundaries. We employ 1D $C^0$ FCEs and FCE-NCs to solve the resultant system, with the free functions represented by the sinusoidal bases given in~\eqref{eq_148}. With FCE-$C^0$, the $C^0$ continuity  and the initial condition~\eqref{eq_151b} are enforced exactly.  
With FCE-NC, we scale the initial condition~\eqref{eq_151b} by a  factor $\sigma>0$, and the $C^0$ continuity conditions by a  factor $\sigma_0>0$ in the least squares collocation formulation, in a way analogous to equations~\eqref{eq_150b} and~\eqref{eq_150d}. We employ $\sigma=\sigma_0=\frac{1}{h^2}$ in  FCE-NC, where $h=\frac{1}{N}$ denotes the element size. This scaling is important for the FCE-NC accuracy. Without this scaling (i.e.~$\sigma=\sigma_0=1$) the FCE-NC results would be significantly poorer. Note that no scaling is performed with FCE-$C^0$.
With each element we employ $q=p+3$ Gauss-Lobatto-Legendre collocation points.

Figure~\ref{fg_7} shows the $l^{\infty}$ errors of FCE-$C^0$ as a function of the element size (plot (a)) for several element orders, and as a function of the element order $p$ for several element sizes. With sinusoidal bases we observe that the convergence rate of FCE-$C^0$ is 2nd-order with respect to the element size (for fixed element orders). This behavior is again different from that when the free functions are represented by polynomials. 
For a given mesh size, we observe the  exponential convergence rate as the element order increases (plot (b)) with the sinusoidal bases. 

Table~\ref{tab_8} is a comparison of the $l^{\infty}$ errors obtained with FCE-$C^0$ and FCE-NC for solving~\eqref{eq_151}. While the errors of both FCE-$C^0$ and FCE-NC decrease exponentially with increasing element order, the FCE-$C^0$ results are considerably more accurate than those of FCE-NC under the same or comparable element orders.

\subsection{Numerical Tests with Relative Boundary Conditions}

We next test the FCE method using  boundary value problems involving a class of non-traditional boundary conditions. Specifically, these boundary conditions involve relative constraints, which can be linear or nonlinear, of the solution field or its derivatives on the domain boundary and also possibly on the domain interior. We refer to such conditions as relative boundary conditions. The FCE method can handle this type of problems  straightforwardly and enforce the BCs exactly. These problems, however, would be much more challenging to traditional element-based techniques (e.g.~finite element or spectral element methods). In the following simulations we employ Legendre polynomials for representing the free functions in FCE.

\subsubsection{1D Example with Linear Relative Boundary Conditions}
\label{sec_431}

\begin{figure}[tb]
\centering 
\subfigure[h-refinement]{
\includegraphics[width=0.4\textwidth]{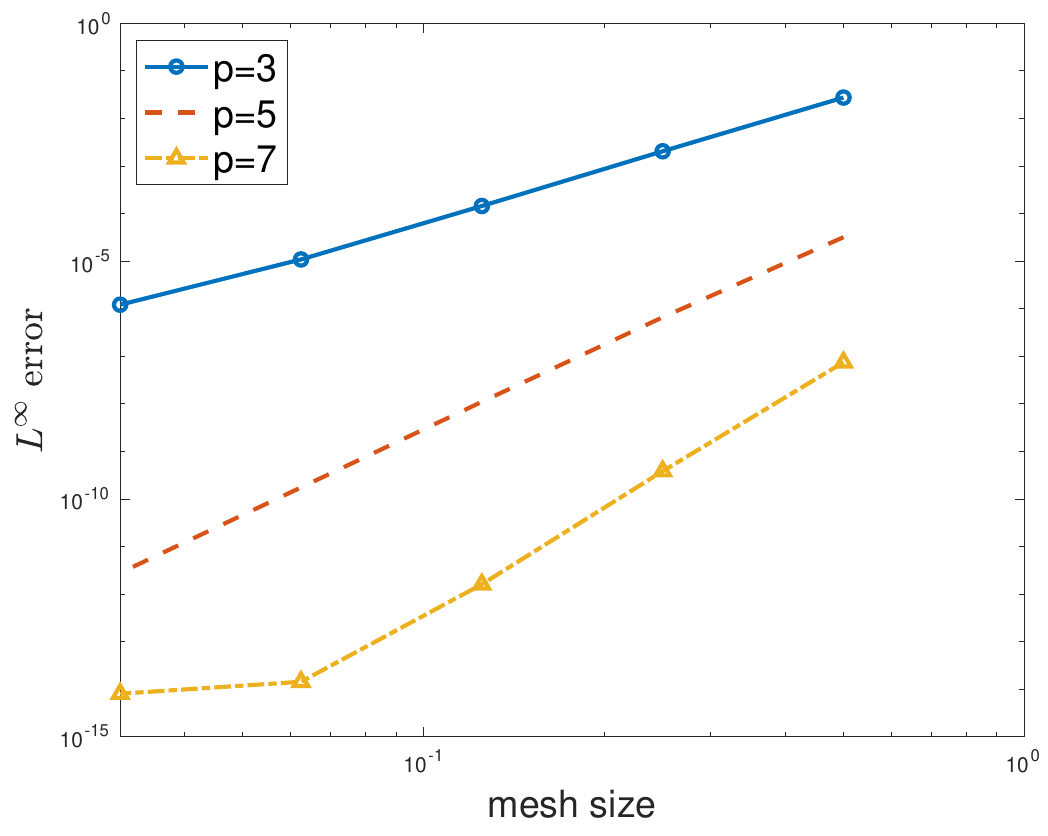}}
\subfigure[p-refinement]{
\includegraphics[width=0.4\textwidth]{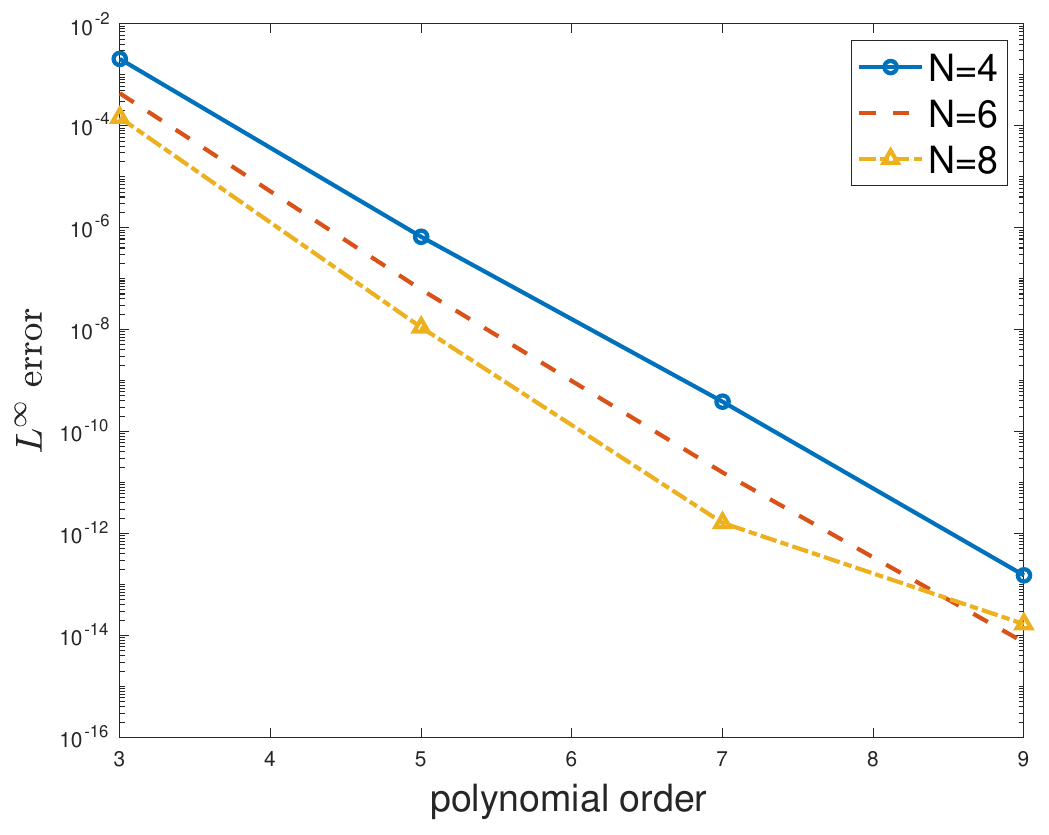}}
\caption{1D linear relative BC: $l^{\infty}$ errors of FCE-$C^1$ versus (a) the element size, and (b) the polynomial order. $q=p+2$ Gauss-Lobatto-Legendre collocation points in each element, where $p$ is the polynomial order. $N_1=N_2=\frac{N}{2}$, where $N$ denotes the total number of elements.
} 
\label{fg_8}
\end{figure}

Consider the domain $\Omega=[0,1]$ and the following BVP on $\Omega$,
\begin{subequations}\label{eq_152}
\begin{align}
& \frac{d^2u}{dx^2} - u = -(1+\pi^2)\cos(\pi x), \quad x\in\Omega, \label{eq_152a} \\
& u(0) = u(0.5) + 1, \label{eq_152b} \\
& \left.\frac{du}{dx}\right|_{x=1} = \left.\frac{du}{dx}\right|_{x=0.5} + \pi, \label{eq_152c}
\end{align}
\end{subequations}
where $u(x)$ is to be computed. Note that the BCs~\eqref{eq_152b} and~\eqref{eq_152c} represent constraint relations between $u(0)$ and $u(0.5)$ and between $u'(1)$ and $u'(0.5)$. This problem has an exact solution $u_{ex}(x)=\cos(\pi x)$.

We partition the domain $\Omega_1=[0,0.5]$ into $N_1$ uniform elements and the domain $\Omega_2=[0.5,1]$ into $N_2$ uniform elements, leading to a total of $N=N_1+N_2$ elements on $\Omega$. Let $X_i$ ($0\leqslant i\leqslant N$) denote the boundary points of the elements, where $X_0=0$, $X_{N_1}=0.5$ and $X_{N}=1$. We impose $C^1$ continuity  across the element boundaries, since~\eqref{eq_152a} is of second order.
The problem~\eqref{eq_152} is then reduced to  (see Remark~\ref{rem_6}),
\begin{subequations}\label{eq_153}
\begin{align}
& \frac{d^2u_i}{dx^2} - u_i = -(1+\pi^2)\cos(\pi x), \quad x\in[X_i,X_{i+1}], \quad 0\leqslant i\leqslant N-1; \label{eq_153a} \\
& u_i(X_{i+1}) = u_{i+1}(X_{i+1}), \quad 0\leqslant i\leqslant N-2; \label{eq_153b} \\
& u_i'(X_{i+1}) = u'_{i+1}(X_{i+1}), \quad 0\leqslant i\leqslant N-2; \label{eq_153c} \\
& u_0(X_0) = u_{N_1}(X_{N_1}) + 1, \label{eq_153d} \\
& u'_{N-1}(X_N) = u'_{N_1}(X_{N_1}) + \pi. \label{eq_153e}
\end{align}
\end{subequations}
Here $u_i(x)$ ($0\leqslant i\leqslant N-1$) is the restriction of $u(x)$ to the element $[X_i,X_{i+1}]$. 

We employ 1D $C^1$ FCEs to solve the system~\eqref{eq_153}. The equations~\eqref{eq_153b} and~\eqref{eq_153c} are automatically satisfied thanks to the $C^1$ FCE formulation. Equations~\eqref{eq_153d} and~\eqref{eq_153e} lead to the following relations,
\begin{equation}\label{eq_154}
\alpha_0 = \alpha_{N_1} + 1, \quad
\beta_{N} = \beta_{N_1} + \pi,
\end{equation}
where $\alpha_i$ and $\beta_i$ ($0\leqslant i\leqslant N$) are the parameters in the 1D FCE-$C^1$ formulation~\eqref{eq_68}. In light of these relations, only the parameters $\alpha_i$ ($1\leqslant i\leqslant N$) and $\beta_i$ ($0\leqslant i\leqslant N-1$) are independent and need to be computed. Therefore, equation~\eqref{eq_153a} is the only equation to be solved with $C^1$ FCE. By enforcing this equation on a set of collocation points within each element, one can solve the resultant algebraic system by the linear least squares method for the parameters $\alpha_i$ ($1\leqslant i\leqslant N$), $\beta_i$ ($0\leqslant N-1$), and $\hat g_{ij}$ involved in~\eqref{eq_68}.

Figure~\ref{fg_8} illustrates the convergence behavior of  $C^1$ FCE  in the h-refinement and the p-refinement for solving this problem. Here we employ $N_1=N_2=\frac{N}{2}$ elements on both $[0,0.5]$ and $[0.5,1]$. Legendre polynomials of order $p$ are employed to represent the free functions in FCE-$C^1$, with $q=p+2$ Gauss-Lobatto-Legendre collocation points in each element. Figure~\ref{fg_8}(a) shows the $l^{\infty}$ errors as a function of the element size, where the polynomial order is fixed. Figure~\ref{fg_8}(b) shows the $l^{\infty}$ errors as a function of the polynomial order for several given mesh sizes. The current method has captured the solution accurately in the presence of relative boundary conditions.

\subsubsection{1D Example with Nonlinear Relative Boundary Conditions}

\begin{figure}[tb]
\centering 
\includegraphics[width=0.4\textwidth]{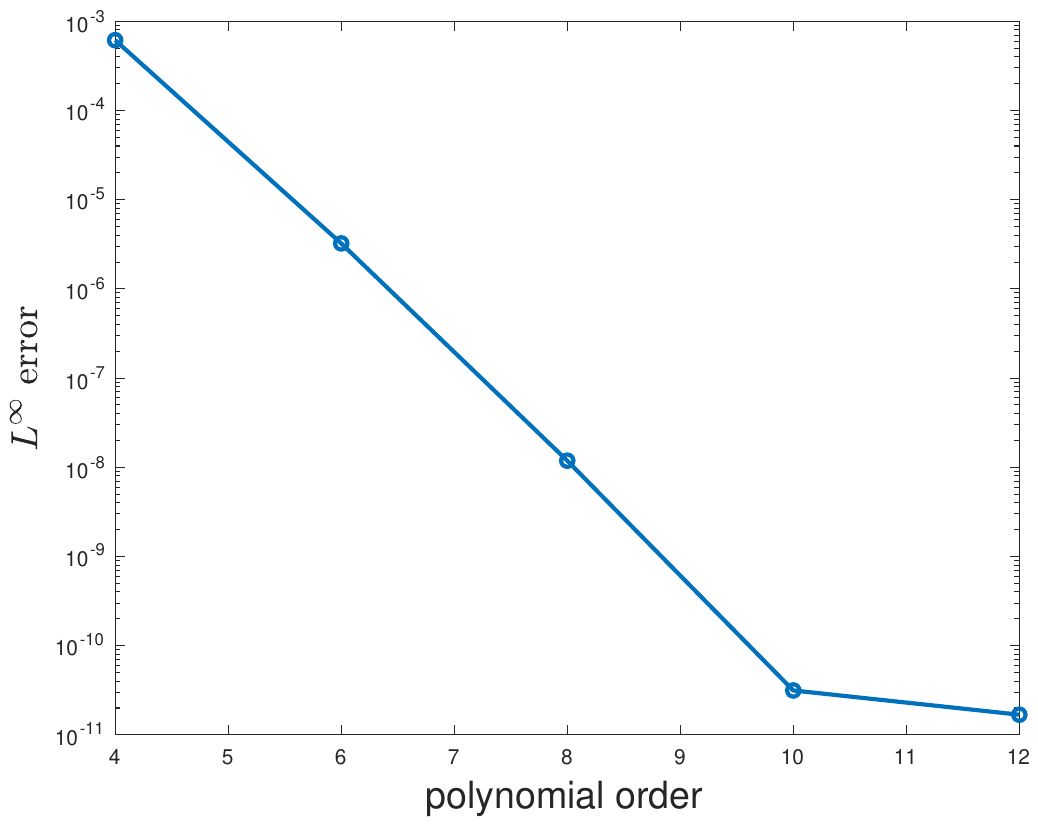}
\caption{1D nonlinear relative BC: $l^{\infty}$ error of FCE-$C^1$ versus the polynomial order. $N=2$ uniform elements, and $q=20$ uniform collocation points in each element.
} 
\label{fg_9}
\end{figure}

Consider the domain $\Omega=[0,1]$ and the following BVP on $\Omega$,
\begin{subequations}\label{eq_155}
\begin{align}
& \frac{d^2u}{dx^2} - u = -\frac12(1+\pi^2)\sin(\pi x) - 1, \quad x\in\Omega, \label{eq_155a} \\
& \left[u(0)\right]^3 = u(1), \label{eq_155b} \\
& \left.\frac{du}{dx}\right|_{x=0} = 2\left.\frac{du}{dx}\right|_{x=0.5} + \frac{\pi}{2}, \label{eq_155c}
\end{align}
\end{subequations}
where $u(x)$ is to be solved for. Note that this problem involves relative BCs, and that the condition~\eqref{eq_155b} is nonlinear. This problem has an exact solution $u_{ex}(x) = 1+\frac12\sin(\pi x)$.

The procedure for solving the problem~\eqref{eq_155} follows that of Section~\ref{sec_431}. We partition each of the domains $[0,0.5]$ and $[0.5,1]$ into $N_1$ uniform elements, with a total of $N=2N_1$ uniform elements on $\Omega$. We impose $C^1$ continuity  across the element boundaries, and  employ  $C^1$ FCEs to solve this problem. The boundary conditions~\eqref{eq_155b}-\eqref{eq_155c} lead to the following relations
\begin{equation}\label{eq_156}
\alpha_N = \alpha_0^3, \quad
\beta_0 = 2\beta_{N_1} + \frac{\pi}{2},
\end{equation}
where $\alpha_i$ and $\beta_i$ ($0\leqslant i\leqslant N$) are the parameters in the FCE-$C^1$ formulation~\eqref{eq_68}. In light of these relations, only the parameters $\alpha_i$ ($0\leqslant i\leqslant N-1$) and $\beta_i$ ($1\leqslant i\leqslant N$) are independent and to be solved for. After enforcing the equation~\eqref{eq_155a} on a set of collocation points within each element, we attain a system of algebraic equations, which is nonlinear due to~\eqref{eq_156}, about the parameters $\alpha_i$ ($0\leqslant i\leqslant N-1$), $\beta_i$ ($1\leqslant i\leqslant N$), and $\hat g_{ij}$ in~\eqref{eq_68}. 
This system is solved by the nonlinear least squares method (see Remark~\ref{rem_4}).

Figure~\ref{fg_9} shows the $l^{\infty}$ errors of FCE-$C^1$ as a function of the polynomial order $p$ for solving this problem. We have employed $N=2$ uniform elements, with $q=20$ uniform collocation points within each element. The figure clearly shows an exponential convergence rate of the current method relative to the polynomial order (before saturation beyond polynomial order $10$).

\subsubsection{2D Example with Relative Boundary Conditions}

We consider the 2D domain $\Omega=[0,1]\times[0,1]$ and the following BVP with the Poisson equation on $\Omega$,
\begin{subequations}\label{eq_157}
\begin{align}
& \frac{\partial^2u}{\partial x^2} + \frac{\partial^2u}{\partial y^2} = f(x,y), \quad (x,y)\in\Omega; \label{eq_157a} \\
& u(0,y) = C(0,y), \quad 
u(1,y) = u(0.5,y) + g(y), \quad y\in[0,1]; \label{eq_157b} \\
& u(x,0) = C(x,0), \quad u(x,1) = C(x,1), \quad x\in[0,1]. \label{eq_157c}
\end{align}
\end{subequations}
Here $u(x,y)$ is the field function to be computed. $C(x,y)$ is a prescribed function on the domain boundaries $x=0$, $y=0$, and $y=1$. $g(y)$ is another prescribed function that is compatible with $C(x,y)$, i.e.
\begin{equation}
g(0) = C(1,0) - C(0.5,0), \quad
g(1) = C(1,1) - C(0.5,1).
\end{equation}
Note that a relative BC is involved in~\eqref{eq_157b}, which imposes a constraint on $u(1,y)$ and $u(0.5,y)$. We choose the source term $f$ and the boundary data $C(x,y)$ and $g(y)$ such that the system~\eqref{eq_157} has the exact solution $u_{ex}(x,y) = \sin(\pi x)\cos(\pi y)$.

We partition each of the domains $\Omega_1=[0,0.5]\times[0,1]$ and $\Omega_2=[0.5,1]\times[0,1]$ into $N_{1}\times N_y$ uniform elements, leading to $N_x\times N_y$ elements in $\Omega=\Omega_1\cup\Omega_2$, with $N_x=2N_1$ elements along the $x$ direction and $N_y$ elements along the $y$ direction. We impose $C^1$ continuity  across the element boundaries in both $x$ and $y$ directions.

We employ 2D $C^0$ FCEs to solve this problem under two settings: (i) the relative BC in~\eqref{eq_157b} is enforced exactly with  FCE-$C^0$, and (ii) the relative BC in~\eqref{eq_157b} is enforced only approximately (in the least squares sense) with FCE-$C^0$.

To enforce the relative BC in~\eqref{eq_157b} exactly with $C^0$ FCEs, we modify the least squares collocation formulation in Section~\ref{sec_3} as follows. We use the same notations as in Sections~\ref{sec_2} and~\ref{sec_3}. 
Let $X_i$ ($0\leqslant i\leqslant N_x$) and $Y_j$ ($0\leqslant j\leqslant N_y$) denote the coordinates of the element boundaries along $x$ and $y$, respectively, with $X_0=Y_0=0$, $X_{N_1}=0.5$ and $X_{N_x}=Y_{N_y}=1$.
The boundary conditions~\eqref{eq_157b}--\eqref{eq_157c} then lead to the following relations,
\begin{subequations}\label{eq_159}
\begin{align}
&  \hat\alpha_{0j} = C(0,Y_j), \quad 0\leqslant j\leqslant N_y; \label{eq_159a0} \\ 
&\hat\alpha_{N_x,j} = \hat\alpha_{N_1,j} + g(Y_j), \quad 0\leqslant j\leqslant N_y; \label{eq_159a} \\
& \hat\alpha_{i0} = C(X_i,0), \quad \hat\alpha_{i,N_y} = C(X_i,1), \quad 0\leqslant i\leqslant N_x; \label{eq_159b} \\
& G_{0j}(y) = C(0,y), \quad y\in[Y_j,Y_{j+1}], \quad 0\leqslant j\leqslant N_y-1; \label{eq_159c0} \\
& G_{N_x,j}(y)=  G_{N_1,j}(y) + g(y), 
\quad y\in[Y_j,Y_{j+1}], \quad 0\leqslant j\leqslant N_y-1; \label{eq_159c} \\
& H_{i0}(x) = C(x,0), \quad H_{i,N_y}(x)= C(x,1), \quad x\in[X_i,X_{i+1}], \quad 0\leqslant i\leqslant N_x-1. \label{eq_159d}
\end{align}
\end{subequations}
Therefore, in the 2D $C^0$ FCE formulation~\eqref{eq_79}, the functions $G_{ij}(y)$ are given by~\eqref{eq_159c0}--\eqref{eq_159c}, and by~\eqref{eq_76} for $(1,0)\leqslant(i,j)\leqslant(N_x-1,N_y-1)$.  The functions $H_{ij}(x)$ are given by~\eqref{eq_159d}, and by~\eqref{eq_77} for $(0,1)\leqslant(i,j)\leqslant(N_x-1,N_y-1)$. In the parametric form~\eqref{eq_79}, the unknown parameters are given by,
\begin{equation}\label{eq_160}
\begin{split}
\bm\Theta = \{\ &\hat g_{e_{ij},k},\ \text{for}\ (0,0,1)\leqslant(i,j,k)\leqslant(N_x-1,N_y-1,\mathcal M);  \\
& \mathscr{\hat G}_{ij,k},\ \text{for}\ (1,0,1) \leqslant(i,j,k)\leqslant(N_x-1,N_y-1,m); \\
& \mathscr{\hat H}_{ij,k}, \ \text{for}\ 
(0,1,1)\leqslant(i,j,k)\leqslant(N_x-1,N_y-1,m); \\
& \hat\alpha_{ij}, \ \text{for}\ 
1\leqslant(i,j)\leqslant(N_x-1,N_y-1)
 \ \}.
\end{split}
\end{equation}
With the above modification to the FCE-$C^0$ formulation, the only equations that need to be solved are,
\begin{subequations}\label{eq_16a}
\begin{align}
& \frac{\partial^2u_{e_{ij}}}{\partial x^2} + \frac{\partial^2u_{e_{ij}}}{\partial y^2} = f(x,y), \quad (x,y)\in[X_i,X_{i+1}]\times[Y_j,Y_{j+1}], \quad 0\leqslant (i,j)\leqslant (N_x-1,N_y-1); \label{eq_161a} \\
& \left.\frac{\partial u_{e_{ij}}}{\partial x}\right|_{(X_{i+1},y)} = \left.\frac{\partial u_{e_{i+1,j}}}{\partial x}\right|_{(X_{i+1},y)}, \quad y\in[Y_j,Y_{j+1}], \quad 0\leqslant (i,j)\leqslant(N_x-2,N_y-1); \label{eq_161b} \\
& \left.\frac{\partial u_{e_{ij}}}{\partial y}\right|_{(x,Y_{j+1})} = \left.\frac{\partial u_{e_{i,j+1}}}{\partial y}\right|_{(x,Y_{j+1})}, \quad x\in[X_i,X_{i+1}], \quad 0\leqslant(i,j)\leqslant(N_x-1,N_y-2). \label{eq_161c}
\end{align}
\end{subequations}
By enforcing these equations on the collocation points, we can solve the resultant linear algebraic system for the unknown parameters $\bm\Theta$ defined in~\eqref{eq_160} by the linear least squares method.

% how to enforce relative BC approximately

To enforce the relative BC in~\eqref{eq_157b} approximately using $C^0$ FCEs, we note that the rest of the boundary conditions lead to the relations~\eqref{eq_159a0},~\eqref{eq_159b},~\eqref{eq_159c0}, and~\eqref{eq_159d}. Therefore, in the parametric form~\eqref{eq_79}, the unknown parameters consist of,
\begin{equation}\label{eq_162}
\begin{split}
\bm\Theta = \{\ &\hat g_{e_{ij},k},\ \text{for}\ (0,0,1)\leqslant(i,j,k)\leqslant(N_x-1,N_y-1,\mathcal M);  \\
& \mathscr{\hat G}_{ij,k},\ \text{for}\ (1,0,1) \leqslant(i,j,k)\leqslant(N_x,N_y-1,m); \\
& \mathscr{\hat H}_{ij,k}, \ \text{for}\ 
(0,1,1)\leqslant(i,j,k)\leqslant(N_x-1,N_y-1,m); \\
& \hat\alpha_{ij}, \ \text{for}\ 
1\leqslant(i,j)\leqslant(N_x,N_y-1)
 \ \}.
\end{split}
\end{equation}
With the FCE-$C^0$ formulation, the equations that need to be solved are equations~\eqref{eq_161a}--\eqref{eq_161c}, and
\begin{align}
& u_{e_{N_x,j}}(X_{N_x},y) = u_{e_{N_1,j}}(X_{N_1},y) + g(y), \quad y\in[Y_j,Y_{j+1}], \quad 0\leqslant j\leqslant N_y-1.
\end{align}
These equations are enforced on the collocation points, and the resultant linear algebraic system is solved for $\bm\Theta$ by the linear least squares method.

\begin{figure}[tb]
\centering 
\includegraphics[width=0.4\textwidth]{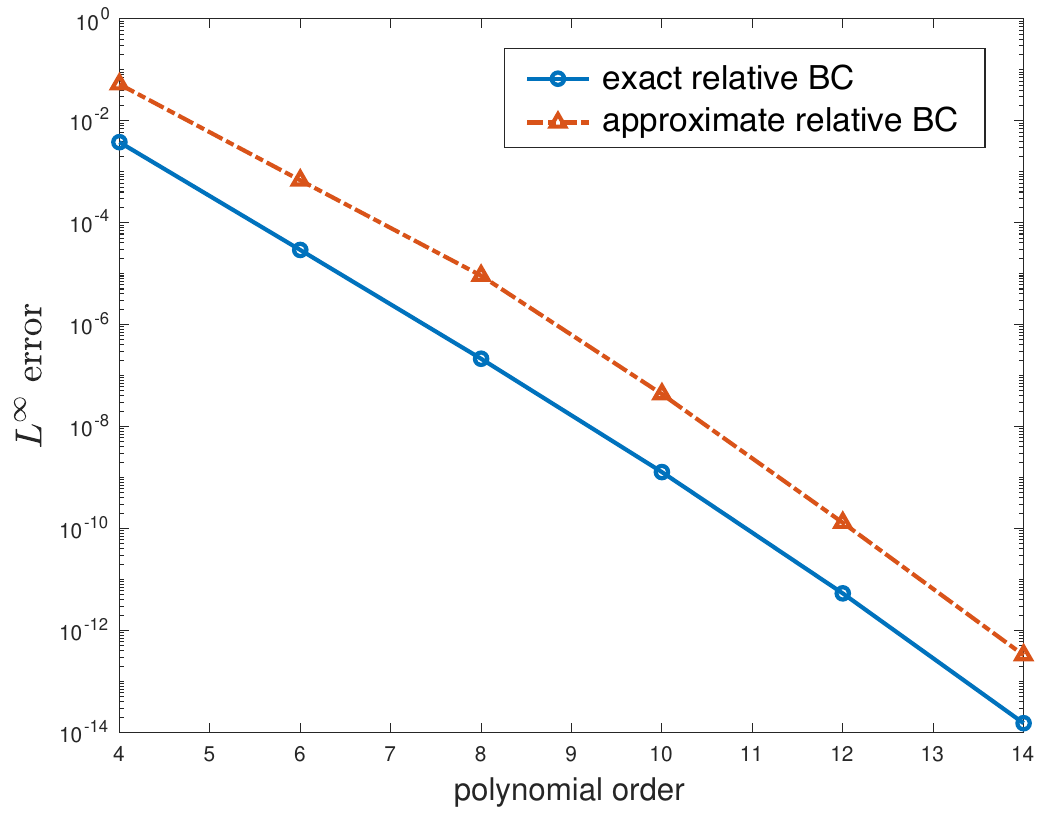}
\caption{2D Relative BC: $l^{\infty}$ errors of FCE-$C^0$ versus the polynomial order $p$, with the relative BC enforced exactly (blue line, circles) and enforced approximately (red line, triangles). $(N_x,N_y)=(2,1)$ uniform elements, with $q=p+2$ Gauss-Lobatto-Legendre collocation points in each direction.
} 
\label{fg_10}
\end{figure}

Figure~\ref{fg_10} shows the $l^{\infty}$ errors of FCE-$C^0$ as a function of the polynomial order, obtained with the relative BC enforced exactly and approximately. We have employed two uniform elements in the domain (partition along the $x$ direction), and $q=p+2$ Gauss-Lobatto-Legendre collocation points in both directions within each element. The errors decrease exponentially with increasing polynomial order. The simulation result is significantly more accurate (by an order of magnitude) when the relative BC is enforced exactly in  FCE-$C^0$.

\section{Concluding Remarks}
\label{sec_summary}

% what have you done in this paper?
% what are the results?
%  (i) construction of C^0 and C^1 FCEs, most general forms
%  (ii) least squares collocation approach for solving PDEs/ODEs
%  (iii) non-polynomial bases
%  (iv) relative BCs
% what are the implications of these results?
% what remains to be done?

In this paper we have presented a method for solving boundary and initial value problems based on functionally connected elements (FCE). Constructions of $C^0$ and $C^1$ FCEs in one and two dimensions are developed based on a strategy stemming from the theory of functional connections (TFC). The FCE formulation provides the {\em general form} of piece-wise functions having an intrinsic $C^0$ or $C^1$ continuity across the element boundaries, 
%We refer to these general functions with an intrinsic $C^0$ (or $C^1$) continuity as $C^0$ (or $C^1$) FCEs, respectively. 
which contains certain free functions or free parameters defined on the elements or associated with the element boundaries. The FCE construction ensures the exact satisfaction of $C^0$ or $C^1$ continuity across the elements, irrespective of the form of free functions or the value of free parameters. 
The free functions are represented by polynomial or non-polynomial bases in the current work.
%This property enables the use of non-conventional bases to represent the free functions with FCE, apart from the traditional polynomial bases. This is a novel capability that seems to be  lacking with traditional element-based methods. 

% least squares approach with FCE for PDE

We adopt FCEs together with  a least squares collocation approach for solving boundary value problems. The basic procedure can be summarized as follows. By partitioning the overall domain into sub-domains, 
we enforce the PDE/ODE on the sub-domains, the boundary condition on those sub-domain boundaries corresponding to the  domain boundary, and the $C^0$ or $C^1$ continuity  on the shared sub-domain boundaries. We use $C^0$ or $C^1$ FCEs to represent the solution field. Consequently, those $C^0$ or $C^1$ continuity conditions in the partitioned problem are automatically and exactly satisfied. The boundary condition can also be exactly satisfied by setting the FCE parameters appropriately. As a result, only the PDE/ODE needs to be enforced on the sub-domains. 
For this purpose, we choose a set of collocation points on each sub-domain, using either uniform grid points or Gauss-Lobatto quadrature points in the current work, and enforce the PDE/ODE on these points.
This process gives rise to a system of algebraic equations about the unknown parameters in the FCE parametric form. This system is linear for linear boundary value problems and becomes nonlinear if either the PDE or the boundary condition is nonlinear, and it generally does not have the same number of equations as the unknowns. We determine the unknown FCE parameters by solving this algebraic system via the linear or nonlinear least squares (Gauss-Newton) method~\cite{Bjorck1996}.

Variations of the above basic procedure have been investigated. We have looked into how to solve the partitioned boundary value problem using FCEs with a lower intrinsic continuity, for instance, using $C^0$ FCEs or FCE-NCs (FCEs with no continuity) for boundary value problems with  $C^1$ continuity imposed across sub-domains. In this case, a subset of the continuity conditions is satisfied automatically and exactly due to the FCE formulation, while the remaining continuity conditions are enforced via the least squares computation.  
Initial value problems and initial/boundary value problems can be solved by the FCE method based on the space-time approach, in which the time and space variables are treated on the same footing. Therefore, the problem effectively reduces to a ``boundary" value problem for the space-time domain.

% numerical results

We  test the FCE method  using a number of first-order and second-order initial/boundary value problems in one or two dimensions that are linear or nonlinear. Polynomial bases and non-polynomial bases (sinusoidal functions) for representing the FCE free functions have been studied numerically. The FCE method exhibits a spectral-like accuracy (exponential error convergence) as the number of bases increases within each element. 

The FCE method has a unique advantage over  traditional element-based methods for  boundary value problems with relative boundary conditions, in which the BCs involve linear or nonlinear relative constraints of the solution variables or their derivatives. The FCE method can enforce the relative BCs exactly and handle this type of problems quite straightforwardly. In contrast, these problems are considerably more challenging to traditional element-based techniques such as the finite element or spectral element methods.

% comment on weaknesses:
% construction of higher C^k FCEs
% weakness: only regular domains, regular domain decompositions
% unable to go to high dimensions

In this paper we have concentrated on FCE formulations with an intrinsic continuity up to $C^1$, and in up to two spatial dimensions. It should be noted that the FCE constructions presented herein are systematic and can be extended to FCEs having higher intrinsic continuity or to three and higher dimensions. For example, the general form of piecewise functions satisfying an intrinsic $C^k$ ($k\geqslant 2$) continuity across the element boundaries can be constructed analogously. 

The FCE construction presented here has a limitation. It is confined to regular partitions of the domain, where the element boundaries are required to be aligned with the coordinate lines (or planes). How to extend the construction to more general domain partitions is an interesting problem, and will be pursued in a future work.

% what else to discuss here?

\section*{Acknowledgement}

This work was partially supported by the US National Science Foundation (DMS-2012415).

\bibliographystyle{plain}
\bibliography{lsq,sem,varpro,elm1,mypub}

%%%%%%%%%%%%%%%%%

\end{document}